\documentclass[a4paper,oneside,reqno]{amsart}
\usepackage[utf8]{inputenc}

\usepackage{amssymb,amsmath}
\usepackage{thmtools}
\usepackage{thm-restate}
\usepackage{mathtools}
\usepackage{bbm}
\usepackage{enumitem}
\usepackage{hyperref}
\usepackage[capitalize]{cleveref}
\usepackage{stmaryrd} 

\usepackage{comment}
\usepackage{tikz}

\usepackage{geometry}
\declaretheorem[numberwithin=section]{theorem}
\declaretheorem[sibling=theorem]{lemma}
\declaretheorem[sibling=theorem]{fact}
\declaretheorem[sibling=theorem]{claim}
\crefname{claim}{Claim}{Claims}

\declaretheorem[sibling=theorem]{corollary}
\declaretheorem[sibling=theorem]{proposition}

\declaretheorem[sibling=theorem,style=definition]{definition}

\declaretheorem[name=Linear Program,style=definition,numbered=no]{LP}

\makeatletter
\renewcommand{\theHtheorem}{\thesection.\arabic{theorem}}
\newcommand{\FixSiblingH}[1]{%
  \expandafter\renewcommand\csname theH#1\endcsname{\theHtheorem}}
\makeatother

\FixSiblingH{lemma}
\FixSiblingH{fact}
\FixSiblingH{claim}
\FixSiblingH{conjecture}
\FixSiblingH{corollary}
\FixSiblingH{proposition}
\FixSiblingH{observation}
\FixSiblingH{definition}
\FixSiblingH{remark}

\newcommand{\br}[1]{\llbracket{#1}\rrbracket}

\renewcommand{\le}{\leqslant}
\renewcommand{\ge}{\geqslant}
\renewcommand{\leq}{\leqslant}
\renewcommand{\geq}{\geqslant}

\renewcommand{\Pr}{\mathbb{P}}
\newcommand{\DKL}{D_{\mathrm{KL}}}
\newcommand{\Qr}{\mathbb{Q}}
\renewcommand{\emptyset}{\varnothing}

\newcommand{\cA}{\mathcal{A}}
\newcommand{\cB}{\mathcal{B}}
\newcommand{\cC}{\mathcal{C}}

\newcommand{\cE}{\mathcal{E}}

\newcommand{\cH}{\mathcal{H}}

\newcommand{\cJ}{\mathcal{J}}
\newcommand{\cL}{\mathcal{L}}

\newcommand{\cS}{\mathcal{S}}

\newcommand{\eps}{\varepsilon}

\newcommand{\1}{\mathbbm{1}}

\DeclareMathOperator{\Ex}{\mathbb{E}}

\DeclareMathOperator{\Aut}{Aut}
\DeclareMathOperator{\Emb}{Emb}

\DeclareMathOperator{\UT}{UT}

\DeclareMathOperator{\spn}{span}

\DeclareMathOperator{\ent}{H}
\DeclareMathOperator{\Pois}{Pois}

\title[Upper tails for irregular graphs beyond the mean-field regime]{Upper tails for irregular graphs \\ beyond the mean-field regime}

\thanks{This research was supported by: the Israel Science Foundation grant 2110/22 and the ERC Consolidator Grant 101044123 (RandomHypGra).}

\date{\today}

\author{Asaf Cohen Antonir}
\address{School of Mathematical Sciences, Tel Aviv University, Tel Aviv 6997801, Israel}
\email{asaf123654@gmail.com}

\author{Matan Harel}
\address{Department of Mathematics, Northeastern University, Boston, MA, USA}
\email{m.harel@northeastern.edu}

\author{Frank Mousset}
\email{moussetfrank@gmail.com}

\author{Wojciech Samotij}
\address{School of Mathematical Sciences, Tel Aviv University, Tel Aviv 6997801, Israel}
\email{samotij@tauex.tau.ac.il}

\begin{document}

\begin{abstract}
  Let $G_{n,p}$ be the binomial random graph of density $p$ and let $X_H$ be the number of copies of a fixed graph $H$ in $G_{n,p}$.
  We prove asymptotically tight bounds on the logarithmic upper-tail probability of $X_H$ whenever $H$ is a connected, irregular graph with maximum degree $\Delta \ge 2$ and $p \ge n^{-1/\Delta - \eps_H} (\log n)^{\omega(1)}$ for an explicit $\eps_H >0$.
  These bounds are expressed in terms of a new variational problem that generalises the combinatorial optimisation problem arising from the na\"ive mean-field approximation.
  This new variational problem includes an entropy term that corresponds to the large number of embeddings of certain highly structured graphs in $K_n$.
  For a certain class of irregular graphs $H$ that we call stable, we show that this description of the upper-tail probability is valid in a range of densities that is optimal up to a poly($\log\log n$) factor. 
  For a further subclass of stable graphs, which includes all irregular complete bipartite graphs, we show that this range of densities is optimal up to a multiplicative constant. 
\end{abstract}

\maketitle

\section{Introduction}
\label{sec:introduction}

Given a positive integer $n$ and $p \coloneqq  p(n) \in [0,1]$, we define the binomial random graph model $G_{n,p}$ by including each edge of $K_n$, the complete graph on $n$ vertices, with probability $p$, independently of all other edges. For a fixed nonempty graph $H$, let $X_H \coloneqq X_H^{n,p}$ be the number of copies of $H$ in $G_{n,p}$, that is, the number of subgraphs of $G_{n,p}$ isomorphic to $H$.

The typical behaviour of $X_H$ in the asymptotic regime $n \to \infty$ is well-understood in all relevant regimes. A simple first-moment argument shows that $X_H = 0$ with high probability as long as $p \ll n^{-1/m(H)}$, where $m(H) \coloneqq \max \{e_J/v_J : \emptyset \subsetneq J \subseteq H\}$ is the so-called \emph{(maximum) density} of $H$. On the other hand, if $p \gg n^{-1/m(H)}$, then $\Ex[X_H] \to \infty$ and $X_H/\Ex[X_H] \to 1$ in probability, which can be easily proved via a second-moment argument.  If we additionally assume that $p$ does not approach 1 too quickly, $X_H$ satisfies a central limit theorem (see, e.g.~\cite{barbour1989central,MicNieSer24,Ros11}). The goal of this work is to study the atypical behaviour of $X_H$ as $n \to \infty$, in as large a swath of the sparse regime as possible. More precisely, we wish to compute the asymptotic rate of the logarithmic upper-tail probability $\log \Pr\left(X_H \ge (1 + \delta) \Ex[X_H]\right)$ for every fixed $\delta > 0$ and all $p \ll 1$. For notational convenience, we refer to this event as $\UT_{H,\delta}$.

Although this problem is very natural, the non-linearity of $X_H$ makes the analysis of the large-deviation regime much more elaborate than a similar analysis of the typical behaviour. One manifestation of this can be seen in the divergent phenomenologies displayed by the lower- and upper-tail events. On the one hand, a combination of Harris's inequality~\cite{harris1960lower} and Janson's inequality~\cite{janson1990poisson} immediately implies that, for any fixed $\delta \in (0,1]$, the logarithmic lower-tail probabilities satisfy
\[
  \log \Pr\bigl(X_H \le (1-\delta) \Ex[X_H]\bigr) = - \Theta_\delta\bigl(\min\{\Ex[X_J] : \emptyset \neq J \subseteq H\}\bigr).
\]
We note that, although determining the order of magnitude in the logarithmic lower-tail probability is relatively straightforward, evaluating the exact asymptotics is more delicate (see Kozma and Samotij~\cite{kozma2023lower}).

In contrast, the upper-tail behaviour can be influenced by the appearance of small but highly-structured subgraphs, which can increase $X_H$ dramatically without changing the global structure of the graph.
Obtaining good upper and lower bounds on the logarithmic upper-tail probability of $X_H$ proved challenging already in the seemingly simple case where $H= K_3$.
Progressively stronger upper bounds were established in~\cite{janson2004deletion, kim2004divide, Vu02}, ending with the work of Janson, Oleszkiewicz, and Ruciński~\cite{JanOleRuc04}, who proved upper and lower bounds on the logarithmic upper-tail probability that disagreed by a factor of $O(\log(1/p))$; see the classical survey of Janson and Ruciński~\cite{JanRuc2002} for more details.
In the case where $H$ is a clique, Chatterjee~\cite{chatterjee2012missing} and DeMarco--Kahn~\cite{DeMKah12-K3,demarco2012tight} independently added the logarithmic term missing in the upper bound, thus determining the order of magnitude of the logarithmic upper-tail probability for such graphs.

Starting with the seminal work of Chatterjee and Dembo~\cite{chatterjee2016nonlinear}, the last decade has seen the development of general, analytic theory of large deviations of non-linear functions of independent random variables.
By viewing $X_H$ as a polynomial in Bernoulli-$p$ random variables, one can apply this theory to prove a variational description of the logarithmic upper-tail probability for certain graphs $H$ and densities $p$.
The original work of Chatterjee and Dembo held for general $H$, but only when $p \ge n^{-c_H}$ for some small $c_H > 0$.
Further developments by Eldan~\cite{eldan2018gaussian}, Cook--Dembo~\cite{cook2020large}, and Augeri~\cite{augeri2020nonlinear,augeri2021transportation} extended the validity of these estimates to a broader range of densities.
From the perspective of the upper tails of $X_H$, the work of Cook, Dembo, and Pham~\cite{CooDemPha2024} is the strongest result in this vein.
It combines combinatorial and analytic techniques to study large deviations of sub(hyper)graph counts in binomial random {\em hypergraphs}.
Specialising the results of~\cite{CooDemPha2024} to the setting of subgraph counts in $G_{n,p}$, one obtains a variational description of the logarithmic upper-tail probability for $X_H$ for every graph $H$ with maximum degree $\Delta \ge 2$ and all $p \geq n^{-1/\Delta + c_H}$, for some explicit $c_H >0$.
The associated variational problem (discussed in detail in~\cref{sec:naive-mean-field}) was solved by Lubetzky--Zhao~\cite{lubetzky2017variational} in the case where $H$ is a clique, for all densities $p$ (except when $p = \Theta(n^{-1/\Delta})$), and by Bhattacharya, Ganguly, Lubetzky, and Zhao~\cite{bhattacharya2017upper} for general $H$, but only when $p \gg n^{-1/\Delta}$.

The work of Harel, Mousset, and Samotij~\cite{HarMouSam22} returned to the combinatorial approach to the upper-tail problem pioneered by Janson, Oleszkiewicz, and Ruciński~\cite{JanOleRuc04}.
At the root of the method is the following simple observation: one can increase the number of copies of $H$ found in $G_{n,p}$ by {\em planting} (i.e., conditioning on the appearance of) a graph $G$. To be more precise, we consider the family of graphs 
\[
  \cS_{\delta} \coloneqq \{G \subseteq K_n: \Ex[ X_H \mid G \subseteq G_{n,p}] \geq (1 + \delta) \Ex[X_H]\},
\]
which we call $\delta$-seeds. 
Heuristically, the event $\UT_{H,\delta}$ is `typical' under $\Pr(\cdot \mid G \subseteq G_{n,p})$, and therefore, $\log \Pr(\UT_{H,\delta})$ will be bounded from below by $\log \Pr( G \subseteq G_{n,p}) = -e(G) \log(1/p)$, up to lower-order corrections. By choosing a seed with the minimum number of edges, we can expect that
\begin{equation}\label{eq:ENMF}
\log \Pr(\UT_{H,\delta}) \geq - (1 + o(1)) \cdot \psi_H(\delta),
\end{equation}
where $\psi_H(\delta)$ is the variational problem
\[
 \psi_H(\delta) \coloneqq \min\big\{e(G) \log(1/p) : G \in \cS_\delta\}.
\]

The main contribution of~\cite{HarMouSam22} is the formulation of a sufficient condition that implies that~(\ref{eq:ENMF}) is asymptotically tight.
First, the work adapts a moment argument from the aforementioned work of Janson, Oleszkiewicz, and Ruciński~\cite{JanOleRuc04} to show that the upper-tail event is dominated by the appearance of `small' seeds, i.e., seeds with $O(\psi_H(\delta))$ edges.
From here, the sufficient condition is formulated in terms of a subclass of these `small' seeds, termed \emph{cores}.
Roughly speaking, these are seeds where every edge supports a rather large number of copies of $H$, in expectation.
The precise definition of cores is chosen so that every `small' seed must contain some core;
in particular, the probability that $G_{n,p}$ contains a `small' seed is bounded from above by the probability that it contains a core with $O(\psi_H(\delta))$ edges.
We say that cores are \emph{entropically stable} if these structural restrictions make these graphs rare, in the sense that there are $(1/p)^{o(m)}$ cores with $m$ edges for every $m = O(\psi_H(\delta))$.
As the probability that a given graph with $m$ edges appears in $G_{n,p}$ is $p^m$, when the entropic stability bound holds, a union bound over cores implies that the upper-tail probability is essentially bounded from above by the probability of appearance of a smallest core.
Combining this with the lower bound~\eqref{eq:ENMF}, and using simple continuity properties of $\psi_H$, one obtains
\[
\log \Pr(\UT_{H,\delta}) = - (1 + o(1)) \cdot \psi_H(\delta).
\]

The second contribution of~\cite{HarMouSam22} is showing that cores are entropically stable whenever $H$ is a $\Delta$-regular, non-bipartite graph for a nearly-optimal range of densities $p$.
Later, Basak--Basu~\cite{BasBas2023} introduced a more flexible notion of entropic stability to extend the argument to bipartite graphs and a sharp range of densities $p$; for lower values of $p$, the logarithmic upper-tail probability is in the `Poisson regime,' and has qualitatively different behaviour.

To reiterate the goal of this work, given the established results on $\Delta$-regular graphs, we wish to study the logarithmic upper-tail probability of $X_H$ whenever $H$ is not regular, for as broad range of the density parameter $p$ as we can.
The fundamental issue we encounter is that, when $p \ll n^{-1/\Delta}$, the cores defined in~\cite{HarMouSam22} are not entropically stable.
In particular, it is not even true that the number of different embeddings of a smallest $\delta$-seed into $K_n$ is asymptotically negligible!
Despite this fact, it is still true that $\delta$-seeds that are minimal in an appropriate sense are highly-structured graphs, which constrains the entropic contribution of the number of their embeddings.
Heuristically, we will argue that, in a large range of densities $p$, the number of minimal $\delta$-seeds is dominated by the choice of the vertices they span, and that the additional entropy corresponding to the choice of edges within that vertex set is asymptotically irrelevant.
In order to formulate this formally, we introduce the notion of a \emph{hub-core}, which captures the properties of minimal $\delta$-seeds which allow us to prove this `second-order entropic stability'. To do so, we introduce some notation.

Let $H$ be a connected, irregular graph with maximum degree $\Delta \ge 2$.
For disjoint $A, B \subseteq V(H)$, we let $N(A)$ be the neighbourhood of $A$ in $H$ and denote by $H[A, B]$ the bipartite subgraph of $H$ induced by the sets $A$ and $B$.
In other words, $H[A, B]$ is the bipartite graph with partite sets $A, B$ that contains all edges of $H$ with one endpoint in each of $A$ and $B$. Let 
\[
\cJ = \cJ(H) \coloneqq \{H[A,N(A)] : A \text{ is an independent set of vertices of degree } \Delta \}.
\]
Let $r = r_H$ be the maximum of $|A|/|N(A)|$ over all $A$ as above and let $\cJ^* = \cJ^*(H)$ be the set of all $J \in \cJ(H)$ which achieve this maximum. It is straightforward to see that $1/\Delta \leq r < 1$, see~\cref{lemma:rH-bounds} below. Finally, we set 
\[
  d_H\coloneqq \min \{\deg_J v : J\in \cJ^*\land v\in V(J)\}.
\]
We also introduce the following shorthand notations, as they repeat many times throughout the paper:
\[
  f = f(n,p)\coloneqq \max\bigl\{n\cdot (np^{\Delta})^r,n^2p^{\Delta}\bigr\}
  \qquad
  \text{and}
  \qquad
  g = g(n,p) \coloneqq \min\{n, f(n,p)\}.
\]
We are now ready to define hub-cores.  Our definition involves a slowly-growing function $\omega$ whose precise choice
\[
  \omega = \omega(n,p) \coloneqq \log \log (1/p)
\]
is not important.

\begin{definition}[hub-cores]
  \label{dfn:hub-core}
Given $p \ll 1$, we say a graph $G \subseteq K_n$ with no isolated vertices is a hub-core if it contains a vertex cover of size at most 
 \[
    \min\left\{e_G / \omega, \bigl(\log\log(1/p)\bigr)^{3} \cdot \max\bigl\{1, (e_G/f)^{v_H}\bigr\} \cdot \max\{1, e_G/n\}\right\}.
  \]
If $p  \leq n^{-1/\Delta} \cdot (\log n)^{-\omega}$, we also require that the minimum degree of $G$ satisfies $\delta(G) \geq d_H$. 
\end{definition}
We set $\cL$ to be the set of all hub-cores in $K_n$.
Unlike the cores of~\cite{HarMouSam22}, hub-cores are not necessarily $\delta$-seeds. However, we will see that, in the cases that interest us, $\delta$-seeds contain large hub-cores, which will be crucial for our analysis.
Having established the notion of hub-cores, we can now consider a refined variational problem.

\begin{definition}[The entropic variational problem]
  For an irregular graph $H$, define its \emph{entropic variational problem} $\psi_H^*\colon [0, \infty) \to [0, \infty]$ by setting
  \[
    \psi_H^*(\delta) \coloneqq \min\big\{e(G) \log(1/p)-v(G)\log(en/g) : G\in \cL \cap \cS_\delta \text{ and } e_G \leq  f \cdot (\log(1/p))^2\big\}.
  \]
\end{definition}

The negative term in this variational problem corresponds to the logarithmic number of embeddings of the vertices of a hub-core into $K_n$; crucially, the entropy of the location of the edges within the vertex set of a hub-core does not need to be included, due to the structural constraints imposed on hub-cores.
Heuristically, the inclusion of this term allows us to upgrade the naive lower bound underlying $\psi$, which estimates the upper-tail event from below by $\{G \subseteq G_{n,p}\}$ for a fixed seed $G$, to a more sophisticated version that estimates the upper-tail event from below by the larger event $\{G' \subseteq G_{n,p} \text{ for some } G' \cong G\}$.
This improvement is asymptotically negligible for every graph $H$ with maximum degree $\Delta$ as long as $p \ge n^{-1/\Delta - o(1)}$; further, when $H$ is $\Delta$-regular, it is negligible except when $p$ is polylogarithmically close to the Poisson regime.
For irregular graphs $H$, however, $\psi^*(\delta)$ is strictly smaller than $\psi(\delta)$, to first order, as soon as $p \le n^{-1/\Delta - \Omega(1)}$. 
  
We note that, in his study of the upper tail of induced 4-cycles, the first author~\cite{Coh2024} also formulated a variational problem which accounted for the number of embeddings of minimal seeds, and used a different method to generalise the notion of entropic stability to prove that it gives tight estimates on the logarithmic upper-tail probability.

Our first theorem states that, for any irregular graph $H$ with maximum degree $\Delta \ge 2$, $\psi_H^*$ is an asymptotically-tight approximation to the logarithmic upper-tail probability of $X_H$ whenever $p \geq n^{-1/\Delta - \eps_H} (\log n)^{\omega(1)}$ for some explicit $\eps_H >0$ that we define in~\eqref{eq:eps_H_def} below.  We set
\[
  p_H\coloneqq \max \{n^{-1/m(H)},n^{-1/\Delta - \eps_H}\},
\]
where we recall that $m(H) = \max \{e_J/v_J : \emptyset \subsetneq J \subseteq H\}$.

\begin{theorem}\label{thm:main-dense-ish}
  For every irregular, connected graph $H$ with maximum degree $\Delta \ge 2$ and all $0 < \varepsilon < \delta$, there is a constant $K$ such that the following holds.
  For all $p_H \cdot (\log n)^K \le p \ll 1$ and all large enough $n$,
  \[
    (1-\varepsilon) \cdot \psi_H^*(\delta-\varepsilon) \le - \log \Pr\big(X_H \ge (1+\delta)\Ex[X_H]\big) \le (1+\varepsilon) \cdot \psi_H^*(\delta+\varepsilon).
  \]
\end{theorem}

We emphasise that, for any irregular graph $H$, \cref{thm:main-dense-ish} holds in a range of densities which extends {\em polynomially further} than $n^{-1/\Delta}$, and therefore covers ranges where the distinction between $\psi^*(\delta)$ and $\psi(\delta)$ is crucial. While this manuscript was prepared, Basak and Karmakar~\cite{BasSha2025} independently proved that $\psi_H(\delta)$ gives tight bounds whenever $p \gg n^{-1/\Delta}$. Their approach is similar to the approach of~\cite{HarMouSam22}; indeed, in this range, entropic stability in the original sense holds.

\subsection{Stable graphs}
In this section, we are interested in graphs $H$ such that $1/\Delta + \varepsilon_H > 1/m(H)$; we call such graphs \emph{stable}.
Although the definition of stable graphs involves the mysterious parameter $\eps_H$, there is another simple, combinatorial characterisation of stable graphs.
Indeed, stable graphs are those graphs $H$ for which $\cJ^*(H)$ coincides with the family of all densest subgraphs of $H$.
Formally, this is given in the following lemma, which we prove in \cref{sec:stable-graphs}.
\begin{restatable}{lemma}{stablechar}
  \label{lem:stable-characterisation}
  An irregular graph $H$ with maximum degree $\Delta$ is stable if and only if
  \[
    \cJ^*(H)=\{J\subseteq H: {e_J}/{v_J} = m(H)\}.
  \]
\end{restatable}

Since every stable graph $H$ satisfies $p_H = n^{-1/m(H)}$,
\cref{thm:main-dense-ish} yields asymptotics of the logarithmic upper-tail probability for stable graphs and $p\geq n^{-1/m(H)}(\log n)^{\omega(1)}$.
This assumption on $p$ is optimal up to the $(\log n)^{\omega(1)}$ factor, as for $p = \Theta\left(n^{-1/m(H)}\right)$, a Poisson-type lower bound for the upper-tail probability is larger than the one implicit in the definition of $\psi^*_H$.
Our next result extends \cref{thm:main-dense-ish} for stable graphs to a wider range of densities $p$ and shows its optimality up to a factor of $(\log\log n)^{O(1)}$.

\begin{theorem}
  \label{thm:main-near-appearance-threshold}
  For every stable graph $H$, there exists a $C_H > 0$ such that the following holds for every $\delta > 0$.
  \begin{enumerate}[label=(\roman*)]
  \item
    \label{item:main-near-appearance-threshold-UB}
    For every $\varepsilon > 0$, there exists an $L > 0$ such that, for all $Lp_H(\log n)^{C_H} \le p \ll 1$,
    \[
      - \log \Pr\big(X_H \ge (1+\delta)\Ex[X_H]\big) = (1\pm\varepsilon) \cdot \psi_H^*(\delta \pm \varepsilon).
    \]
  \item
    \label{item:main-near-appearance-threshold-LB}
    For every $\varepsilon > 0$, there exists an $\ell > 0$ such that, for all $p_H \ll p \leq \ell p_H(\log n/\log\log n)^{C_H}$,
    \[
      -\log \Pr(X_H\geq (1+\delta)\Ex[X_H]) \le \varepsilon \cdot \psi^*_H(\delta + \varepsilon).
    \]
  \end{enumerate}
\end{theorem}

The proof of part~\ref{item:main-near-appearance-threshold-UB} of this theorem relies on graphs we call \emph{tight cores}.
Roughly speaking, a tight core is a union of copies of $J\in \cJ^*$ that overlap in a `generic' way;
this notion is inspired by the counting arguments used in several recent works on Ramsey properties of random graphs (see, e.g., \cite{KupSam24}).
The following two properties make tight cores relevant to our setting.
First, when $p \gg p_H (\log n)^{C_H}$, a maximal tight core of a seed remains a seed.
Second, as long as $p \leq p_H \cdot n^{o(1)}$, $G_{n,p}$ conditioned on the upper-tail event does not contain tight cores with $\omega(\psi_H(\delta) / \log(1/p))$ edges.
Consequently, the upper-tail event is dominated by the appearance of a tight core with $O(\psi_H(\delta)/\log(1/p))$ edges that is also a seed.
This motivates our decision to call these objects `cores', even though their definition is somewhat orthogonal to earlier notions.

For generic stable graphs $H$, both parts of~\cref{thm:main-near-appearance-threshold} can be optimal.
However, if we consider certain subclasses of stable graphs, the theorem may be further tightened.
Without going into the technical details (see \cref{sec:struct-stable-graphs}), we divide stable graphs into \emph{clean} and \emph{non-clean} graphs. The simple structure of clean graphs allows us to relax the upper-bound assumption on $p$ in part \ref{item:main-near-appearance-threshold-LB} so that it matches the lower-bound assumption in part~\ref{item:main-near-appearance-threshold-UB} up to a constant factor.
As a result, we essentially settle the upper-tail problem for all clean graphs.

\begin{theorem}
  \label{thm:clean}
  Suppose that $H$ is a stable graph that is clean and let $C_H$ be as in \cref{thm:main-near-appearance-threshold}.
  For all $\delta, \varepsilon > 0$, there is an $\ell > 0$ such that, for all $p_H \ll p \leq \ell p_H(\log n)^{C_H}$, we have
  \[
    -\log \Pr(X_H\geq (1+\delta)\Ex[X_H]) \le \varepsilon \cdot \psi^*_H(\delta + \varepsilon).
  \]
\end{theorem}

It is easy to check that all irregular complete bipartite graphs, and in particular all stars, are stable and clean.  
(More generally, every stable graph $H$ for which $|\cJ^*(H)| = 1$ is clean.)
In a recent work, Akhmejanova and {\v S}ileikis~\cite{AkhSil2025} resolved the upper-tail problem for stars in all densities $p_H \ll p \ll 1$ and found an explicit function that is asymptotic to the logarithmic upper-tail probability, also in the range $p_H \ll p \ll p_H (\log n)^{C_H}$, where we only know that $\psi_H^*$ is not relevant.

On the other hand, for a wide class of stable, non-clean graphs, which we call one-dimensional, it is the lower bound on $p$ in \ref{item:main-near-appearance-threshold-UB} which may be relaxed. Due to its technical nature, we postpone the definition of one-dimensional graphs to \cref{sec:struct-stable-graphs}. For intuition, we note that every stable graph $H$ where $\cJ^*$ is a chain (in the sense that, for any $J,J' \in \cJ^*$, either $J \subseteq J'$ or $J' \subseteq J$) is one-dimensional.

\begin{theorem}
  \label{thm:non-clean}
  Suppose that $H$ is a stable, one-dimensional graph that is not clean and let $C_H$ be the constant defined in \cref{thm:main-near-appearance-threshold}.
  For all $\delta, \varepsilon > 0$, there is an $L > 0$ such that, for all $L p_H(\log n/\log\log n)^{C_H} \le p \ll 1$, we have
  \[
    -\log \Pr(X_H\geq (1+\delta)\Ex[X_H]) = (1\pm\varepsilon)\cdot \psi^*_H(\delta \pm \varepsilon).
  \]
\end{theorem}

The proofs of both extensions of~\cref{thm:main-near-appearance-threshold} depend on some delicate properties of stable graphs.
In the case of~\cref{thm:clean}, we construct a measure that combines planting a highly-structured graph $G^*$ with a Poisson-type process on subgraphs rooted in $G^*$.  This gives rise to a lower bound on the logarithmic upper-tail probability that surpasses the one given by $\psi_H^*$.  On the other hand, the main obstacle for~\cref{thm:non-clean} lies in the aforementioned fact that maximal tight cores may not be seeds when $p$ is near its lower threshold.
Using the entropy method, we are able to show that the probability of encountering such a pathology is negligible in comparison with the upper-tail probability.

\subsection{The na\"ive mean field approximation}
\label{sec:naive-mean-field}

An alternate perspective on the upper-tail problem comes through the framework of the {\em na\"ive mean field approximation}. For a pair $\Pr$ and $\Qr$ of measures on the same probability space $\Omega$, with $\Qr$ absolutely continuous with respect to $\Pr$, the {\em  Kullback--Leibler divergence} of $\Qr$ from $\Pr$ is defined by
\begin{equation}
  \DKL(\Qr \, \| \, \Pr) \coloneqq \Ex_{\Qr}\left[\log\left(\frac{d\Qr}{d\Pr}\right)\right] =
  \sum_{\omega \in \Omega} \Qr(\omega) \log
  \left(\frac{\Qr(\omega)}{\Pr(\omega)}\right),
\end{equation}
where $\Ex_{\Qr}$ is the expectation operator associated with the measure $\Qr$. It is straightforward to show that the logarithmic probability of {\em any} event $\cA$ can be obtained by optimising the Kullback--Leibler divergence over all measures that assign $\cA$ probability 1:
\begin{equation}\label{eq:DonskerVaradhan1}
  -\log \Pr(\cA) = \inf_{\substack{\Qr \ll \Pr, \\ \Qr (\cA)=1}} \DKL(\Qr \,\| \,\Pr).
\end{equation} 
The usefulness of such a formulation is limited by the fact that measures that assign the upper-tail events probability 1 may be quite difficult to analyse. The idea of the na\"ive mean-field approximation is to replace the complicated variational problem in~\eqref{eq:DonskerVaradhan1} by a simpler one, where the infimum ranges only over product measures (the assumption $\Qr(\cA) = 1$ must then be relaxed somewhat).  Roughly speaking, the na\"ive mean-field approximation holds if minimising over this smaller set still achieves~\eqref{eq:DonskerVaradhan1}, up to lower order corrections.
More formally, set 
\[
  \psi_H^{\text{NMF}}(\delta) \coloneqq
  \inf\left\{\sum_{e \in K_n} i_p(q_e) : \mathbf{q} \in [0,1]^{E(K_n)} \land \Ex_{G_{n,\mathbf{q}}}[X_H] \ge (1+\delta)\Ex[X_H]\right\},
\]
where $i_p(q)$ denotes the Kullback--Leibler divergence of a Bernoulli-$q$ measure from a Bernoulli-$p$ measure, and $G_{n,\mathbf{q}}$ is the inhomogeneous random graph where each edge $e \in K_n$ is included with probability $q_e$, independently. The work of Chatterjee--Dembo~\cite{chatterjee2016nonlinear} and its extensions were phrased as sufficient conditions for the applicability of the na\"ive mean field approximation. In this light, \cite{BasBas2023,HarMouSam22} show that, under appropriate assumptions, the upper-tail problem satisfies an `extremely' na\"ive mean field approximation: $\psi_H$ is the restriction of $\psi_H^{\text{NMF}}$ to vectors $\mathbf{q} \in \{p,1\}^{E(K_n)}$. \Cref{thm:main-dense-ish} proves that the same approximation holds for connected, irregular graphs with maximum degree $\Delta$, whenever $\psi_H^*(\delta) = \psi_H(\delta) \cdot (1 + o(1))$, which holds true when $p \ge n^{-1/\Delta - o(1)}$.

In the regime where $n^{-1/\Delta - \eps_H} \ll p \ll n^{-1/\Delta - \Omega(1)}$, however, the na\"ive mean field approximation may fail.
In most cases, the approximation by product measures gives the correct order of magnitude, but $\psi^*_H(\delta)$ offers a better leading constant. In some cases, however, the breakdown may be much more spectacular, with $\psi_H^*(\delta) \ll \psi_H(\delta)$. To see an example, consider the graph $H$ drawn in~\cref{fig:example}.
This graph is not stable, as the subgraph induced by $\{a_1, a_2, b_2, b_3, b_4\}$, which is not in $\cJ^*(H) = \{H\}$, achieves the maximum density $m(H)$.
A computation will show that $r_H = 2/5$, $p_H = n^{-2/3}$, and $f = n^{7/5} p^{8/5}$ in the range of interest, i.e., $p \le n^{-1/\Delta} = n^{-1/4}$.
  
We set $p = L p_H$ for some slowly growing $L = n^{o(1)}$, so that $g = f = L^{8/5}n^{1/3}$.
Assume that $x = x(n) \to 0$ (to be specified later) and let $G^*$ be the blow-up of $H$ where $b_1$ and $b_5$ are blown up by a factor of $x^{-3} f$ while $b_2$, $b_3$, and $b_4$ are blown up by a factor of $x^2 f$. One can verify that $N(H,G^*) \ge f^5 = n^{v_H} p^{e_H}$, which means that $G^*$ is a $\delta$-seed for some $\delta >0$.
Furthermore, the set $\{a_1,a_2\}$ is a vertex cover and thus $G^*$ is a hub-core.
Finally, since
\[
  e_{G^*} = 6 x^2 f + 2 x^{-3} f
  \qquad
  \text{and}
  \qquad
  v_{G^*} = 2 + 3 x^2 f + 2x^{-3} f \ge 2x^{-3} f,
\]
we find that
\begin{align*}
  \frac{e_{G^*} \log(1/p) - v_{G^*} \log(en/g)}{2f} & \leq (3x^2 + x^{-3}) \bigl(\log (n^{2/3}) - \log L \bigr) - x^{-3} \left(\log (n^{2/3}) - \frac{8 \log L}{5} \right) \\
  & = 2x^2 \log n + \bigl(3x^{-3}/5 + O(x^2)\bigr) \log L.
\end{align*}
 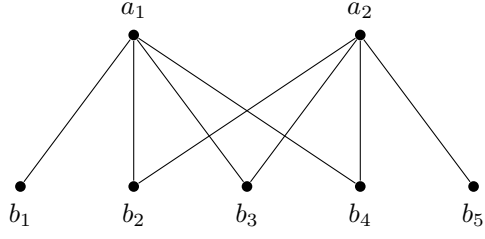
\begin{figure}[h]
   \centering
   \begin{tikzpicture}[
     vertex/.style={circle,fill,minimum size=4pt,inner sep=0pt}
     ]

     \node[vertex] (a1) at (0.5,2) {};
     \node[vertex] (a2) at (3.5,2) {};

     \node[vertex] (b1) at (-1,0) {};
     \node[vertex] (b2) at (0.5,0) {};
     \node[vertex] (b3) at (2,0) {};
     \node[vertex] (b4) at (3.5,0) {};
     \node[vertex] (b5) at (5,0) {};

     \node[above=3pt] at (a1) {$a_1$};
     \node[above=3pt] at (a2) {$a_2$};

     \node[below=3pt] at (b1) {$b_1$};
     \node[below=3pt] at (b2) {$b_2$};
     \node[below=3pt] at (b3) {$b_3$};
     \node[below=3pt] at (b4) {$b_4$};
     \node[below=3pt] at (b5) {$b_5$};

     \draw (a1) -- (b2);
     \draw (a1) -- (b3);
     \draw (a1) -- (b4);

     \draw (a2) -- (b2);
     \draw (a2) -- (b3);
     \draw (a2) -- (b4);

     \draw (a1) -- (b1);
     \draw (a2) -- (b5);

   \end{tikzpicture}   
   \caption{The example graph $H$}
   \label{fig:example}
 \end{figure}
 Setting $x^5 \coloneqq \log L/\log n$, which minimises the above expression (up to a constant multiplicative factor) while maintaining the $x=o(1)$ constraint, we find that, for some $\delta >0$,
 \[
   \psi_H^*(\delta) = O\bigl(f (\log n)^{3/5} (\log L)^{2/5}\bigr) \ll f \log(1/p) = \Theta(\psi_H(\delta)),
 \]
due to our assumption that $L = n^{o(1)}$.
If we additionally require that $L = (\log n)^{\omega(1)}$, \cref{thm:main-dense-ish} ensures that $- \psi^*_H(\delta)$ is asymptotic to the logarithmic upper-tail probability $\log \Pr(\UT_{H,\delta})$.
Further, since
\[
  \min\{\Ex[X_J] : \emptyset \neq J \subseteq H\} = \Theta(n^5p^6) = \Theta(nL^6)
\]
is polynomially larger than $f$, the above construction provides a counterexample to the conjecture of DeMarco and Kahn on possible rates of upper-tail probabilities for subgraph counts~\cite[Conjecture~10.1]{demarco2012tight}, which was disproved by {\v S}ileikis and Warnke~\cite{vsileikis2019counterexample}.

\section{Proof overview}
\label{sec:proof-overview}

Fix a connected, irregular graph $H$ with maximum degree $\Delta \ge 2$ and a constant $\delta>0$ and assume that $p \ll 1$.
Recall that $G \subseteq K_n$ is a \emph{$\delta$-seed} if
\[
  \Ex_G[X_H] \coloneqq \Ex[X_H \mid G\subseteq G_{n,p}] \geq (1+\delta) \Ex[X_H].
\]
Informally, $G$ is a $\delta$-seed if its presence in $G_{n,p}$ triggers the event $\UT_{H, \delta}$ in expectation.
A routine computation relates the surplus $\Ex_G[X_H] - \Ex[X_H]$ to the number of copies of all nonempty subgraphs of $H$ in $G$ (see \cref{cor:difference-in-expectations}):
\begin{equation}
  \label{eq:ExG-Ex-combinatorial}
  \Ex_G[X_H] - \Ex[X_H] = (1+o(1)) \cdot \sum_{\emptyset \neq J\subseteq H} \frac{|\Emb(J,H)|}{|\Aut(H)|} \cdot N(J,G)\cdot n^{v_H-v_J} p^{e_H-e_J},
\end{equation}
where $N(J,G)$ denotes the number of copies of $J$ in $G$ and $\Emb(J,H)$ is the set of embeddings of $J$ into $H$.
An immediate corollary of~\eqref{eq:ExG-Ex-combinatorial} is that, for every seed $G$, there must be some nonempty $J \subseteq H$ that appears in $G$ at least $\Omega(n^{v_J} p^{e_J})$ times.
Consequently, well-known estimates on the number of copies of a fixed graph in a graph with given order and size (which we reprove in \Cref{sec:preliminaries-embeddings} using a simplification of the elegant entropy-based argument of Friedgut and Kahn~\cite{FriKah98}) imply that every seed must have at least $\Omega(f)$ edges and, furthermore, if a seed $G$ has not many more than $O(f)$ edges, the sum in the right-hand side of~\eqref{eq:ExG-Ex-combinatorial} is dominated by the narrow family $\cJ$ of subgraphs of $H$; see \Cref{prop:difference-in-expectations-important-J} for a precise statement.
The range of densities $p$ for which we are able to analyse the upper tail of $X_H$ are precisely those $p$ above the appearance threshold $n^{-1/m(H)}$ for which the contribution of graphs in $\cJ$ dominates this sum.

To make this precise, we define $\eps_H$, the parameter involved in all the main theorems of this work.
First, given a graph $J \subseteq H$, we set $\alpha_J^*$ to be its fractional independence number (see~\cref{sec:preliminaries} for more details). Now, we define
\begin{equation}\label{eq:eps_H_def}
  \eps_H \coloneqq \min \left\{ \frac{v_J - \alpha^*_J - e_J/\Delta}{e_J - r \Delta \alpha^*_J} : \emptyset \neq J \subseteq H \land J \not \in \cJ^* \land r \Delta \alpha^*_J < e_J\right\},
\end{equation}
where the minimum of the empty set is defined to be $\infty$.
As will be shown in~\cref{prop:difference-in-expectations-important-J}, the assumption that $p \geq p_H (\log n)^K$ for a sufficiently large $K$ allows us to neglect the contribution of summands not associated with $\cJ$ in~\eqref{eq:ExG-Ex-combinatorial}.

  In the remainder of this section, we shall sketch our proofs of \cref{thm:main-dense-ish}, \cref{thm:main-near-appearance-threshold}~\ref{item:main-near-appearance-threshold-UB}, and \cref{thm:non-clean}.
  As the proofs of the lower and the upper bounds on the probability of $\UT_{H,\delta}$ are very different, we handle them separately, in \cref{sec:outline-lower-bounds,sec:outline-upper-bounds}, respectively.
  Finally, in \cref{sec:outline-optimality}, we briefly outline the argument proving the remaining results, that is, \cref{thm:main-near-appearance-threshold}~\ref{item:main-near-appearance-threshold-LB} and \cref{thm:clean}.

\subsection{Lower bounds}
\label{sec:outline-lower-bounds}

As was observed in \cite{HarMouSam22}, a simple use of Markov's inequality (\cref{lem:lower}) shows that, for any $(\delta+\varepsilon)$-seed $G$, we have
\begin{equation}
  \label{eq:UTHd-lower}
  \Pr(\UT_{H,\delta}) \geq \Pr(\UT_{H, \delta} \mid G \subseteq G_{n,p}) \cdot \Pr(G \subseteq G_{n,p}) \ge \varepsilon p^{e_H} \cdot p^{e_G}.
\end{equation}
In particular, to obtain the strongest lower bound for the upper-tail probability using this approach, one should consider a seed $G$ with as few edges as possible; this naturally leads to the definition of $\psi_H$.
Combined with straightforward continuity considerations, the above argument shows that $-(1+o(1)) \cdot \psi_H(\delta)$ is a lower bound for $\log \Pr (\UT_{H, \delta})$.  This lower bound turns out to be optimal when $p \ge n^{-1/\Delta-o(1)}$.
As mentioned before, when $p \leq n^{-1/\Delta - \Omega(1)}$, there is an obvious inefficiency in the above strategy.
Instead of conditioning on $G_{n,p}$ containing  a \emph{particular} seed $G \subseteq K_n$, we may condition on the appearance of \emph{some} copy of $G$ in $G_{n,p}$, hoping that
\begin{equation}
  \label{eq:UTHd-lower-with-UB}
  \Pr(\UT_{H, \delta}) \gtrapprox \Pr(G' \subseteq G_{n,p} \text{ for some } G'\cong G),
\end{equation}
which yields a stronger lower bound.

For every seed $G$ that is also a hub-core, we will be able to justify a slight relaxation of~\eqref{eq:UTHd-lower-with-UB}, where $G'$ ranges over a large subfamily of isomorphic copies of $G$, as well as prove that, for each such $G$ whose number of edges is not much larger than $f$,
\[
  \log \Pr(G' \subseteq G_{n,p} \text{ for some } G'\cong G) \ge (1+o(1)) \cdot \left(e_G \log p + \log \binom{n}{v_G}\right).
\]
Since it is not hard to check that every hub-core $G$ satisfies $v_G \approx \min\{n, e_G\}$, we further have
\[
  \log \binom{n}{v_G} \approx v_G \log (en/v_G) \approx v_G \log (en/g);
\]
this leads to the desired bound $\log \Pr(\UT_{H,\delta}) \ge -(1+o(1)) \cdot \psi_H^*(\delta)$.
The above argument requires some care as, for certain graphs $H$, the terms $e_G \log(1/p)$ and $v_G \log(en/g)$ can be asymptotically equal near the threshold $p_H$, see \cref{lem:weak-LB-psi-*}.
The details are presented in \cref{sec:lower-bounds}.  

\subsection{Upper bounds}
\label{sec:outline-upper-bounds}

Using a conditional version of the elegant high-moment argument due to Janson, Oleszkiewicz, and Ruciński~\cite{JanOleRuc04} (\Cref{lem:stability}), we show that the upper-tail event $\UT_{H,\delta}$ is dominated by the appearance of a seed with $O(f \log(1/p))$ edges.
To be more precise, we let $\cS_{\delta,m}$ be those $\delta$-seeds that have at most $m$ edges and denote by $\langle \cS_{\delta,m}\rangle$ the family of all subgraphs of $K_n$ which include some $G \in \cS_{\delta ,m}$.
The moment argument allows us to prove that, for any $\eps >0$, there exists a $K$ such that 
\[
  \Pr(G_{n,p}\in \langle \cS_{\delta - \eps, Kf \log(1/p)} \rangle\mid \UT_{H,\delta}) = 1-o(1).
\]

Suppose that $G$ is a seed with $O(f \log(1/p))$ edges.
As we mentioned above, the sum in the right-hand side of~\eqref{eq:ExG-Ex-combinatorial} is dominated by the contribution of graphs $J$ that belong to the family $\cJ$, see \Cref{prop:difference-in-expectations-important-J}.
By iteratively removing from $G$ edges that support `few' copies of every $J \in \cJ$, we produce a $G^* \subseteq G$ with $G^* \in \cS_{\delta - 2 \eps}$. Crucially, since every edge of $G^*$ now supports many copies of some $J\in \cJ$, a simple combinatorial lemma (\Cref{lemma:core-edge-bipartite,cor:core-edge-bipartite}) that exploits the bipartite structure of graphs in $\cJ$ implies a strong lower bound on $\deg_{G^*}u+\deg_{G^*}v$ for every $uv\in E(G^*)$;  this yields the existence of a small vertex cover --- forcing $G^*$ to be a hub-core (see~\cref{lem:cores}). In the regime $p \ll n^{-1/\Delta} (\log n)^{-\omega(1)}$, the family of `important' subgraphs $J$ can be further narrowed down from $\cJ$ to $\cJ^*$, which allows us to further conclude that $\delta(G^*) \ge d_H$.

Thus, we conclude that the upper-tail event is dominated by the appearance of a hub-core which is also a $\delta - 2\eps$ seed. Thanks to the existence of small vertex covers, the number of hub-cores with $m$ edges in $K_n$ is dominated by the number of choices for the embedding of their (non-isolated) vertices (see~\Cref{lem:core-count}). At densities $p \leq n^{-1/\Delta - \Omega(1)}$, the number of such embeddings may be as large as $p^{-\Theta(m)}$; indeed, as was seen in the example discussed at the end of \cref{sec:naive-mean-field}, the effect of the number of embeddings may even reduce the order of magnitude of $\psi_H^*$.  Nevertheless, the adjustments made by $\psi_H^*$ still capture the logarithmic upper-tail probability as long as the assumptions of~\cref{thm:main-dense-ish} hold. The details are presented in \Cref{sec:upper-bounds}.

The argument sketched above remains valid as long as $p$ is polylogarithmically above $p_H$.
When this is no longer true, two obstructions arise.
First, if $H$ is not stable, then \Cref{prop:difference-in-expectations-important-J} no longer describes the graphs $J$ that dominate the sum in the right-hand side of~\eqref{eq:ExG-Ex-combinatorial} for all possible seeds $G$, as subgraphs of $H$ not in $\cJ^*$ may contribute significantly to the sum.
Moreover, when $p \ll p_H = n^{-1/\Delta-\varepsilon_H}$, the smallest number of edges in a seed is no longer $\Theta(f)$.
Finally, if $H$ is stable, the `small' error term in the aforementioned upper bound on the number of hub-cores with $m$ edges (\Cref{lem:core-count}) is negligible only when $m = O(f)$, but not for larger values of $m$, see~\cref{lem:weak-LB-psi-*}~\ref{item:weak-LB-psi-*-non-stable}.

To remedy this issue and prove \cref{thm:main-near-appearance-threshold}~\ref{item:main-near-appearance-threshold-UB} in the regime $p \le p_H (\log n)^{O(1)}$, we introduce \emph{tight cores}.
Tight cores are graphs that can be built from the empty graph by sequentially adding copies of graphs from $\cJ^*$ in such a way that (at least half the time) the newly added copy intersects the already built graph in a subgraph that does \emph{not} belong to the family $\cJ_{\emptyset}^* \coloneqq \cJ^* \cup \{\emptyset\}$ (see \Cref{def:tight-cores}).
Defining tight cores in this particular fashion guarantees that the expected number of tight cores with $m$ edges is at most $p^{\Omega(m)}$, as shown in \Cref{lem:large-tight-core-is-expensive}.

Given a $(\delta-\varepsilon)$-seed $G$, let $G^* \subseteq G$ be an inclusion-maximal tight core.
The case where $e_{G^*} \gg f$ can be easily ruled out -- the aforementioned \Cref{lem:large-tight-core-is-expensive} asserts that the expected number of such tight cores is much smaller than the upper-tail probability.
We may therefore assume from now on that $G^*$ has $O(f)$ edges.
The maximality of $G^*$ implies that every copy of each $J \in \cJ^*$ in $G$ intersects $G^*$ in some $J_0 \in \cJ_\emptyset^*$.
This fact allows us to bound the number of copies of each such $J$ from above by a polynomial in $e_G$ and $e_{G^*}$ that involves the \emph{dimensions} $\dim_{J_0} J$ of $J$ over its subgraphs $J_0$, defined in~\cref{sec:struct-stable-graphs}, and the fractional independence numbers $\alpha_{J_0 - R_{J_0}(J)}^*$, where $R_{J_0}(J) \coloneqq \{v \in V(J_0) : \deg_J v > \deg_{J_0} v\}$ is the set of \emph{roots} of $J$ over $J_0$, see~\cref{lem:counting-extensions-via-dimension}.
The constant $C_H$ that appears in the statement of~\cref{thm:main-near-appearance-threshold} is defined precisely so that, when $p \gg p_H (\log n)^{C_H}$, the total contribution to the right-hand side of~\eqref{eq:ExG-Ex-combinatorial} of copies of $J$ in $G$ that are not fully contained in $G^*$ is $o(n^{v_H} p^{e_H})$.
This means that, when $p \gg p_H(\log n)^{C_H}$, our maximal tight core $G^*$, which has only $O(f)$ edges, is also a $(\delta-2\varepsilon)$-seed.
This in turn means that we can repeat the argument relying on hub-cores, exactly as we did in the denser regime, as the error term in our upper bound on the number of hub-cores with $O(f)$ edges supplied by \cref{lem:core-count} is negligible.

For certain stable graphs $H$, one can tighten the above argument and prove the validity of the entropic variational principle under the weaker assumption that $p \gg p_H (\log n / \log \log n)^{C_H}$, which is optimal up to a multiplicative constant (as in~\cref{thm:non-clean}).
The approach using tight cores described above runs into an obstacle once $p = O(p_H(\log n)^{C_H})$, as now the copies of some $J \in \cJ^*$ that are not entirely contained in the tight core $G^*$ can carry meaningful contribution to the sum in the right-hand side of~\eqref{eq:ExG-Ex-combinatorial}.
As a result, we cannot rule out the possibility that a $(\delta-\eps)$-seed $G$ contains a maximal tight core $G^*$ with $O(f)$ edges that is not a $(\delta-2\eps)$-seed. Under the assumptions of~\cref{thm:non-clean}, we are able to show that, while deterministically possible, such a seed appears in $G_{n,p}$ with probability that is significantly smaller than the upper-tail probability. This is because, for some pair $J_0, J \in \cJ_\emptyset^*$ with $J_0 \subsetneq J$, the graph $G \setminus G^*$ must contain a collection of pairwise-disjoint copies of $J \setminus J_0$ whose size is much larger than what one would expect to see in $G_{n,p}$. We may thus again assume that $G^*$ is a $(\delta-2\varepsilon)$-seed with $O(f)$ edges and argue as we did before.
The details are presented in \Cref{sec:upper-bounds-stable}.

\subsection{Optimality}
\label{sec:outline-optimality}

Finally, we briefly outline the proof of a lower bound on $\Pr(\UT_{H,\delta})$, valid for all stable graphs $H$, that is different from the one discussed in \cref{sec:outline-lower-bounds} and becomes stronger once $p$ drops somewhat below $p_H (\log n)^{C_H}$.
This alternate lower bound for the upper-tail probability yields \cref{thm:main-near-appearance-threshold}~\ref{item:main-near-appearance-threshold-LB} and \cref{thm:clean}.
We choose a pair $J_0, J \in \cJ_\emptyset^*$ with $J_0 \subsetneq J$ that achieves the maximum in the definition of $C_H$ and trigger the upper-tail event in two stages.
First, we plant in $G_{n,p}$ a blow-up $G^*$ of $J_0$ with $o(f)$ edges, which costs us only $p^{o(f)}$.
Guided by the aforementioned \cref{lem:counting-extensions-via-dimension}, we choose the sizes of the independent sets of $G^*$ that correspond to the vertices of $J_0$ according to a fractional independent set whose restriction to $V(J_0) \setminus R_{J_0}(J)$ achieves its maximum possible value $\alpha_{J_0 - R_{J_0}(J)}^*$.
Second, we condition $G_{n,p}$ on containing a large number of copies of $J \setminus J_0$ that collectively extend the copies of $J_0$ in $G^*$ to $\Omega(n^{v_J} p^{e_J})$ copies of $J$ in $G_{n,p}$;  this is very likely to trigger the upper-tail event $\UT_{H,\delta}$, see~\eqref{eq:ExG-Ex-combinatorial} and~\eqref{eq:UTHd-lower}.
The structural characterisation of the pair $(J_0, J)$ established in~\cref{prop:optimal-pair-decomposition} allows us to represent the latter event as an intersection of $\dim_{J_0} J$ simpler events whose probabilities can be bounded from below by upper-tail probabilities of Poisson random variables (using the naive estimate provided in \cref{lem:Poisson-lemma}), each of which is also $p^{o(f)}$.
The details are presented in \cref{sec:LB-near-the-appearance-threshold}.

\section{Preliminaries}
\label{sec:preliminaries}

\subsection{Notation and basic facts}
\label{sec:preliminaries-notation}

Let $G$ and $H$ be two graphs.  We write $V(G)$ and $E(G)$ for the sets of vertices and edges of $G$, respectively, and denote their cardinalities by $v_G=v(G) \coloneqq |V(G)|$ and $e_G = e(G) \coloneqq |E(G)|$.
The maximum degree of $G$ is denoted by $\Delta(G)$ and its independence number by $\alpha_G$.
The \emph{fractional independence number} of $G$, denoted by $\alpha^*_G$, is the maximum value of $\sum_{v\in V(G)} a_v$ over all $a \colon V(G)\to [0,1]$ such that $a_u+a_v\leq 1$ for all $uv\in E(G)$.
We will frequently use the basic inequality $\alpha_G^* \ge \max\{\alpha_G, v_G/2\}$, which follows by considering the indicator function of some largest independent set in $G$ and the constant function that assigns each vertex of $G$ weight $1/2$.
Finally, we shall often identify a graph with its set of edges and write $uv \in G$ instead of $uv \in E(G)$.

\begin{lemma}
  \label{lemma:rH-bounds}
  For every connected, irregular graph $H$, we have $1/\Delta(H) \le r_H < 1$.
\end{lemma}
\begin{proof}
  Write $\Delta \coloneqq \Delta(H)$.
  To prove the inequality $r_H \ge 1/\Delta$, consider $A \coloneqq \{v\}$ for an arbitrary vertex $v \in V(H)$ with degree $\Delta$.  To see that the claimed upper bound on $r_H$ holds, consider an arbitrary independent set $A \subseteq V(H)$ comprising vertices of degree $\Delta$.  Observe that
  \[
    |A| \cdot \Delta = e\bigl(H[A, N(A)]\bigr) \le |N(A)| \cdot \Delta
  \]
  and that the above inequality holds with equality if and only if every vertex in $N(A)$ has $\Delta$ neighbours in $A$.  In particular, if $|A| = |N(A)|$, then the graph $H[A, N(A)]$ would be $\Delta$-regular, which would contradict the assumption that $H$ is connected, irregular, and has maximum degree~$\Delta$.
\end{proof}

Throughout, we will write $\Pois(\mu)$ to denote the Poisson distribution with mean $\mu$.
Recall the following simple estimate on the upper tail of $\Pois(\mu)$.

\begin{fact}
  \label{fact:lower-bound-for-Poisson}
  For every real $\mu > 0$ and all integers $t \ge \max\{2 \mu, 3\}$,
  \[
    \log \Pr(\Pois(\mu) = t) = \log \left(\frac{e^{-\mu} \mu^t}{t!}\right) \ge -t \log(t/\mu).
  \]
\end{fact}

In our proof of \cref{thm:main-near-appearance-threshold}~\ref{item:main-near-appearance-threshold-LB} and \cref{thm:clean}, we will use the following estimate that bounds the upper-tail probabilities of a broad family of combinatorially-defined random variables from below by the upper-tail probabilities of Poisson random variables with the same expectation, modulo a small error term.  For a collection $\mathcal{H}$ of subsets of a set $V$ (a hypergraph on $V$) and a subset $R \subseteq V$, we let $\mathcal{H}[R] \coloneqq \{A \in \mathcal{H} : A \subseteq R\}$ be the subhypergraph of $\mathcal{H}$ induced by $R$.

\begin{restatable}{lemma}{Poissonlemma}
  \label{lem:Poisson-lemma}
  There is an absolute constant $C>0$ such that the following holds:
  Let $\mathcal{H}$ be a~nonempty collection of $r$-element subsets of a finite set $V$.
  Suppose that $p \in (0, 1)$, let $Y \coloneqq |\mathcal{H}[V_p]|$, and assume that $\mu\coloneqq \Ex[Y] \ge 1$.
  Then, for all $t$ satisfying $2\mu \leq t \leq |\mathcal{H}|/2$,
  \[
    \Pr (Y\geq t) \geq C^{-t} \cdot \max\left\{\frac{\Pr(\Pois(\mu)\geq t)}{|\mathcal{H}|^3} , \frac{\Pr(\Pois(\mu) \ge t)^4}{(4\Ex[Y^2] / \mu)^3} \right\}.
  \]
\end{restatable}

The technical proof of~\cref{lem:Poisson-lemma} is postponed to the appendix (\cref{appendix:Poisson-lemma}).

\subsection{Fractional graph theory}
\label{sec:preliminaries-fractional}

As many prior works on the upper tail problem for subgraph counts, we will use two fundamental inequalities that relate the quantities $e_J$, $v_J$, and $\alpha_J^*$ in a subgraph $J$ of a connected graph with maximum degree $\Delta$.  The proofs of these inequalities use the following half-integrality property of the fractional independence number.

\begin{lemma}[folklore]
 \label{lemma:fractional-duality}
 Every graph $J$ admits a fractional independent set $a \colon V(J) \to [0,1]$ satisfying
 $\sum_{v \in V(J)} a_v = \alpha_J^*$ such that $a_v \in \{0, 1/2, 1\}$ for
 every $v \in V(J)$.  Moreover, there is a partition $V(J) = V_1 \cup V_2$ with
 $|V_1|/2 + |V_2| = \alpha_J^*$ such that $V_1$ can be covered by a collection
 of vertex-disjoint edges and cycles of $J$.
\end{lemma}

The following lemma is a reformulation of \cite[Lemma~5.3]{HarMouSam22}.  Even though it follows from the proof of \cite[Lemma~5.3]{HarMouSam22}, we present the argument here for the sake of completeness.

\begin{lemma}\label{lemma:eJ-alphaJ-clique}
 Suppose that $J$ is a nonempty subgraph of a connected graph $H$ with maximum degree at most $\Delta$. Then
 \[
   e_J \leq \Delta\cdot(v_J-\alpha_J^*)
   \leq
   \Delta\cdot\alpha_J^*.
 \]
 The first inequality is tight if and only if
 \begin{enumerate}[label=(Q\arabic*)]
 \item
   \label{item:QH-regular}
   $J = H$ and $H$ is $\Delta$-regular or
 \item
   \label{item:QH-bipartite}
   $J$ admits a bipartition $V(J) = A \cup B$ such that $\deg_Ja = \Delta$ for all $a \in A$.
 \end{enumerate}
 If both inequalities are tight, then $J = H$ and $H$ is $\Delta$-regular.
\end{lemma}
\begin{proof}
  By \cref{lemma:fractional-duality}, $J$ has a fractional independent set $a$ such that $a_v \in \{0,\frac12,1\}$.
  Then
 \begin{equation}
   \label{eq:optimal-alpha}
   e_J \leq \sum_{uv\in J}(2-a_u-a_v) =\sum_{v\in V(J)} (1-a_v) \deg_Jv \le \Delta \cdot \sum_{v\in V(J)} (1-a_v)  = \Delta \cdot (v_J - \alpha_J^*),
 \end{equation}
 which is the first inequality. For the second inequality, simply recall that $\alpha^*_J \geq v_J/2$.

 Further, observe that $e_J = \Delta\cdot (v_J-\alpha^*_J)$ if and only if both inequalities in~\eqref{eq:optimal-alpha} are equalities;
 this happens if and only if $a_u + a_v = 1$ for every edge $uv \in J$ and $\deg_Jv = \Delta$ whenever $a_v \neq 1$.
 Let $A$, $B$, and $C$ denote the sets of vertices that $a$ maps to $0$, $1$, and $1/2$, respectively.
 Thus, $e_J = \Delta \cdot (v_J - \alpha_J^*)$ if and only if each vertex in $A \cup C$ has degree $\Delta$ and each edge of $J$ has  either both endpoints in $C$ or one endpoint in each of $A$ and $B$.
 In
 particular, if $C$ is not empty, then it induces a $\Delta$-regular graph;
 hence $C = V(J)$ and $J = H$, as $H$ is connected and has maximum degree at most $\Delta$.
 Otherwise,
 if $C$ is empty, then $A \cup B$ is a bipartition of $J$ and all vertices
 of $A$ have degree $\Delta$.
 
 Lastly, suppose that $e_J = \Delta\cdot \alpha_J^*$,
 which implies $e_J = \Delta\cdot (v_J-\alpha^*_J)$. We claim that $J$ is $\Delta$-regular, which implies that
 $J = H$ as $H$ is a connected graph with maximum degree at most $\Delta$.
 Let $A,B,C$ be the same partition as above.
 If $C$ is nonempty, then $J$ is $\Delta$-regular, and we are done.
 Otherwise,  if $C$ is empty, then $A \cup B$ is a bipartition of $J$
 and all vertices of $A$ have degree $\Delta$. Moreover,
 \[
   |A| = e_J/\Delta =  \alpha_J^* = v_J - e_J/\Delta = v_J-|A|= |B|,
 \]
 and thus also every vertex of $B$ has degree $\Delta$.
\end{proof}

\begin{lemma}
  \label{lem:unique-frac-ind-set}
  Suppose that $J$ is a bipartite graph with maximum degree $\Delta$ and no $\Delta$-regular connected components.
  Suppose further that $J$ has a bipartition $V(J) = A \cup B$ such that $\deg a = \Delta$ for every $a \in A$.
  Then, the characteristic function of $B$ is the unique largest fractional independent set of $J$.
\end{lemma}

\begin{proof}
  Let $\alpha \colon V(J) \to [0,1]$ be some largest fractional independent set.
  We have
  \begin{multline}
    \label{eq:unique-frac-ind-set}
    (v_J - |B|) \cdot \Delta = |A| \cdot \Delta = e_J \le \sum_{uv \in J} (2 - \alpha_u - \alpha_v) \\
    = \sum_{v \in V(J)} (1-\alpha_v) \cdot \deg v \le \sum_v (1-\alpha_v) \cdot \Delta = (v_J - \alpha_J^*) \cdot \Delta.
  \end{multline}
  Since $\alpha_J^* \ge |B|$, as $B$ is an independent set, we not only have $\alpha_J^* = |B|$, but also the two inequalities in~\eqref{eq:unique-frac-ind-set} hold with equality.
  This means that $\alpha_u + \alpha_v = 1$ for every $uv \in J$ and $\alpha_v = 1$ for every $v$ with $\deg v < \Delta$;  consequently, $\alpha_u = 1$ also for every $u$ that is at an even distance from some $v$ with $\deg v < \Delta$.
  Since $J$ has no $\Delta$-regular connected components and every vertex of $A$ has degree $\Delta$, every vertex of $B$ must lie at an even distance from some vertex with degree less than $\Delta$ and thus $\alpha_b = 1$ for all $b \in B$, and hence $\alpha_a = 0$ for all $a \in A$.
\end{proof}

We close this section with a simple fact that will be used repeatedly throughout the paper.

\begin{fact}
  \label{fact:f-p-pH}
  For every $J \in \cJ^*$ and all $p \ge p_H$, we have
  \[
    f^{\alpha^*_J} \ge n^{v_J}p^{e_J}\geq (p/p_H)^{e_J}.
  \]
  Moreover, if $p \le n^{-1/\Delta}$, then $f^{\alpha_J^*} = n^{v_J} p^{e_J}$.
\end{fact}
\begin{proof}
  Let $J\in \cJ^*$ and note that, by~\cref{lem:unique-frac-ind-set}, we have $r = (v_J-\alpha^*_J)/\alpha_J^*$.
  Therefore,  
  \[
    \Delta r=\frac{\Delta(v_J-\alpha^*_J)}{\alpha_J^*}=\frac{e_J}{\alpha^*_J}
    \qquad
    \text{and}
    \qquad
    1+r=\frac{v_J}{\alpha^*_J}.
  \]
  Consequently,
  \begin{equation}
    \label{eq:f-p-pH}
    f^{\alpha^*_J} \geq \left(n(np^{\Delta})^r\right)^{\alpha_J^*} = n^{v_J}p^{e_J} = n^{v_J} p_H^{e_J} \cdot (p/p_H)^{e_J} \ge (p/p_H)^{e_J},
  \end{equation}
  where the final inequality follows as $p_H \ge n^{-1/m(H)} \ge n^{-v_J/e_J}$.
  Moreover, if $p \le n^{-1/\Delta}$, then $f = n (np^{\Delta})^r$ and the first inequality in~\eqref{eq:f-p-pH} holds with equality.
\end{proof}

\subsection{Counting copies and embeddings}
\label{sec:preliminaries-embeddings}

We say that an injective map $\varphi\colon V(H)\to V(G)$ is an \emph{embedding} of $H$ into $G$ if for every $uv \in H$ we also have $\varphi(uv)\coloneqq \varphi(u)\varphi(v)\in G$; we write $\Emb(H,G)$ to denote the set of all such embeddings.
We refer to embeddings of $G$ into itself as automorphisms and write $\Aut(G) \coloneqq \Emb(G,G)$.
Given an edge $uv\in G$, we will write $\Emb(H,G;uv)$ to denote the set of all $\varphi \in \Emb(H,G)$ such that $uv\in \varphi(H)$.
Moreover, we write $N(H,G)$ to denote the number of \emph{unlabelled} copies of $H$ in $G$, that is $N(H,G)=|\Emb(H,G)|/|\Aut(H)|$.
Finally, for positive integers $m$ and $n$, we let $N(H,n,m)$ be the maximum of $N(H,G)$ over all graphs $G$ with $v_G\leq n$ and $e_G\leq m$.

The following theorem supplies an optimal (up to a constant multiplicative factor that depends only on $H$) bound on the number of embeddings of a fixed graph into a graph with a given number of edges and vertices.  It was proved by Janson, Oleszkiewicz, and Ruciński~\cite{JanOleRuc04}, but the version stated here is \cite[Theorem~5.7]{HarMouSam22}.
It refines an earlier result of Alon~\cite{Alo81}, who obtained optimal bounds on the number of embeddings of a fixed graph in a graph with a given number of edges only (i.e., without restricting the number of vertices).
The proof in~\cite{JanOleRuc04} is an adaptation of an elegant, entropy-based argument of Friedgut and Kahn~\cite{FriKah98}, who extended Alon's result to hypergraphs.
(We refer the interested reader to the excellent survey of Galvin~\cite{Gal} for an exposition of this argument that explicitly uses the language of entropy.)
Here, we present a simplified version of the arguments of~\cite{FriKah98,JanOleRuc04} that does not use linear programming duality.
A variation of this argument will also play a crucial role in the proof of \Cref{thm:main-near-appearance-threshold}, see \cref{claim:bound-on-Wlm}.

\begin{theorem}[{\cite{JanOleRuc04}}]
  \label{thm:max-copies}
  For every nonempty graph $J$ without isolated vertices and every graph $G$,
  \[
    |\Emb(J, G)| \le (2e_G)^{v_J - \alpha_J^*} \cdot \min\{2e_G, v_G\}^{2\alpha_J^* - v_J}.
  \]
\end{theorem}

The following proposition complements \Cref{thm:max-copies} by showing that the asserted upper bound on $|\Emb(J,G)|$ is best-possible up to a constant multiplicative factor.

\begin{proposition}
  \label{prop:max-copies}
  For every nonempty graph $J$ without isolated vertices, there exists a positive constant $c_J$ such that, for all integers $m$ and $n$ with $v_J^2 \le m \le n^2$, there is a graph $G$ with at most $n$ vertices and $m$ edges that is a blow-up of $J$ and satisfies
  \[
    N(J, G) \ge c_J \cdot m^{v_J-\alpha_J^*} \cdot \min\{m, n\}^{2\alpha_J^*-v_J}.
  \]
  Further, if $J$ is bipartite, then some such $G$ has a vertex cover of cardinality at most $\max\{v_J, m/n\}$.
\end{proposition}

As was observed already in~\cite{Alo81}, the problem of counting embeddings of a graph $J$ in another graph is intimately related to the following linear program.

\begin{LP}
  For a graph $J$ and $A, B \ge 0$, the linear program $\mathcal{P}(J;A,B)$ is:
  \begin{align*}
    \textbf{maximise} & \qquad \sum_{v \in V(J)} x_v \\
    \textbf{subject to} & \qquad     0\le x_v \le A \quad \text{ for all $v \in V(J)$}, \\
                      & \qquad 0 \le x_v+x_w \le B \quad \text{ for all $vw\in E(J)$}.
  \end{align*}  
\end{LP}

The following proposition was (implicitly) proved by Janson, Oleszkiewicz, and Ruciński~\cite{JanOleRuc04}.

\begin{lemma}
  \label{lem:LP}
  For every nonempty graph $J$ without isolated vertices and all $0 \le B \le 2A$, the solution to $\mathcal{P}(J;A,B)$ is
  \begin{equation}
    \label{eq:LP-solution}
    (v_J - \alpha_J^*) \cdot B + (2\alpha_J^* - v_J) \cdot \min\{A, B\}.
  \end{equation}
  Moreover, if $J$ is bipartite, then there is an optimal solution $x$ that takes the value $\max\{0, B-A\}$ on all vertices of some vertex cover of $J$.
\end{lemma}

\begin{proof}
  We first prove an upper bound on the solution to $\mathcal{P}(J;A,B)$.
  If $B \le A$, then the vertex constraints are implied by the edge constraints (because $J$ has no isolated vertices) and thus the solution to $\mathcal{P}(J;A,B)$ is $\alpha_J^* \cdot B$.
  If $B = 2A$, then the edge constraints are implied by the vertex constraints and thus the solution to $\mathcal{P}(J;A,B)$ is $v_J \cdot A$.
  We may thus assume that $A < B < 2A$.
  Let $x \colon V(J) \to [0,A]$ be an optimal solution to $\mathcal{P}(J;A,B)$.  Observe that $x_v \geq B-A$ for each $v \in V(J)$; indeed, otherwise, replacing $x_v$ with $B-A$ would yield a feasible solution with larger weight.  Define
  \[
    \xi \coloneqq \frac{x + A - B}{2A-B}.
  \]
  Since the feasibility of $x$ implies that $0 \le \xi_v \le 1$ for every $v \in V(J)$ and $\xi_v + \xi_w \le 1$ for every $vw \in E(J)$,
  we may conclude that $\xi$ is a fractional independent set of $J$ and thus
  \[
    \sum_{v \in V(J)} x_v
    = v_J \cdot (B-A) + \sum_{v \in V(J)} \xi_v \cdot (2A-B)
    \le v_J \cdot (B-A) + \alpha_J^* \cdot (2A-B).
  \]

  We now describe a solution to $\mathcal{P}(J;A,B)$ with weight~\eqref{eq:LP-solution}.
  Suppose that $0 \le B \le 2A$ and let $\xi \colon V(J) \to [0,1]$ be a fractional independent set with weight $\alpha_J^*$.
  We claim that
  \[
    x \coloneqq (1-\xi) \cdot \max\{0, B-A\} + \xi \cdot \min\{A, B\}
  \]
  is a feasible solution to $\mathcal{P}(J;A,B)$ of weight~\eqref{eq:LP-solution}
  Indeed, since $A \ge \max\{0, B-A\}$, we have $0 \le x_v \le A$ for all $v \in V(J)$ and
  \[
    \begin{split}
      x_v + x_w
      & = 2 \cdot \max\{0, B-A\} + (\xi_v + \xi_w) \cdot (\min\{A, B\} - \max\{0, B-A\}) \\
      & \le \max\{0, B-A\} + \min\{A, B\} = B
    \end{split}
  \]
  for all $vw \in E(J)$, so $x$ is a feasible solution.
  Further,
  \[
    \begin{split}
      \sum_{v \in V(J)} x_v
      & = v_J \cdot \max\{0, B-A\} + \sum_{v \in V(J)} \xi_v \cdot (\min\{A, B\} - \max\{0,B-A\}) \\
      & = v_J \cdot \max\{0, B-A\} + \alpha_J^* \cdot (\min\{A, B\} - \max\{0, B-A\}),
    \end{split}
  \]
  which is equal to~\eqref{eq:LP-solution}.
  Finally, assume that $J$ is bipartite.
  In this case, $\alpha_J^*$ is achieved by some $\xi \colon V(J) \to \{0,1\}$ and $\xi^{-1}(0) \subseteq x^{-1}(\max\{0, B-A\})$ is a vertex cover of $J$.
\end{proof}

\begin{proof}[Proof of \Cref{thm:max-copies}]
  Let $\preceq$ be an arbitrary linear order on $V(J)$ and let $\varphi \colon V(J)\to V(G)$ be a uniformly chosen random embedding of $J$ into $G$. 
  Define $x \colon V(J)\to [0,\infty)$ by $x_v \coloneqq \ent\bigl(\varphi(v) \mid (\varphi(w) : w \prec v)\bigr)$ and note that, by the chain rule for conditional entropies,
  \begin{equation}
    \label{eq:log-EmbJG}
    \log |\Emb(J, G)| = \ent(\varphi) = \sum_{v \in V(J)} \ent\bigl(\varphi(v) \mid (\varphi(w) : w \prec v)\bigr)  =\sum_{v \in V(J)} x_v.
  \end{equation}
  Further, observe that standard properties of entropy imply that $x_v \le \ent(\varphi (v)) \le \log v_G$ for every $v \in V(J)$ and
  \[
    x_v+x_w \le \ent\bigl(\varphi(v)\bigr)+\ent\bigl(\varphi(w) \mid \varphi(v)\bigr) = \ent\bigl(\varphi(v),\varphi(w)\bigr)\le \log(2e_G)
  \]
  for every $vw \in E(J)$ with $v \prec w$.  In particular, $x$ is a solution to $\mathcal{P}(J; \log v_G, \log(2e_G))$ and thus,
  by \Cref{lem:LP},
  \[
    \sum_{v \in V(J)} x_v \le (v_J-\alpha_J^*) \cdot \log(2e_G) + (2\alpha_J^*-v_J) \cdot \min\{\log v_G, \log(2e_G)\}.
  \]
  Substituting this inequality into~\eqref{eq:log-EmbJG} gives the assertion of the theorem.
\end{proof}

\begin{proof}[Proof of \Cref{prop:max-copies}]
  Set $A \coloneqq \log(n/v_J)$ and $B \coloneqq \log(m/v_J^2)$ and note that $0 \le B \le 2A$.
  Let $x \colon V(J) \to [0, \infty)$ be an optimal solution to $\mathcal{P}(J; A, B)$ and let $G$ be the graph obtained from $J$ by blowing up each $v \in V(J)$ by a factor of $\lfloor e^{x_v} \rfloor$ and replacing the edges of $J$ by complete bipartite graphs.
  By definition,
  \[
    v_G \le \sum_{v \in V(J)} e^{x_v} \le v_J \cdot e^A = n
    \qquad
    \text{and}
    \qquad
    e_G \le \sum_{vw \in E(J)} e^{x_v+x_w} \le e_J \cdot e^B \le m,
  \]
  whereas, by \Cref{lem:LP} and because $x_v \ge 0$ for all $v$,
  \[
    N(J, G) \ge \prod_{v \in V(J)} \lfloor e^{x_v} \rfloor \ge 2^{-v_J} \cdot \exp\left(\sum_{v \in V(J)} x_v\right) = 2^{-v_J} \cdot \left(\frac{m}{v_J^2}\right)^{v_J - \alpha_J^*} \cdot \min\left\{\frac{m}{v_J^2}, \frac{n}{v_J}\right\}^{2\alpha_J^*-v_J},
  \]
  which implies the claimed estimate for a sufficiently small constant $c_J$.
  Finally, assume that $J$ is bipartite.
  By \Cref{lem:LP}, we may assume that $x$ takes the value $B-A$ on all vertices of some vertex cover $U \subseteq J$.
  In particular, the cardinality of the blow-up of $U$ is at most
  \[
    \sum_{v \in U} e^{x_v} \le v_J \cdot e^{\max\{0, B-A\}} = \max\{v_J, m/n\},
  \]
  as claimed.
\end{proof}

Our next lemma bounds the number of embeddings of certain bipartite graphs into a larger graph that contain a specified edge in the image.

\begin{lemma}[{\cite[Lemma~5.14]{HarMouSam22}}]
  \label{lemma:core-edge-bipartite}
  Suppose that $J$ is a nonempty, connected bipartite graph with maximum degree $\Delta$ that admits a bipartition $V(J) = A \cup B$ such that $|A| < |B|$ and $\deg_J a = \Delta$ for every $a \in A$. Then, for every graph $G$ and every $uv \in G$, 
  \[
    |\Emb(J,G; uv)| \le e_J \cdot (\deg_G u + \deg_G v) \cdot (2e_G)^{|A|-1} \cdot \bigl(\min\{e_G, v_G\}\bigr)^{|B|-|A|-1}.
  \]
\end{lemma}

In the sequel, we will only use the following simple corollary of~\cref{lemma:core-edge-bipartite}.
We shall not reproduce the proof of the lemma here; instead, we refer the interested reader to~\cite{HarMouSam22}.

\begin{corollary}
  \label{cor:core-edge-bipartite}  
  For every $J \in \cJ$, all graphs $G$ with $v_G \ge v_J$ and $e_G \ge e_J$, and every $uv \in G$,
  \[
    N(J, G; uv) \le C_J \cdot (\deg_G u + \deg_G v) \cdot \frac{N(J, v_G, e_G)}{e_G \cdot \min\{v_G, e_G\}}
  \]
  for some constant $C_J$ that depends only on $J$.
\end{corollary}
\begin{proof}
  Consider an arbitrary $J \in \cJ$ and let $V(J) = A \cup B$ be its bipartition satisfying $\deg_J a = \Delta$ for every $a \in A$.  Since $|A| \le r \cdot |N(A)| = r|B|$, see the definition of $r$, and $r < 1$ by \Cref{lemma:rH-bounds}, we may bound $N(J, G; uv)$ from above using \Cref{lemma:core-edge-bipartite}.  Further, since $\alpha_J^* = |B|$, by \cref{lem:unique-frac-ind-set}, we have $|A| = v_J - \alpha_J^*$ and $|B| - |A| = 2\alpha^*_J - v_J$.  The claimed bound now follows by combining \Cref{lemma:core-edge-bipartite} and \Cref{prop:max-copies}.
\end{proof}



\section{Planted subgraphs and rate functions}
\label{sec:planted-graphs-rate}

The goal of this section is threefold. 
First, we shall establish the following asymptotic formula for the difference $\Ex_G[X] - \Ex[X]$ that is valid for all $G \subseteq K_n$ and all $p \ll 1$.

\begin{proposition}
  \label{cor:difference-in-expectations}
  For all $p \le 1/2$ and $G \subseteq K_n$,
  \begin{equation}
    \label{eq:difference-between-expectations}
    \Ex_G[X]-\Ex[X] =\bigl(1+\Theta(n^{-1}+p)\bigr) \cdot \sum_{\emptyset \neq J\subseteq H} \frac{|\Emb(J,H)|}{|\Aut(H)|} \cdot N(J,G)\cdot n^{v_H-v_J} p^{e_H-e_J},
  \end{equation}    
  where the sum is over all unlabelled subgraphs of $H$ without isolated vertices.
\end{proposition}

Second, we identify the important summands in the right-hand side of~\eqref{eq:difference-between-expectations} whenever the planted graph $G$ has $O(f \log(1/p))$ edges.

\begin{proposition}
  \label{prop:difference-in-expectations-important-J}
  Suppose that $p \ll 1$ and let $G \subseteq K_n$ be a graph with $O(f \log(1/p))$ edges.
  Suppose further that either $\eps_H = \infty$ or $\eps_H < \infty$ and $p \gg n^{-1/\Delta - \eps_H} (\log(1/p))^K$ for some large constant $K$ that depends only on $H$.
  Then
  \begin{equation*}
    \Ex_G[X]-\Ex[X] \leq \sum_{J\in \cJ} \frac{|\Emb(J,H)|}{|\Aut(H)|} \cdot N(J,G)\cdot n^{v_H-v_J}p^{e_H-e_J}+o(n^{v_H}p^{e_H}).
  \end{equation*}
  Furthermore, if $p \le n^{-1/\Delta-\Omega(1)}$, then
  \[
    \Ex_G[X]-\Ex[X] \leq \sum_{J\in \cJ^*} \frac{|\Emb(J,H)|}{|\Aut(H)|} \cdot N(J,G)\cdot n^{v_H-v_J}p^{e_H-e_J}+o(n^{v_H}p^{e_H}).
  \]  
\end{proposition}

Third, we will determine the order of magnitude of the rate functions $\psi_H$ and $\psi^*_H$ in the entire range $p_H \ll p \ll 1$, except for the case where $H$ is not stable and $p = p_H (\log n)^{O(1)}$.

\begin{proposition}
  \label{lem:psi--psi*-lower}
  For every $\delta > 0$, there exist $c_\delta,C_\delta > 0$ such that, for all $\omega \cdot p_H\ll p\ll 1$,
  \[
    c_\delta f \log(1/p)\le \psi_H(\delta) \le C_\delta f \log (1/p).
  \]
    Further, $\psi_{H}^*(\delta)\leq C_{\delta}f\log(1/p)$ and there is a constant $K$ for which 
    \[
      \psi^*_H(\delta) \geq c_{\delta} f\cdot
      \begin{cases}
        \log (1/p) &\text{if $H$ is stable},\\
        \min\{\log (p/p_H), \log(1/p)\} & p \geq p_H (\log n)^K.
  \end{cases}
  \]
\end{proposition}

\subsection{Planted subgraphs}
\label{sec:planted-subgraphs}

We first prove a variant of \Cref{cor:difference-in-expectations} that is easier to prove but more difficult to apply.  We write $J \sqsubseteq H$ to indicate that $J$ is a \emph{spanning} subgraph of~$H$.

\begin{lemma}
  \label{lem:difference-in-expectations}
  Suppose that $p \le 1/2$ and that $G \sqsubseteq K_n$. Letting $X'\coloneqq |\Emb (H,G_{n,p})|$, we have
  \[
    \Ex_G[X']-\Ex[X'] =(1+\Theta(p))\sum_{\emptyset \neq J\sqsubseteq H} |\Emb(J,G)| \cdot p^{e_H-e_J}.
  \]
\end{lemma}

\begin{proof}
  Write $X'=\sum_{\varphi \in \Emb(H,K_n)}Y_\varphi$, where $Y_\varphi$ is the indicator for the event $\varphi\in \Emb(H,G_{n,p})$.
  For every $\varphi\in \Emb(H,K_n)$, we write $\varphi(H)$ for the image of $H$ under $\varphi$, so that $\varphi \in \Emb(H, G_{n,p})$ if and only if $\varphi(H) \subseteq G_{n,p}$.
  Suppose that $G$ is a spanning, labelled subgraph of $K_n$.
  By linearity of expectation and conditional expectation, we have
  \[
    \Ex_G[X']-\Ex[X']
    = \sum_{\varphi\in \Emb(H,K_n)}p^{e_{H}-e_{\varphi(H)\cap G}}-p^{e_H}
    = (1+\Theta(p)) \sum_{\substack{\varphi\in \Emb(H,K_n)\\ \varphi(H) \cap G \neq \emptyset}}  p^{e_H - e_{\varphi(H) \cap G}}.
  \]
  We now claim that, for every $\varphi \in \Emb(H,K_n)$ such that $\varphi(H) \cap G \neq \emptyset$,
  \begin{equation}
    \label{eq:Mobius-asymptotic-equivalence}
    p^{e_H - e_{\varphi(H) \cap G}} = (1+\Theta(p)) \sum_{\substack{\emptyset \neq J\sqsubseteq H \\ \varphi(J) \subseteq G}} p^{e_H-e_J},
  \end{equation}
  which implies the assertion of the lemma, as for every $J \sqsubseteq H$ we have
  \[
    \{\varphi\in \Emb(H,K_n) : \varphi(J)\subseteq G\} = \Emb(J,G).
  \]
  Indeed, it is clear that the left-hand side of~\eqref{eq:Mobius-asymptotic-equivalence} is at most as large as the right-hand side (consider the largest $J \sqsubseteq H$ such that $\varphi(J) \subseteq G$).
  On the other hand, it is not hard to see that, for every $\varphi \in \Emb(H,K_n)$,
  \[
    \sum_{\substack{J \sqsubseteq H \\ \varphi(J) \subseteq G}} p^{e_H-e_J} = p^{e_H-e_{\varphi(H) \cap G}} \cdot (1+p)^{e_{\varphi(H)\cap G}} \le (1+O(p)) \cdot p^{e_{H}-e_{\varphi(H)\cap G}}.
    \qedhere
  \]
\end{proof}

\begin{proof}[Proof of \Cref{cor:difference-in-expectations}]
  Since the left-hand side and the right-hand side of \eqref{eq:difference-between-expectations} do not change if we add to $G$ isolated vertices, we may assume that $G\sqsubseteq K_n$.
  For any graph $\emptyset \neq J \sqsubseteq H$, let $J^*$ be the subgraph of $J$ formed by removing all isolated vertices.  Then, 
  \begin{equation}\label{eq:embedding-of-graph-with-isolated-vertices}
    |\Emb(J, G)|=|\Emb(J^*,G)|\cdot (n-v_{J^*})_{v_H-v_{J^*}}=(1+O(n^{-1}))\cdot |\Emb(J^*,G)|\cdot n^{v_H-v_{J^*}},
  \end{equation}
  where $(n)_k=\prod_{i=0}^{k-1} (n-i)$ is the $k$-th falling factorial.
  Write $X' \coloneqq |\Emb(H, G_{n,p})|$.
  \cref{lem:difference-in-expectations} and \eqref{eq:embedding-of-graph-with-isolated-vertices} imply that
  \[
    \begin{split}
      \Ex_G[X']-\Ex[X']
      &= (1+O(n^{-1}+p)) \cdot \sum_{\emptyset \neq J\sqsubseteq H} |\Emb(J^*, G)| \cdot n^{v_H-v_{J^*}} p^{e_H-e_{J^*}}\\
      &= (1+O(n^{-1}+p)) \cdot \sum_{\emptyset \neq J^* \subseteq H} N(J^*,H) \cdot |\Emb(J^*,G)| \cdot n^{v_H-v_{J^*}} p^{e_H-e_{J^*}},
    \end{split}
  \]
  where the second sum ranges over all unlabelled subgraphs of $H$ without isolated vertices.
  The assertion of the proposition follows, as $|\Emb(K,L)| = |\Aut(K)|\cdot N(K,L)$ for all graphs $K,L$.
\end{proof}

The following lemma plays the key role in the proof of \Cref{prop:difference-in-expectations-important-J}.

\begin{lemma}
  \label{lem:bound-on-NJG}
  The following holds for every nonempty $J \subseteq H$ without isolated vertices, all $x \in (0, \infty)$, and every $G \subseteq K_n$ with at most $xf$ edges.
  Writing
  \[
    \pi(x) \coloneqq \max\{(2x)^{v_J-\alpha_J^*},(2x)^{\alpha_J^*}\} \le \max\{2x, (2x)^{v_J}\},
  \]
  we have:
  \begin{enumerate}[label=(\roman*)]
  \item
    \label{item:NJG-cJ}
    If $J \in \cJ$, then, for all $p$,
    \[
      \frac{N(J,G)}{n^{v_J}p^{e_J}} \le \pi(x).
    \]
  \item
    \label{item:NJG-cJ-not-star-sparse}
    If $J \in \cJ \setminus \cJ^*$, then, for all $p < n^{-1/\Delta}$,
    \[
      \frac{N(J,G)}{n^{v_J}p^{e_J}} \le \pi(x) \cdot (np^\Delta)^{\gamma_H},
    \]
    where $\gamma_H$ is a positive constant that depends only on $H$.
  \item
    \label{item:NJG-not-cJ-eps-infty}
    If $J \notin \cJ$ and $\eps_H = \infty$, then, for all $p$,
    \[
      \frac{N(J,G)}{n^{v_J}p^{e_J}} \leq \pi(x) \cdot \max\{p^{1/2}, n^{-1/(2\Delta)}\}.
    \]
  \item
    \label{item:NJG-not-cJ-eps-finite}
    If $J \notin \cJ$ and $\eps_H<\infty$, then, for every $L \ge 1$ and all $p \ge Ln^{-1/\Delta-\eps_H}$,
    \[
      \frac{N(J,G)}{n^{v_J}p^{e_J}}\leq \pi(x) \cdot \max\{p^{1/2}, n^{-1/(2\Delta)}, L^{-\beta_H}\},
    \]
    where $\beta_H$ is a positive constant that depends only on $H$.
  \end{enumerate}
\end{lemma}

\begin{proof}
  Fix a nonempty $J \subseteq H$ without isolated vertices and a positive real $x$ and let $G \subseteq K_n$ be a graph with at most $x f$ edges.
  \Cref{thm:max-copies} asserts that 
  \begin{equation}
    \label{eq:embedding_lemma}
    N(J,G)\leq |\Emb(J,G)|\leq (2e_G)^{v_J-\alpha_J^*}\cdot \min\{2e_G,v_G\}^{2\alpha_J^*-v_J}\leq \begin{cases}
      (2e_G)^{v_J-\alpha_J^*} \cdot n^{2\alpha_J^*-v_J}\\
      (2e_G)^{\alpha_J^*}
    \end{cases}.
  \end{equation}
  As $H$ is irregular, \Cref{lemma:eJ-alphaJ-clique} implies that $e_J \le \Delta \cdot (v_J-\alpha_J^*)$ for every nonempty $J \subseteq H$ and that equality holds if and only if $J \in \cJ$.  Further, since $e_J,v_J$, and $2\alpha_J^*$ are integers, see \Cref{lemma:fractional-duality}, we actually have
  \begin{equation}
    \label{eq:eJ-Delta-alpha-v-ineq}
    e_J\leq \Delta \cdot (v_J-\alpha_J^*)-1/2 \cdot \1[J \notin \cJ].
  \end{equation}
  We will consider two cases, depending on whether or not $p \ge n^{-1/\Delta}$.

  \noindent
  \textit{Case 1:}  $p\ge n^{-1/\Delta}$. \\ 
  Our assumption on $G$ translates to $e_G\leq x\cdot n^2p^{\Delta}$.
  Then, by \eqref{eq:embedding_lemma} and \eqref{eq:eJ-Delta-alpha-v-ineq},
  \[
    \begin{split}
      \frac{N(J,G)}{n^{v_J}p^{e_J}}
      & \le \frac{(2x\cdot n^2p^\Delta)^{v_J-\alpha^*_J} \cdot n^{2\alpha_J^*-v_J}}{n^{v_J}p^{e_J}} = (2x)^{v_J - \alpha_J^*} \cdot p^{\Delta(v_J-\alpha^*_J)-e_J}\leq  \pi(x) \cdot p^{\1[J \notin \cJ]/2},
    \end{split}
  \]
  as claimed.
  
  \noindent
  \textit{Case 2:} $p < n^{-1/\Delta}$. \\
  Our assumption on $G$ translates to $e_G\leq x\cdot n(np^{\Delta})^r$.
  It follows from \eqref{eq:embedding_lemma} that 
  \begin{equation}
    \label{eq:base-bound}
    \frac{N(J,G)}{n^{v_J}p^{e_J}}
    \leq \frac{(2x\cdot n(np^{\Delta})^r)^{\alpha_J^*}}{n^{v_J} p^{e_J}} \eqqcolon (2x)^{\alpha_J^*} \cdot h_J(n,p).
  \end{equation}
  In order to estimate $h = h_J(n,p)$, we consider two cases, depending on whether or not $J \in \cJ$.

  \noindent
  \textit{Subcase 2a:} $J \in \cJ$. \\
  Let $A \subseteq V(H)$ be the set of vertices of degree $\Delta$ such that $J = H[A, N(A)]$ and note that $\alpha_J^* \ge \alpha_J \ge |N(A)|$.
  Further, by the definition of $r$, we have $r|N(A)| \ge |A|$ and equality holds if and only if $J \in \cJ^*$.
  It follows that, for some positive constant $\gamma_H$,
  \[
    (1+r)\alpha_J^* \ge (1+r) |N(A)| \ge v_J + \gamma_H \cdot \1[J \notin \cJ^*].
  \]
  Consequently, since $np^\Delta \le 1$, and using \eqref{eq:eJ-Delta-alpha-v-ineq},
  \[
    h = (np^\Delta)^{(1+r)\alpha_J^*-v_J} \cdot p^{\Delta(v_J-\alpha_J^*) - e_J} \le (np^\Delta)^{\gamma_H \cdot \1[J \notin \cJ^*]}.
  \]
  Substituting this inequality into \eqref{eq:base-bound} yields the desired estimate.

  \noindent
  \textit{Subcase 2b:} $J \notin \cJ$. \\
  Let $\xi \ge 0$ be the number satisfying $p = n^{-1/\Delta - \xi}$ and note that
  \[
    h = n^{(1+r)\alpha_J^* - v_J} \cdot p^{r\Delta\alpha_J^* - e_J} = n^{e_J/\Delta + \alpha_J^* - v_J + \xi(e_J - r\Delta\alpha_J^*)}.
  \]
  Observe now that~\eqref{eq:eJ-Delta-alpha-v-ineq} implies that $e_J/\Delta + \alpha_J^* - v_J \le -1/(2\Delta)$.
  Consequently, $h \le n^{-1/(2\Delta)}$ whenever $e_J \le r\Delta\alpha_J^*$.
  Substituting this inequality into~\eqref{eq:base-bound} yields the claimed estimate, unless $e_J > r\Delta\alpha_J^*$.
  Assume now that $e_J > r\Delta\alpha_J^*$, which implies that $\eps_H < \infty$.
  Since the definition of $\eps_H$ implies that
  \[
    e_J/\Delta + \alpha_J^* - v_J + \eps_H \cdot (e_J - r\Delta\alpha_J^*) \le 0,
  \]
  we have
  \[
    h \le n^{(\xi - \eps_H) \cdot (e_J - r\Delta\alpha_J^*)} = \left(\frac{n^{-1/\Delta-\eps_H}}{p}\right)^{e_J - r\Delta\alpha_J^*}.
  \]
  Substituting this inequality into~\eqref{eq:base-bound} gives the assertion of~\ref{item:NJG-not-cJ-eps-finite} with
  \[
    \beta_H \coloneqq \min\{e_J - r\Delta\alpha_J^* : J \subseteq H \wedge e_J > r\Delta\alpha_J^*\} > 0
  \]
  and completes the proof of the lemma.
\end{proof}

\begin{proof}[Proof of \Cref{prop:difference-in-expectations-important-J}]
  Suppose that $p \ll 1$ and let $G$ be a graph with $O(f\log(1/p))$ edges.
  In view of \Cref{cor:difference-in-expectations}, in order to prove the first assertion of the proposition, it suffices to show that $N(J,G) \ll n^{v_J}p^{e_J}$ for all $J \notin \cJ$ and $(n^{-1} + p) \cdot N(J,G) \ll n^{v_J}p^{e_J}$ for all $J \in \cJ$.
  These two asymptotic estimates are easy consequences of \Cref{lem:bound-on-NJG}.
  Indeed, the former statement follows from items \ref{item:NJG-not-cJ-eps-infty} and \ref{item:NJG-not-cJ-eps-finite}, provided that $K$ is sufficiently large, while the latter statement follows from item \ref{item:NJG-cJ}.
  Finally, assume that $p \le n^{-1/\Delta-\Omega(1)}$ and suppose that $J \in \cJ \setminus \cJ^*$.  Since item~\ref{item:NJG-cJ-not-star-sparse} in \Cref{lem:bound-on-NJG} yields $N(J,G) \ll n^{v_J}p^{e_J}$, the second assertion of the proposition follows.
\end{proof}

\subsection{Estimating the rate functions}
\label{sec:rate-functions}

We now turn to the proof of \Cref{lem:psi--psi*-lower}.
We start with the (easier) upper bounds on the rate functions $\psi_H$ and $\psi_H^*$.

\begin{lemma}
  \label{lem:weak_estimation_for_Phi}
  For every $\delta > 0$, there exists a constant $D$ such that, for all $\omega \cdot p_H \ll p \ll 1$,
  \[
    \max\{\psi_H(\delta), \psi_H^*(\delta)\} \le Df\log(1/p).
  \]
\end{lemma}
\begin{proof}
  Pick a large constant $D$ and an arbitrary $J \in \cJ^*$.
  By \Cref{prop:max-copies}, we may find a graph $G$ with at most $n$ vertices and $Df$ edges that is a blow-up of $J$, satisfies
  \[
    N(J, G)
    \ge c_J \cdot (Df)^{v_J - \alpha_J^*} \cdot \min\{Df, n\}^{2\alpha_J^* - v_J} \ge D c_J \cdot f^{v_J-\alpha_J^*} \cdot \min\{f, n\}^{2\alpha_J^* - v_J},
  \]
  and has a vertex cover of cardinality at most
  \[
    \max\{v_J, Df/n\} \le v_J \cdot \max\{1, e_G/n\} \le \frac{e_G}{\omega},
  \]
  where the last inequality follows as $e_G \ge f \ge p/p_H$, see \cref{fact:f-p-pH}.
  Since $\delta(G) \ge \delta(J) \ge d_H$, the graph $G$ is a hub-core.
  We claim that
  \[
    N \coloneqq f^{v_J-\alpha_J^*} \cdot \min\{f, n\}^{2\alpha_J^* - v_J} = n^{v_J} p^{e_J}.
  \]
  Indeed, if $p \ge n^{-1/\Delta}$ and thus $f = n^2p^\Delta \ge n$, then
  \[
    N = (n^2p^\Delta)^{v_J - \alpha_J^*} \cdot n^{2\alpha_J^* - v_J} = n^{v_J} p^{\Delta (v_J - \alpha_J^*)} = n^{v_J} p^{e_J},
  \]
  where the last equality follows from \Cref{lemma:eJ-alphaJ-clique}.
  If $p < n^{-1/\Delta}$ and thus $f = n (np^\Delta)^r$, then
  \[
    N = f^{\alpha_J^*} = (n^{1+r} p^{\Delta r})^{\alpha_J^*} = n^{v_J} p^{e_J},
  \]
  as every graph $J \in \cJ^*$ satisfies $v_J = (1+r) \cdot \alpha_J^*$ and $e_J = \Delta(v_J-\alpha_J^*)$.
  By \Cref{cor:difference-in-expectations},
  \[
    \Ex_G[X] - \Ex[X] \ge \frac{N(J,G) \cdot  n^{v_H-v_J} p^{e_H-e_J}}{2 |\Aut(H)|} \ge \frac{Dc_J \cdot n^{v_H}p^{e_H}}{2|\Aut(H)|} \ge \delta \Ex[X],
  \]
  provided that $D$ is sufficiently large.
  We may conclude that $\psi_H(\delta) \le e_G\log(1/p) \le Df\log(1/p)$ and, as $G$ is a hub-core, also that $\psi_H^*(\delta) \le Df\log(1/p)$.
\end{proof}

Next, the lower bound on $\psi_H$ is a simple consequence of \Cref{cor:difference-in-expectations,lem:bound-on-NJG}.

\begin{lemma}
  \label{lem:psi-lower}
  For every $\delta > 0$, there exists a constant $c_\delta > 0$ such that, for all $p_H\ll p\ll 1$,
  \[
    \psi_H(\delta) \ge c_\delta f \log(1/p).
  \]
\end{lemma}
\begin{proof}
  We need to show that every $G \subseteq K_n$ with at most $c f$ edges satisfies $\Ex_G[X] < (1+\delta)\Ex[X]$, provided that $c > 0$ is small enough.  Our lower-bound assumption on $p$ and \Cref{lem:bound-on-NJG} imply that, for every nonempty $J\subseteq H$ without isolated vertices, we have $N(J,G) \leq 2c \cdot n^{v_J}p^{e_J}$.
  Consequently, \Cref{cor:difference-in-expectations} implies that there is a constant $K$ that depends only on $H$ such that
  \[
    \frac{\Ex_G[X]-\Ex[X]}{\Ex[X]} \leq K \cdot \sum_{\emptyset \neq J\subseteq H} \frac{N(J,G)}{n^{v_J}p^{e_J}}\leq 2^{e_H} K c < \delta,
  \]
  provided that $c$ is sufficiently small.  This completes the proof of the lemma.
\end{proof}

Finally, we turn to the (much more subtle) proof of the lower bound on $\psi_H^*$.
The following technical lemma captures the essence of the assumption that $\delta(G^*) \ge d_H$ in the definition of hub-cores and the relation of $d_H$ to the stability property of $H$.

\begin{lemma}
  \label{lem:f-npd}
  There exists a positive constant $\beta$ such that, for every $L \ge 1$ and all $Lp_H \le p \le n^{-1/\Delta}$, we have
  \[
    \frac{f}{np^{d_H}} \geq \beta \cdot
    \begin{cases}
      n^\beta & \text{if $H$ is stable}, \\
      L^\beta & \text{otherwise}.
    \end{cases}
  \]
\end{lemma}
\begin{proof}
  Choose an arbitrary $J \in \cJ^*$ with minimum degree $d_H$ and let $J'$ be the graph obtained from $J$ by removing an arbitrary vertex of degree $d_H$.  Clearly, $e_{J'} = e_J - d_H$, $v_{J'} = v_J -1$, and $\alpha_{J'}^* \ge \alpha_{J}^* - 1$.
  Since $f \le n$ by our assumption on $p$, \Cref{prop:max-copies} gives
  \begin{equation}
    \label{eq:NJ'nf-lower}
    N(J', n, f) \ge c f^{\alpha_{J'}^*} \ge cf^{\alpha_J^*-1} \ge \frac{c' n^{v_J}p^{e_J}}{f},
  \end{equation}
  for some positive constants $c$ and $c'$ that depend only on $H$.
  Further, let $J''$ be the graph obtained from $J'$ by removing all isolated vertices.
  Since $e_J - e_{J''} = e_J - e_{J'} = d_H \in \{1, \dotsc, \Delta-1\}$ and the number of edges of each graph in $\cJ$ is divisible by $\Delta$, we have $J'' \notin \cJ$.
  Since $N(J', G) \le N(J'', G) \cdot v_G$ for every graph $G$, \Cref{lem:bound-on-NJG} implies that
  \begin{equation}
    \label{eq:NJ'nf-upper}
    np^{d_H} \cdot \frac{N(J',n,f)}{n^{v_J}p^{e_J}} = \frac{N(J',n,f)}{n^{v_{J'}}p^{e_{J'}}} \le \frac{N(J'',n,f)}{n^{v_{J''}}p^{e_{J''}}} \le
    \begin{cases}
      n^{-\beta}, & \text{$H$ is stable}, \\
      L^{-\beta}, & \text{otherwise},
    \end{cases}
  \end{equation}
  where $\beta > 0$ depends only on $H$.
  Combining~\eqref{eq:NJ'nf-lower} and~\eqref{eq:NJ'nf-upper} yields the desired inequality.
\end{proof}

The lower bound on $\psi_H^*$ asserted in~\Cref{lem:psi--psi*-lower} will be a straightforward consequence of the following technical lemma and the fact that every seed must have $\Omega(f)$ edges, which follows from~\Cref{lem:psi-lower}.
The stronger, but more technical, statement will be necessary for controlling the union bound over all hub-cores in the proofs of \cref{thm:main-dense-ish}, \cref{thm:main-near-appearance-threshold}~\ref{item:main-near-appearance-threshold-UB}, and \cref{thm:non-clean}.
(Crucially, the explicit lower bound on the difference $e_G \log(1/p) - v_G \log(en/g)$ will allow us to show that the error term in the upper bound on the number of hub-cores proved in~\Cref{lem:core-count} below is negligible.)

\begin{lemma}
  \label{lem:weak-LB-psi-*}
  For every $\delta > 0$, there exists a positive constant $c$ such that the following holds for every $G \in \mathcal{L}$ satisfying $\Ex_G[X] \ge (1+\delta)\Ex[X]$:
  \begin{enumerate}[label=(\roman*)]
  \item
    \label{item:weak-LB-psi-*-non-stable}
    If $Lp_H \le p \ll 1$ for some $L \ge (\log n)^{1/c}$ and $e_G \le f(\log n)^2$, then
    \[
      e_G\log(1/p)-v_G\log(en/g)\geq ce_G \min\{\log L, \log(1/p)\}.
    \]
  \item
    \label{item:weak-LB-psi-*-stable}
    If $H$ is stable and $p_H \ll p \ll 1$, then
    \[
      e_G\log(1/p)-v_G\log(en/g)\geq ce_G\log(1/p).
    \]
  \end{enumerate}
\end{lemma}

\begin{proof}
  Suppose that $G\in \mathcal{L}$ and write
  \[
    \tau(G) \coloneqq e_G \log(1/p) - v_G \log(en/g).
  \]
  Fix a small constant $x > 0$.  We consider two cases, depending on whether or not $p \ge n^{-1/(\Delta-x)}$.
  
  \noindent
  \textit{Case 1:}  $n^{-1/(\Delta-x)} \le p\ll 1$. \\
  Since clearly $v_G \le 2e_G$, we have
  \[
    \tau(G) \ge e_G \log(1/p) - 2e_G \log(en/g) = e_G \log\left(\frac{g^2}{e^2n^2p}\right).
  \]
  Suppose first that $p\geq n^{-1/\Delta}$, which implies that $f=n^2p^{\Delta}\geq n$ and $g=n$.
  As $p \ll 1$, we have
  \[
    \log\left(\frac{g^2}{e^2n^2p}\right) = \log\left(\frac{1}{e^2p}\right) \ge \frac{\log(1/p)}{2}.
  \]
  Suppose now that $n^{-1/(\Delta-x)} \le p\leq n^{-1/\Delta}$, which implies that $f = n(np^\Delta)^r \le n$ and $g = f$.
  By our lower-bound assumption on $p$, we have $g = f \ge n p^{xr}$ and thus
  \[
    \log\left(\frac{g^2}{e^2n^2p}\right) = \log\left(\frac{p^{2xr}}{e^2p}\right) \ge \frac{\log(1/p)}{2},
  \]
  provided that $x$ is sufficiently small.

  \noindent
  \textit{Case 2:} $p \leq n^{-1/(\Delta-x)}$. \\
  Let $D_{G}$ be a smallest vertex cover in $G$ and let $\nu_G \coloneqq |D_{G}|$.
  Since every vertex of $G$ outside of $D_G$ has at least $\delta(G) \ge d_H$ neighbours in $D_G$, we have $e_G \ge d_H \cdot (v_G - \nu_G)$ or, equivalently,
  \[
    v_G \le e_G/d_H + \nu_G.
  \]
  Since $g = f$, as $p \le n^{-1/\Delta}$, it follows that
  \begin{equation}
    \label{eq:tau-G-lower-i}
    \tau(G) \ge e_G \log(1/p) - (e_G/d_H + \nu_G) \log (en/g) = \frac{e_G}{d_H} \log\left(\frac{f}{enp^{d_H}}\right) - \nu_G \log(en/f)
  \end{equation}
  while the assumption that $G \in \cL$ guarantees that
  \begin{equation}
    \label{eq:dG-upper}
    \nu_G \le \min\bigl\{e_G / \omega, \bigl(\log \log (1/p)\bigr)^3 \cdot \max\{1, (e_G/f)^{2v_H}\}\bigr\}.
  \end{equation}
  Suppose now that $p \ge Lp_H$ for some $L \gg 1$.  \Cref{lem:f-npd} supplies a constant $\beta > 0$ such that
  \begin{equation}
    \label{eq:fenpdH-lower}
    \log \left(\frac{f}{enp^{d_H}}\right) \ge
    \beta \cdot
    \begin{cases}
      \log n & \text{if $H$ is stable,}\\
      \log L & \text{otherwise}.
    \end{cases}
  \end{equation}
  Further, \Cref{lem:psi-lower} implies that $e_G \ge c_\delta f$ for some $c_{\delta} > 0$ that depends only on $\delta$ and $H$.
  Substituting \eqref{eq:dG-upper} and \eqref{eq:fenpdH-lower} into~\eqref{eq:tau-G-lower-i} gives the following:
  \begin{enumerate}[label=(\roman*)]
  \item
    If $e_G \le f (\log n)^2$, then
    \[
      \tau(G) \ge \frac{e_G}{d_H} \cdot \beta \log L - (\log n)^{4v_H+2},
    \]
    which gives the desired inequality, as $e_G \ge c_\delta f \ge c_\delta L \gg (\log n)^{4v_H+2}$ provided that $L \ge (\log n)^{1/c}$ for a sufficiently small constant $c > 0$.
  \item
    If $H$ is stable, then
    \[
      \tau(G) \ge \frac{e_G}{d_H} \cdot \beta \log n - \frac{e_G}{\omega} \cdot \log n \ge \frac{e_G}{2d_H} \cdot \beta \log n,
    \]
    which again gives the desired inequality.\qedhere
  \end{enumerate}
\end{proof}

\begin{proof}[Proof of \Cref{lem:psi--psi*-lower}]
  The upper bounds $\psi_H(\delta), \psi_H^*(\delta)\le C_\delta f\log(1/p)$ and the lower bound $\psi_H(\delta)\ge c_\delta f\log(1/p)$ follow immediately from \Cref{lem:weak_estimation_for_Phi,lem:psi-lower}, respectively;
  we only need to prove the asserted lower bound on $\psi_H^*(\delta)$.
  Let $G \in \mathcal{L} \cap \mathcal{S}_\delta$ be a graph achieving the minimum in the definition of $\psi_H^*(\delta)$.
  \Cref{lem:psi-lower} implies that $e_G \ge c_\delta' f$ for some positive constant $c_\delta'$.
  If $H$ is stable, \Cref{lem:weak-LB-psi-*}~\ref{item:weak-LB-psi-*-stable} supplies a positive constant $c$ such that
  \[
    \psi_H^*(\delta) = e_G\log(1/p)-v_G\log(en/g) \ge c e_G\log(1/p) \ge c c_\delta' f\log(1/p),
  \]
  as required.
  Suppose now that $H$ is not stable, choose a sufficiently large constant $K$, assume that $p \ge p_H (\log n)^K$, and let $L \coloneqq p/p_H$.
  In this case, \Cref{lem:weak-LB-psi-*}~\ref{item:weak-LB-psi-*-non-stable} supplies a positive constant $c$ such that
  \[
    \psi_H^*(\delta) \ge c e_G\min\{\log L,\log(1/p)\} \ge c c_\delta' f\min\{\log(p/p_H), \log(1/p)\},
  \]
  provided that $K \ge 1/c$, as desired.
\end{proof}

\section{Upper bounds}
\label{sec:upper-bounds}

In this section, we prove the upper bound in \cref{thm:main-dense-ish}.
As we have mentioned, our proof follows the general approach of Harel, Mousset and Samotij~\cite{HarMouSam22}, with the main difference being the more combinatorial construction of minimal seeds, which we term hub-cores.
For the sake of completeness, we include proofs of all the probabilistic lemmas borrowed from~\cite{HarMouSam22}.
For the sake of brevity, we shall write $X$ in place of $X_H$.
We start with a key definition.

\begin{definition}
  We say that a graph $G\subseteq K_n$ is a \emph{$t$-seed} if $\Ex_G[X] \geq \Ex[X]+t$.  Further, we let $\cS_{t,m}$ be the family of all $t$-seeds with at most $m$ edges and let $\langle \cS_{t,m} \rangle$ denote the family of all subgraphs of $K_n$ that contain some graph in $\cS_{t,m}$ as a subgraph.
\end{definition}

The first key result of this section is that the upper tail event $\UT_{H,\delta}$ is dominated by the appearance of a seed with $O(\psi_H(\delta))$ edges.

\begin{proposition}
  \label{cor:exists-a-seed-condition-on-UT}
  Suppose that $0 < \eps \le \delta$ and that $\psi_H(\delta+\varepsilon) \ge -\log(\varepsilon p^{e_H})$ and $\psi_H(\delta+\varepsilon) \gg 1$.
  There exists a $D=D(\delta,\varepsilon)$ such that, letting $t \coloneqq (\delta-\varepsilon)\Ex[X]$ and $m \coloneqq \lceil D \psi_H(\delta+\varepsilon) \rceil$, we have
  \[
    \Pr\bigl(G_{n,p} \in \langle\cS_{t,m}\rangle \mid X\geq (1+\delta)\Ex[X]\bigr) =1-o(1).
  \]  
\end{proposition}

The heart of the proof of~\cref{cor:exists-a-seed-condition-on-UT} is a conditional version of the elegant high-moment argument due to Janson, Oleszkiewicz, and Ruciński~\cite{JanOleRuc04}, which yields the following lemma.

\begin{lemma}[\cite{HarMouSam22}]
  \label{lem:stability}
  Let $\delta,t>0$ and let $m$ be an integer. Then,
  \[
    \Pr\big(X \geq (1+\delta)\Ex[X]\land G_{n,p} \notin \langle\cS_{t,m}\rangle\big) \leq \left(\frac{\Ex[X]+t}{(1+\delta)\Ex[X]}\right)^{m/e_H}.
  \]
\end{lemma}

\begin{proof}
  Let $Z$ be the indicator for the event that $G_{n,p} \notin \langle\cS_{t,m}\rangle$.
  From the definition of $Z$ we get that $XZ\geq 0$ and $Z^k=Z$. By Markov's inequality, for every positive integer $k$, we have
  \begin{equation}\label{eq:stability2}
    \Pr\big(X \geq (1+\delta)\Ex[X]\land G_{n,p} \notin \langle\cS_{t,m}\rangle\big)
    = \Pr\big(X Z \geq (1+\delta)\Ex[X]\big) \leq 
    \frac{\Ex[X^k Z]}{\big((1+\delta)\Ex[X]\big)^k}.
  \end{equation}
  For every $U\subseteq E(K_n)$, let $Y_U$ be the indicator for the event that $U$ is contained in $G_{n,p}$. Write $X = \sum_S Y_S$,
  where the sum ranges over all copies of $H$ in $K_n$. Then, for every $k$,
  \[
    \Ex[X^k Z]
    = \sum_{S_1,\dotsc,S_k}\Ex[Y_{S_1}\dotsb Y_{S_k} \cdot Z]
    = \sum_{S_1,\dotsc,S_{k-1}}
    \Ex[Y_{S_1}\dotsb Y_{S_{k-1}}\cdot Z]
    \cdot \Ex[X
    \mid Y_{S_1}\dotsb Y_{S_{k-1}}\cdot Z =1],
  \]
  where the latter sum ranges only over sequences $S_1,\dotsc,S_{k-1}$
  for which $Y_{S_1}\dotsb Y_{S_{k-1}}\cdot Z=1$ holds with positive probability.
  Note that if $Y_{S_1}\dotsb Y_{S_{k-1}}\cdot Z=1$, then $S_1\cup \dotsb \cup S_{k-1}\subseteq G_{n,p}$ and $G_{n,p} \notin \langle\cS_{t,m}\rangle$; in particular, $S_1\cup \dotsb \cup S_{k-1} \notin \cS_{t,m}$.
  Fix some such sequence $S_1,\dotsc, S_{k-1}$ and note that
  \[
    \Ex[X \mid Y_{S_1} \dotsb Y_{S_{k-1}} \cdot Z = 1] = \Ex[X \mid Y_{S_1 \cup \dotsb \cup S_{k-1}} \cdot Z = 1] = \Ex_{S_1 \cup \dotsb \cup S_{k-1}}[X \mid Z = 1],
  \]
  where $\Ex_{S_1\cup\dotsb \cup S_{k-1}}$ is the expectation operator in the conditional probability space of $G_{n,p}$ conditioned on the event $\{S_1\cup \dotsb \cup S_{k-1}\subseteq G_{n,p}\}$, which is a product space.
  Since $X$ and $Z$ are increasing and decreasing (respectively) functions defined on this space, Harris's inequality gives
  \[
    \Ex_{S_1 \cup \dotsb \cup S_{k-1}}[X \mid Z=1] \le \Ex_{S_1 \cup \dotsb \cup S_{k-1}}[X].
  \]
  Finally, if $k \le \lceil m/e_H \rceil$, then the facts that $S_1\cup \dotsb \cup S_{k-1}\notin \cS_{t,m}$ and $e(S_1 \cup \dotsb \cup S_{k-1}) \le (k-1)e_H < m$ imply that $\Ex_{S_1 \cup \dotsb \cup S_{k-1}}[X] < \Ex[X] +t$.
  It follows that, for $k \leq \lceil m / e_H \rceil$,
  \[
    \sum_{S_1,\dotsc,S_k}
    \Ex[Y_{S_1}\dotsb Y_{S_k}\cdot Z]\leq (\Ex[X]+t)\cdot 
    \sum_{S_1,\dotsc,S_{k-1}}
    \Ex[Y_{S_1}\dotsb Y_{S_{k-1}}\cdot Z].
  \]
  By induction, $\Ex[X^k Z] 
  < (\Ex[X]+t)^k$ for every integer $k \le \lceil m / e_H \rceil$. Substituting this inequality into~\eqref{eq:stability2} implies the assertion of the lemma.
\end{proof}

In order to deduce~\cref{cor:exists-a-seed-condition-on-UT} from \Cref{lem:stability}, we require a lower bound on $\Pr(\UT_{H,\delta})$.

\begin{lemma}
  \label{lem:lower}
  If an event $\mathcal{E}$ satisfies $\Ex[X\mid \mathcal{E}] \ge (1+\delta+\varepsilon)\Ex[X]$ for some $\delta, \varepsilon > 0$, then
  \[
    \Pr\bigl(X \ge (1+\delta)\Ex[X] \mid \mathcal{E}\bigr) \ge \varepsilon p^{e_H}.
  \]
\end{lemma}
\begin{proof}
  Let $u \coloneqq (1+\delta)\Ex[X]$.
  Since $X\leq \Ex[X]\cdot p^{-e_H}$ always, we have
  \[
    u+\varepsilon \Ex[X]\leq \Ex[X\mid \mathcal{E}] \leq u+\Pr(X\geq u \mid \mathcal{E}) \cdot \Ex[X]\cdot p^{-e_H},
  \]
  and thus $\Pr(X\geq u\mid \mathcal{E})\geq \varepsilon p^{e_H}$. 
\end{proof}

\begin{corollary}
  \label{cor:lower}
  For all $\eps,\delta>0$, we have $\Pr\big(X\geq (1+\delta)\Ex[X]\big) \geq \varepsilon p^{e_H}\cdot e^{-\psi_H(\delta+\eps)}$.
\end{corollary}

\begin{proof}
  If $\psi_H(\delta+\eps) = \infty$, then
  the assertion of the lemma is vacuous. 
  Otherwise, there exists a graph $G$ with $e_G\log (1/p) =
  \psi_H(\delta+\eps)$ and with $\Ex_G[X] \geq (1+\delta+\eps)\Ex[X]$.
  Now, letting $u \coloneqq (1+\delta)\Ex[X]$, \Cref{lem:lower} implies the assertion of the corollary as follows:
  \[ 
    \Pr(X\geq u) \geq
    \Pr(X\geq u \mid G\subseteq G_{n,p}) \cdot \Pr(G\subseteq G_{n,p})
    \geq \varepsilon p^{e_H} \cdot p^{e_G} =\varepsilon p^{e_H}\cdot e^{-\psi_H(\delta+\varepsilon)}.\qedhere
  \]
\end{proof}

\begin{proof}[Proof of~\Cref{cor:exists-a-seed-condition-on-UT}]
  On the one hand, \Cref{lem:stability} yields
  \[
    \begin{split}
      \Pr\big(X \geq (1+\delta)\Ex[X]\land G_{n,p} \notin\langle \cS_{t,m}\rangle\big)
      & \leq \left(\frac{(1+\delta-\varepsilon)\Ex[X]}{(1+\delta)\Ex[X]}\right)^{m/e_H}
        \leq \left(\frac{1+\delta-\varepsilon}{1+\delta}\right)^{D \psi_H(\delta+\varepsilon) / e_H} \\
      & \le e^{-3\psi_H(\delta+\varepsilon)},
    \end{split}
  \]
  provided that $D$ is sufficiently large.
  On the other hand, by Corollary \ref{cor:lower} and our assumptions on $\psi_H(\delta+\varepsilon)$,
  \[
    \Pr\bigl(X \ge (1+\delta)\Ex[X]\bigr) \ge \varepsilon p^{e_H} \cdot e^{-\psi_H(\delta+\varepsilon)} \gg e^{-3\psi_H(\delta+\varepsilon)}.\qedhere
  \]
\end{proof}

The second key result of this section states that every $t$-seed of nearly minimal size contains a $t'$-seed, for some $t' \approx t$, that is also a hub-core.

\begin{proposition}
  \label{lem:cores}
  There exists a constant $K$ such that the following holds.
  Suppose that either:
  \begin{enumerate}[label=(\roman*)]
  \item
    $p_H \cdot (\log n)^K \ll p \ll 1$ and $m \le f \cdot \bigl(\log (1/p)\bigr)^2$; or
  \item
    $p_H \cdot (\log \log n)^K \ll p \ll 1$, and $m \le f \cdot \log \log (1/p)$.
  \end{enumerate}
  For all $t \ge \Omega(\Ex[X])$, letting $t' \coloneqq t - o(\Ex[X])$, we have
  \[
    \cS_{t,m} \subseteq \langle\cS_{t',m} \cap \cL\rangle.
  \]
\end{proposition}

\begin{proof}
  Denote $\ell \coloneqq 2\log\log(1/p)$ and fix an arbitrary $G \in \cS_{t,m}$.
  Let $G^*$ be the graph obtained from $G$ by repeatedly deleting from a graph $G' \subseteq G$ all edges $e$ that satisfy
  \begin{equation}
    \label{eq:core-trimming}
    N(J, G'; e) < \frac{n^{v_J}p^{e_J}}{\ell^2 \cdot \max\{e_{G'}, f\}}
  \end{equation}
  for all $J \in \cJ$ until no such edges are left and then removing all isolated vertices.
  The definition of $G^*$ guarantees that, for every $J \in \cJ$,
  \[
    N(J, G) - N(J, G^*) < \sum_{i=e_{G^*}+1}^{e_G}  \frac{n^{v_J}p^{e_J}}{\ell^2 \cdot \max\{i,f\}} \le \frac{n^{v_J} p^{e_J}}{\ell^2} \cdot \left(1 + \sum_{i=f+1}^{m} \frac{1}{i} \right) \le \frac{3n^{v_J} p^{e_J}}{\ell},
  \]
  as $m/f \le \bigl(\log(1/p)\bigr)^2 = e^{2\ell}$.
  Consequently, \Cref{prop:difference-in-expectations-important-J} yields
  \[
    \Ex_G[X] - \Ex_{G^*}[X] \le \sum_{J \in \cJ} \frac{|\Emb(J,H)|}{|\Aut(H)|} \cdot \frac{3n^{v_H}p^{e_H}}{\ell} + o(n^{v_H}p^{e_H}) \ll \Ex[X],
  \]
  which means that $G^* \in \mathcal{S}_{t',m}$.
  Since $t' \ge \Omega(\Ex[X])$, \Cref{lem:psi-lower} implies that $e_{G^*} \ge \Omega(f)$.

  All that remains is to show that $G^* \in \cL$.
  By construction, for every edge $e = uv \in G^*$, there exists some $J \in \cJ$ for which the converse of~\eqref{eq:core-trimming} holds and thus, by \Cref{cor:core-edge-bipartite},
  \begin{equation}
    \label{eq:NJG*-uv-two-sided}
    \frac{n^{v_J} p^{e_J}}{\ell^2 \cdot \max\{e_{G^*},f\}} \le N(J, G^*; uv) \le C_J \cdot (\deg_{G^*} u + \deg_{G^*}v) \cdot \frac{N(J, n, e_{G^*})}{e_{G^*} \cdot \min\{e_{G^*}, n\}}.
  \end{equation}
  Writing $x \coloneqq e_{G^*}/f$, \Cref{lem:bound-on-NJG} further implies that
  \begin{equation}
    \label{eq:NJG*-upper}
    N(J, n, e_{G^*}) \le \max\{2x, (2x)^{v_J} \} \cdot n^{v_J} p^{e_J} \cdot Z,
  \end{equation}
  where
  \[
    Z =
    \begin{cases}
      (np^{\Delta})^{\gamma_H} & \text{if $J \in \cJ \setminus \cJ^*$ and $p < n^{-1/\Delta}$}, \\
      1 & \text{otherwise}.
    \end{cases}
  \]
  Substituting \eqref{eq:NJG*-upper} into~\eqref{eq:NJG*-uv-two-sided} and rearranging the terms yields
  \[
    \begin{split}
      \deg_{G^*} u + \deg_{G^*}v
      & \ge \frac{1}{C_J \cdot \ell^2 \cdot \max\{1, 1/x\} \cdot \max\{2x, (2x)^{v_J}\} \cdot Z} \cdot \min\{e_{G^*},n\} \\
      & \ge \frac{1}{C_J' \cdot \ell^2 \cdot \max\{1, x^{v_J}\} \cdot Z} \cdot \min\{e_{G^*},n\},
    \end{split}
  \]
  where $C_J'$ depends only on $C_J$ and $H$.
  Further, if $p \ll n^{-1/\Delta} (\log n)^{-\omega}$, then the trivial inequality $\deg_{G^*} u + \deg_{G^*} v \le 2v_{G^*} \le 4\min\{e_{G^*}, n\}$ and our assumption that $x \le m/f \le (\log n)^2$ imply that $Z$ must be equal to one and thus $uv$ lies in a copy of some $J \in \cJ^*$ in $G^*$.  This means, in particular, that both $\deg_{G^*}u$ and $\deg_{G^*} v$ are at least $\delta(J) \ge d_H$.
  Finally, since $x \le m/f$ and $Z \le 1$, this implies that every edge of $G^*$ must have an endpoint of degree at least
  \[
    d \coloneqq \frac{\min\{e_{G^*},n\}}{2C_J' \cdot \ell^2 \cdot \max\{1, (e_{G^*}/f)^{v_H}\}},
  \]
  which in turn implies that $G^*$ has a vertex cover of size at most
  \[
    \frac{2e_{G^*}}{d} \le \ell^3 \cdot \max\{1, (e_{G^*}/f)^{v_H}\} \cdot \max\left\{1, \frac{e_{G^*}}{n}\right\}.
  \]
  In order to conclude that $G^* \in \mathcal{L}$, we still need to show that $d \ge 2\omega$.
  To this end, recall that $e_{G^*} = \Omega(f)$ and thus
  \[
    d \ge \frac{\min\{f, n\}}{\ell^3 \cdot \max\{1, (m/f)^{v_H}\}}.
  \]
  Since $f \ge p / p_H$, see \cref{fact:f-p-pH}, this follows under both sets of assumptions of the lemma, provided that $K$ is sufficiently large.
\end{proof}

The final key result of this section gives an upper bound on the number of hub-cores with a given number of edges that matches the trivial lower bound up to a `small' error term.

\begin{lemma}
  \label{lem:core-count}
  Suppose that $p_H \ll p \ll 1$.
  For every $m \le f \cdot (\log(1/p))^2$ and every family $\cL' \subseteq \cL$ of hub-cores with $m$ edges,
  \[
    \log |\cL'| \le \max_{G \in \cL'} v_G \cdot \log(en/g) + f + O\bigl(m \log\log(1/p)\bigr).
  \]
\end{lemma}

\begin{proof}
  Denote $v^* \coloneqq \max_{G \in \cL'} v_G$.
  Every graph $G \in \cL'$ can be constructed as follows.
  First, choose its vertex set $V(G) \subseteq \br{n}$ of size at most $v^*$ and then a vertex cover $D_G \subseteq V(G)$ of size at most
  \[
    d_m \coloneqq \bigl(\log\log(1/p)\bigr)^{3} \cdot \max\bigl\{1, (m/f)^{v_H}\bigr\} \cdot \max\{1, m/n\}.
  \]
  Given $V(G)$ and $D_G$, the edges of $G$ are chosen from among the at most $d_m v^*$ pairs of vertices that intersect $D_G$.
  To summarise,
  \begin{equation}
    \label{eq:cL'-upper}
    |\cL'| \le \sum_{v = 0}^{v^*} \binom{n}{v} \cdot 2^{v^*} \cdot \binom{d_mv^*}{m} \le \left(\frac{en}{v^*}\right)^{v^*} \cdot 2^{v^*} \cdot \left(\frac{ed_mv^*}{m}\right)^m.
  \end{equation}
  Further, as
  \[
    v^* \cdot \max\left\{1, \frac{m}{n}\right\} \le \min\{2m, n\} \cdot \max\left\{1, \frac{m}{n}\right\} \le 2m,
  \]
  we have
  \[
    d_mv^* \le 2m \cdot \bigl(\log\log(1/p)\bigr)^{3} \cdot \max\left\{1, \left(\frac{m}{f}\right)^{v_H}\right\} \le m \cdot \bigl(\log(1/p)\bigr)^{2v_H+3}.
  \]
  Substituting this estimate into~\eqref{eq:cL'-upper}, we obtain
  \begin{equation}
    \label{eq:log-cL'-upper}
    \log |\cL'| \le v^* \log\left(\frac{2en}{v^*}\right) + C m \log \log(1/p)
  \end{equation}
  for some constant $C$ that depends only on $H$.
  Finally, since $x \log(ea/x) \le a$ for all $x, a > 0$,
  \[
    v^*\log\left(\frac{2en}{v^*}\right) - v^*\log\left(\frac{en}{g}\right) \le v^* \log\left(\frac{eg}{v^*}\right) \le g \le f.
  \]
  Combining this estimate with~\eqref{eq:log-cL'-upper} gives the assertion of the lemma.
\end{proof}

Finally, we conclude this section with the proof of the lower bound in \cref{thm:main-dense-ish}.

\begin{proof}[Proof of the lower bound in \cref{thm:main-dense-ish}]
  Let $K = K(H, \delta, \varepsilon)$ and $D = D(H, \delta, \varepsilon)$ be large constants, assume that $p \ge p_H \cdot (\log n)^K$, and set $t \coloneqq (\delta-\varepsilon) \cdot \Ex[X]$ and $t' \coloneqq (\delta-2\varepsilon) \cdot \Ex[X]$.
  Further, for every positive integer $m$, define
  \[
    \cL_m\coloneqq \{G \in \cS_{t',m} \cap \cL:e_G=m\}
    \qquad
    \text{and}
    \qquad
    v_m \coloneqq \max\{v_G : G \in \mathcal{L}_m\},
  \]
  let $m_0 \coloneqq \min\{m:\cL_m\neq \emptyset\}$ and $m_1 \coloneqq \lceil D \psi_H(\delta+\varepsilon) \rceil$,
  and observe that \cref{lem:psi--psi*-lower} implies that $f/D \le m_0 \le m_1 \le D^2 f \log(1/p)$, provided that $D$ is sufficiently large.
  It follows from \cref{cor:exists-a-seed-condition-on-UT,lem:cores} that
  \begin{equation}\label{eq:UT-dense-ish-regime-bound-by-cores}
    \begin{split}
      \Pr(\UT_{H,\delta})
      & \le 2 \cdot \Pr\bigl(G_{n,p} \in \langle\cS_{t,m_1}\rangle \bigr)\leq 2 \cdot \Pr\bigl(G_{n,p} \in  \langle\cS_{t',m_1} \cap \cL\rangle\bigr)\\
      & \le 2 \cdot \sum_{m=m_0}^{m_1} |\cL_m|\cdot  p^{m} \le  n^2 \cdot \max \bigl\{|\cL_m|\cdot p^{m}: m_0 \le m \le m_1\bigr\}.
    \end{split}
  \end{equation}
  Further, applying \cref{lem:core-count} with $\cL'=\cL_m$ for each $m \in [m_0, m_1]$, we get
  \[
    \begin{split}
      \log \bigl(|\cL_m| \cdot p^m\bigr)
      & \le v_m\log(en/g) - m \log(1/p) + O(m \log \log(1/p)) \\
      & \le -(1-\varepsilon) \cdot \bigl(m\log (1/p) - v_m \log(en/g)\bigr),
    \end{split}
  \]
  where the last inequality follows from~\cref{lem:weak-LB-psi-*} and our assumption that $p \ge p_H \cdot (\log n)^K$ for some sufficiently large $K$.
  Substituting this estimate into~\eqref{eq:UT-dense-ish-regime-bound-by-cores} yields
  \[
    \begin{split}
      \log \Pr(\UT_{H,\delta})
      & \le - (1-\varepsilon) \cdot \min\bigl\{m \log(1/p) - v_m \log(en/g) : m_0 \le m \le m_1 \bigr\} + 2 \log n \\
      & \le - (1-\varepsilon) \cdot \psi_H^*(\delta-2\varepsilon) + 2\log n \le -(1-2\varepsilon) \cdot \psi_H^*(\delta-2\varepsilon),
    \end{split}
  \]
  where the last inequality follows as $\psi_H^*(\delta-2\varepsilon) \ge f \ge p/p_H \gg \log n$, by \Cref{lem:psi--psi*-lower} and~\cref{fact:f-p-pH}.
\end{proof}

\section{Stable graphs}
\label{sec:stable-graphs}

In this section, we will explore the structure of stable graphs.
In particular, we will introduce the key concept of dimension associated with a pair of graphs in $\cJ_\emptyset^* = \cJ^* \cup \{\emptyset\}$ that is necessary to define the constant $C_H$ that appears in the statement of \cref{thm:main-near-appearance-threshold} as well as the notions of one-dimensional and clean graphs.

\subsection{An alternative characterisation of stable graphs}
\label{sec:alt-char-stable-graphs}

We start our discussion with an alternative characterisation of stable graphs, which was already mentioned in the introduction and which we restate here for convenience.

\stablechar*

\begin{proof}[Proof of \cref{lem:stable-characterisation}]
  Observe first that every $J \in \cJ^*$ satisfies $e_J / v_J = r\Delta / (1+r)$ and $e_J = r \Delta \alpha_J^*$.
  Further, let $\mathcal{E} \coloneqq \{J \subseteq H : e_J > r \Delta \alpha_J^*\}$ and recall that
  \begin{equation}
    \label{eq:epsH-plus-Delta}
    \frac{1}{\Delta} + \varepsilon_H = \frac{1}{\Delta} + \min_{J \in \mathcal{E}} \frac{v_J - \alpha_J^* - e_J/\Delta}{e_J - r \Delta \alpha_J^*} = \min_{J \in \mathcal{E}} \frac{v_J - (1+r) \alpha_J^*}{e_J - r \Delta \alpha_J^*}.
  \end{equation}
  Finally, let $\mathcal{M} \coloneqq \{J \subseteq H : e_J / v_J = m(H)\}$ be the family of densest subgraphs of $H$.

  Suppose first that $m(H) > r\Delta/(1+r)$.
  In this case, $\cJ^* \neq \mathcal{M}$, as the two sets are nonempty and disjoint.
  We claim that $H$ cannot be stable.
  Indeed, consider an arbitrary $J \in \mathcal{M}$ and note that, by \Cref{lemma:eJ-alphaJ-clique},
  \[
    r \cdot \Delta \alpha_J^* \le r \cdot (\Delta v_J - e_J) = e_J \cdot ( r\Delta / m(H) - r ) < e_J
  \]
  and thus $J \in \mathcal{E}$.
  In particular, it follows from~\eqref{eq:epsH-plus-Delta} that $1/\Delta + \varepsilon_H \le v_J / e_J = 1/m(H)$.

  We may thus assume that $m(H) = r\Delta/(1+r)$, which means that $\cJ^* \subseteq \mathcal{M}$.
  Suppose first that $\cJ^* \neq \mathcal{M}$ and let $J \in \mathcal{M} \setminus \cJ^*$ be arbitrary.
  Since $J \notin \cJ$, \Cref{lemma:eJ-alphaJ-clique} yields
  \[
    r \cdot \Delta \alpha_J^* < r \cdot (\Delta v_J - e_J) = e_J \cdot ( r\Delta / m(H) - r ) = e_J
  \]
  and thus $J \in \mathcal{E}$.
  As before, we may conclude that $1/\Delta + \varepsilon_H \le v_J / e_J = 1/m(H)$, which means that $H$ is not stable.

  Suppose now that $H$ is not stable and let $J \in \mathcal{E}$ be an arbitrary graph such that
  \[
    \frac{1}{\Delta} + \varepsilon_H = \frac{v_J - (1+r)\alpha_J^*}{e_J - r\Delta\alpha_J^*} \le \frac{1}{m(H)} = \frac{1+r}{r\Delta}.
  \]
  Since $e_J/v_J \le m(H)$, this must mean that $e_J/v_J = m(H)$ and thus $J \in \mathcal{M} \cap \mathcal{E} \subseteq \mathcal{M} \setminus \cJ^*$.
\end{proof}

\subsection{The structure of stable graphs}
\label{sec:struct-stable-graphs}

We start this subsection with the definitions of the notions of span, dimension, basis, and root.
Recall that $\cJ_\emptyset^* = \{\emptyset\} \cup \cJ^*$.
We first note the following important structural property of $\cJ_{\emptyset}^*$.

\begin{lemma}
  \label{lem:intersections-and-unions}
  If $H$ is stable, then the family $\cJ_\emptyset^*$ is closed under unions and intersections.
\end{lemma}

\begin{proof}
  Suppose that $J_1,J_2\in \cJ^*_{\emptyset}$. Suppose further that $J_1,J_2\neq \emptyset$, as otherwise the assertion is trivial. By \cref{lem:stable-characterisation}, we have $e_{J_i} = m(H) \cdot v_{J_i}$ for both $i \in \br{2}$.
  Since $e_{J_1\cup J_2} \le m(H) \cdot v_{J_1 \cup J_2}$ and $e_{J_1 \cap J_2} \le m(H) \cdot v_{J_1 \cap J_2}$, we have
  \[
    m(H)\geq \frac{e_{J_1\cup J_2}}{v_{J_1\cup J_2}} = \frac{e_{J_1} + e_{J_2} - e_{J_1\cap J_2} }{v_{J_1} + v_{J_2} - v_{J_1\cap J_2}}  \ge \frac{m(H) \cdot (v_{J_1} + v_{J_2} - v_{J_1\cap J_2}) }{v_{J_1} + v_{J_2} - v_{J_1\cap J_2}} =m(H),
  \]
  which implies that $e_{J_1 \cup J_2} = m(H) \cdot v_{J_1 \cup J_2}$ and $e_{J_1 \cap J_2} = m(H) \cdot v_{J_1 \cap J_2}$.
\end{proof}

\begin{definition}
  Suppose that $J_0, J \in \cJ_\emptyset^*$ satisfy $J_0 \subsetneq J$.
  The \emph{span} of a set $E$ of edges of $J$ over $J_0$ is the graph
  \[
    \spn_{J_0}E \coloneqq \bigcap \bigl\{J' \in \cJ^* : J_0 \cup E \subseteq J'\bigr\}.
  \]
  The \emph{dimension} of $J$ over $J_0$ is
  \[
    \dim_{J_0}J \coloneqq \min\bigl\{|E| : E \subseteq J \wedge \spn_{J_0}E = J\bigr\}.
  \]
  A set $E$ of edges of $H$ is a \emph{basis} of $J$ over $J_0$ if $\spn_{J_0}E = J$ and $|E| = \dim_{J_0}J$.
  Finally, a vertex $v \in V(J_0)$ is a \emph{root} of $J$ over $J_0$ if $\deg_Jv > \deg_{J_0}v$;
  the set of all roots of $J$ over $J_0$ is denoted by $R_{J_0}(J)$.
\end{definition}

We can finally formally define the constant $C_H$ from the statement of~\Cref{thm:main-near-appearance-threshold}:
\[
  C_H \coloneqq \frac{1}{\Delta r} \cdot \max\left\{\frac{\dim_{J_0}J}{\alpha_{J}^* - \alpha_{J_0 - R_{J_0}(J)}^* - \dim_{J_0} J} : J_0, J \in \cJ_\emptyset^* \wedge J_0 \subsetneq J\right\}.
\]
Our next lemma shows that the above definition is valid (the inequality $\alpha_{J_0}^* \ge \alpha_{J_0-R_{J_0}(J)}^*$ is trivial).

\begin{lemma}\label{lem:C_H-is-well-defined}
  For all $J_0, J \in \cJ_\emptyset^*$ with $J_0 \subsetneq J $, we have
  \[
    \alpha_J^* > \alpha_{J_0}^* + \dim_{J_0}J.
  \]
\end{lemma}
\begin{proof}
  Note that it suffices to show that, for every $J' \in \cJ_\emptyset^*$ satisfying $J_0 \subseteq J' \subsetneq J$, there is an edge $e \in J \setminus J'$ such that, letting $J_e \coloneqq \spn_{J'} \{e\} \in \cJ_\emptyset^*$, we have $\alpha_{J_e}^* > \alpha_{J'}^* + 1$.
  Recall that each $J' \in \cJ_\emptyset^*$ admits a bipartition $A_{J'} \cup B_{J'} = V(J')$ such that all vertices in $A_{J'}$ have degree $\Delta$ and $\alpha_{J'}^* = |B_{J'}| = |A_{J'}| / r$, see \cref{lem:unique-frac-ind-set} (if $J' = \emptyset$, this is vacuously true).
  Consequently, for an arbitrary $J' \in \cJ_\emptyset^*$ with $J_0 \subseteq J' \subsetneq J$ and an edge $e \in J \setminus J'$, we have
  \[
    \alpha_{J_e}^* - \alpha_{J'}^* = |B_{J_e}| - |B_{J'}| = \frac{|A_{J_e}| - |A_{J'}|}{r} \ge \frac{1}{r} > 1,
  \]
  as desired.
\end{proof}

\begin{definition}[optimal pairs]
  \label{dfn:optimal-pair}
  A pair $(J_0, J)$ of graphs from $\cJ_\emptyset^*$ is called \emph{optimal} if it achieves the maximum in the definition of $C_H$.
\end{definition}

It follows from~\cref{lem:intersections-and-unions} that, for every pair $J_0, J \in \cJ_\emptyset^*$ with $J_0 \subsetneq J$ and all $E \subseteq J$,
\[
  \spn_{J_0} E = \bigcup_{e \in E} \spn_{J_0} e.
\]
In particular, for every such pair, letting $d \coloneqq \dim_{J_0} J$, there are $J_1, \dotsc, J_d \in \cJ^*$ such that $J = J_1 \cup \dotsb \cup J_d$ and $\dim_{J_0} J_i = 1$ for all $i$.
Our next proposition states that optimal pairs admit such a decomposition that has the extra property that $J_i \cap J_j = J_0$ for all distinct $i,j \in \br{d}$.

\begin{proposition}
  \label{prop:optimal-pair-decomposition}
  Suppose that  $(J_0, J)$ is an optimal pair and let $d \coloneqq \dim_{J_0}J$.
  There are $J_1, \dotsc, J_d \in \cJ^*$ such that $J_1 \cup \dotsb \cup J_d = J$, $\dim_{J_0} J_i = 1$ for each $i \in \br{d}$, and $J_i \cap J_j = J_0$ for all distinct $i, j \in \br{d}$.
\end{proposition}

Given an optimal pair $(J_0, J)$ with $d = \dim_{J_0} J$, a sequence $(J_1, \dotsc, J_d)$ satisfying the assertion of \cref{prop:optimal-pair-decomposition} will be called a \emph{fan decomposition} of $J$ over $J_0$.
One can in fact prove that every optimal pair has a unique fan decomposition, but we will not need to use this in the sequel.

\begin{definition}[one-dimensional graphs]
  \label{def:1-dim-graphs}
  A stable graph $H$ is called \emph{one-dimensional} if $\dim_{J_0}J = 1$ for every optimal pair $(J_0, J)$.
\end{definition}

Finally, we define a subfamily of stable graphs called \emph{clean graphs}.
In \cref{sec:LB-near-the-appearance-threshold}, we will derive lower bounds on the upper-tail probability, and for the family of clean graphs, we will be able to derive a stronger statement.

\begin{definition}[clean graphs]
  \label{def:clean-graphs}
  We say that $H$ is a \emph{clean graph} if some optimal pair $(J_0, J)$ admits a fractional independent set $a \colon V(J)\to [0,1]$ and a fan decomposition $J_1, \dotsc, J_d$ such that, writing $R_i$ for the set of roots of $J_i$ over $J_0$, we have:
  \begin{enumerate}
  \item
    For every $i \in \br{d}$, the pair $(J_0, J_i)$ is optimal and $a(R_i) = |R_i|$.
  \item
    The sets $R_1, \dotsc, R_d$ are pairwise disjoint.
  \item
    The set $R = R_1 \cup \dotsb \cup R_d$ of roots of $J$ over $J_0$ satisfies $a(V(J_0)\setminus R)=\alpha^*_{J_0-R}$.
  \end{enumerate}
\end{definition}

The proof of \Cref{prop:optimal-pair-decomposition} will crucially use the following technical estimate.

\begin{lemma}\label{claim:J0-in-J'0-not-maximal}
  Suppose that $J_0,J_0',J\in \cJ_\emptyset^*$ satisfy $J_0 \subsetneq J_0'$ and let $R_{0},R'_{0}$ be the sets of roots of $J$ over $J_0,J'_{0}$, respectively. Then, $\alpha_{J_0' - R_0'} > \alpha_{J_0 - R_0}$.
\end{lemma}

\begin{proof}
  The key observation here is that $W_0 \coloneqq V(J_0) \setminus R_0$ is the set of $w \in V(J_0)$ satisfying $\deg_{J_0} w = \deg_J w$.
  Consequently, every edge of $J$ with an endpoint in $W_0$ belongs to $J_0$.
  The analogous holds true for $W_0' \coloneqq V(J_0') \setminus R_0'$.
  Since $J_0 \subseteq J_0'$, we thus have $W_0 \subseteq W_0'$.

  Let $I_0 \subseteq W_0 \subseteq W_0'$ be a largest independent set in $J_0 - R_0$.
  Since $J[W_0] = J_0[W_0] = J_0'[W_0]$, by the above observation, $I_0$ is also independent in $J_0' - R_0' = J_0'[W_0']$.
  Note that any $v \in V(J_0') \setminus V(J_0)$ with $\deg_{J_0'} v = \Delta$ belongs to $W_0'$ and may be added to $I_0$ while maintaining independence, as $N(v) \cap W_0 = \emptyset$.
  We may thus assume that all vertices of degree $\Delta$ in $J_0'$ already belong to~$J_0$.

  In this case, we argue differently.
  Let $(A_0, B_0)$ and $(A_0', B_0')$ be the bipartitions of $J_0$ and $J_0'$, respectively.
  Since $J_0\subsetneq J_0'$, our assumption that all vertices of degree $\Delta$ are in $V(J_0)$ implies that there must be some vertex $v \in A_0'$ such that $0 < \deg_{J_0} v < \Delta$ and thus $v \in B_0 \cap R_0$.
  Let $J_v$ be the connected component of $v$ in $J_0$ and note that $A_0' \cap V(J_v) = B_0 \cap V(J_v)$ and $J_v \in \cJ^*$.
  Finally, let $I \coloneqq (I_0 \setminus V(J_v)) \cup (A_0' \cap V(J_v))$ and note that $I \subseteq W_0'$ as $A_0' \subseteq W_0'$.

  We claim that $I$ is an independent set in $J_0'$ and $|I| > |I_0|$.
  Indeed, $I$ is an independent set since both $I_0$ and $A_0'$ are independent and $N(A_0' \cap V(J_v)) \subseteq (B_0' \cap V(J_v)) \cup R_0 \subseteq (I_0 \setminus V(J_v))^c$.
  Further, since $B_0 \cap V(J_v) \ni v$ is the unique largest independent set of $J_v$, while $I_0 \cap V(J_v) \not\ni v$, as $v \in R_0$, we have $|A_0' \cap V(J_v)| = |B_0 \cap V(J_v)| > |I_0 \cap V(J_v)|$ and thus $|I| > |I_0|$.
\end{proof}

\begin{proof}[Proof of \cref{prop:optimal-pair-decomposition}]
  We may clearly assume that $d \ge 2$, since otherwise the assertion is trivial.
  Fix an arbitrary $E = \{e_1, \dotsc, e_d\} \subseteq J$ such that $\spn_{J_0} E = J$ and, for each $i \in \br{d}$, let $J_i \coloneqq \spn_{J_0} e_i$.
  Further, for each $i \in \br{d}$, let $J_{(i)} \coloneqq \bigcup_{j \neq i} J_j$, define $J_0' \coloneqq \bigcap_{i=1}^d J_{(i)}$, and observe that $J_0 \subseteq J_0' \subseteq J$; note that $J_{(i)},J_0' \in \cJ_\emptyset^*$, by \cref{lem:intersections-and-unions}.
  Fix a pair of distinct $i, j \in \br{d}$. Since $J_i \cap J_j \subseteq J_{(k)}$ for every $k \in \br{d}$, we have $J_0 \subseteq J_i \cap J_j \subseteq J_0'$.
  It thus suffices to show that $J_0 = J_0'$.

  We first show that $\dim_{J_0} J = \dim_{J_0'}J$, which will follow if we show that $E$ is a basis of $J$ over $J_0'$. Indeed, it is clear that $E$ is a spanning set of $J$ over $J_0'$, as $J = \spn_{J_0} E \subseteq \spn_{J_0'} E \subseteq J$.
  Let us show that $E$ is also a smallest spanning set.
  Suppose towards contradiction that there was a set $F$ of size at most $d-1$ that spans $J$ over $J_0'$.
  Since $J = \bigcup_{i=1}^d J_i$, each edge of $F$ must belong to $J_j$ for some $j \in \br{d}$.
  Consequently, there must be some $i \in \br{d}$ such that $F \subseteq J_{(i)}$, which means that $J_0' \cup F \subseteq J_{(i)} \subsetneq J$, contradicting the assumption that $F$ spans $J$ over $J_0'$.

  Denote by $R_0$ and $R_0'$ the sets of roots of $J$ over $J_0$ and $J_0'$, respectively.
  If $J_0 \subsetneq J_0'$, then \cref{claim:J0-in-J'0-not-maximal} would give $\alpha_{J_0' - R_0'}^* = \alpha_{J_0' - R_0'} > \alpha_{J_0 - R_0} = \alpha^*_{J_0 - R_0}$.
  However, since $\dim_{J_0'} J = \dim_{J_0} J$, this would yield the following contradiction to the optimality of $(J_0, J)$
  \[
    \frac{\dim_{J_0'} J}{\alpha_{J}^* - \alpha_{J_0' - R_0'}^* - \dim_{J_0'}J}
    > \frac{\dim_{J_0} J}{\alpha_{J}^* - \alpha_{J_0 - R_0}^* - \dim_{J_0} J} = r\Delta C_H.\qedhere
  \]
\end{proof}

\section{Upper bounds for stable graphs}
\label{sec:upper-bounds-stable}

Throughout this section we assume that $H$ is a stable graph, that is, $\eps_H>1/m(H)-1/\Delta$, and in particular $p_H=n^{-1/m(H)}$.
Our main goal here is to prove upper bounds on the upper-tail probability that are implicit in \cref{thm:main-near-appearance-threshold}~\ref{item:main-near-appearance-threshold-UB} and its sharper version, \cref{thm:non-clean}.

\subsection{Tight cores}
\label{sec:tight-cores}

We start by formally defining the notion of tight cores and establish several key properties of these graphs.
It will be convenient to denote, given a graph $G$,
\[
  \cJ^*(H,G) \coloneqq \{K \subseteq G : K \cong J \text{ for some }J\in \cJ^*\}.
\]

\begin{definition}[Tight cores]\label{def:tight-cores}
  We say that a graph $C\subseteq K_n$ is a \emph{tight core} if there exists a sequence of graphs $\emptyset = C_0 \subseteq C_1 \subseteq \dotsb \subseteq C_\ell=C$ such that, for every $i \in \br{\ell}$, either
  \begin{enumerate}[label=\textbf{(C\arabic*)}]
  \item
    \label{item:tight-core-one-copy}
    $C_i = C_{i-1} \cup K$ for some $K \in \cJ^*(H,C)$ such that $K\cap C_{i-1} \notin \cJ_\emptyset^*$ or
  \item
    \label{item:tight-core-two-copies}
    $C_i = C_{i-1} \cup K_1 \cup K_2$ for some $K_1,K_2\in \cJ^*(H,C)$ such that $K_1 \cap C_{i-1} \in \cJ_\emptyset^*$ but $K_2\cap (C_{i-1} \cup K_1) \notin \cJ_\emptyset^*$;
  \end{enumerate}
  we call $(C_0, C_1, \dotsc, C_\ell)$ a \emph{witness sequence} of $C$.
  Finally, we denote by $\cC$ the set of all tight cores, and by $\cC_m$ the set of all tight cores with at most $m$ edges. 
\end{definition}

The first key result of this section establishes a certain trichotomy for the family of $t$-seeds with $O(f \log(1/p))$ edges.
Each such seed either:
\begin{enumerate}[label=(\roman*)]
\item
  contains a tight core with $\omega(f)$ edges; or
\item
  contains a $t'$-seed, for some $t' \approx t$, that has only $O(f)$ edges;
\item
  contains a maximal tight core $G^*$ with $O(f)$ edges that is not a $t'$-seed for any $t' \approx t$.
\end{enumerate}
Given two pairs of graphs, $J_0 \subseteq J$ and $G^* \subseteq G$, we shall denote by $N(J,G; J_0,G^*)$ the number of copies of $J$ in $G$ that intersect $G^*$ on a subgraph that is isomorphic to $J_0$.

\begin{proposition}
  \label{lem:seeds-contains-a-small-tight-core}
  For every $\varepsilon > 0$, there exists a constant $c > 0$ such that the following holds for all $p_H \le p \le n^{-1/\Delta-\Omega(1)}$, all $t$, $m_c$, and $m = O\bigl(f \log(1/p)\bigr)$, and every $G \in \cS_{t,m}$:
  \begin{enumerate}[label=(\roman*)]
  \item
    \label{item:large-tight-core-or-small-seed}
    Either $G \in  \langle\cC \setminus \cC_{m_c}\rangle \cup \langle\cS_{t - \eps \Ex[X], m_c}\rangle$ or
  \item
    \label{item:small-tight-core-with-Poisson-event}
    $G$ contains a maximal tight core $G^*$ with at most $m_c$ edges such that $N(J,G;J_0,G^*)\geq c n^{v_J}p^{e_J}$ for some $J_0 ,J \in \cJ_{\emptyset}^*$ satisfying  $J_0\subsetneq J$.
  \end{enumerate}
\end{proposition}

\begin{proof}
  Consider an arbitrary $G \in \cS_{t,m}$ and let $G^* \subseteq G$ be a maximal tight core.
  We may assume that $e_{G^*} \le m_c$, as otherwise $G \in \langle\cC \setminus \cC_{m_c}\rangle$.
  Further, we may assume that $ \Ex_G[X] - \Ex_{G^*}[X]  \ge \varepsilon \Ex[X]$, as otherwise $G^*\in \cS_{t - \eps \Ex[X], m_c}$.
  By \Cref{prop:difference-in-expectations-important-J}, our assumption on $p$, and the fact that $1/\Delta + \eps_H > 1/m(H)$, which holds due to the stability of $H$, we have
  \[
    \sum_{J\in \cJ^*} \frac{|\Emb(J,H)|}{|\Aut(H)|} \cdot n^{v_H-v_J}p^{e_H-e_J}\cdot (N(J,G)-N(J,G^*))\geq \frac{\eps}{2}  \Ex[X],
  \]
  while the maximality of $G^*$ implies that, for every $J \in \cJ^*,$
  \[
    N(J, G) - N(J, G^*) = \sum_{\substack{J_0 \subsetneq J \\ J_0 \in \cJ_\emptyset^*}} N(J, G; J_0, G^*).
  \]
  Since $\Ex[X] = \Theta(n^{v_H} p^{e_H})$, there must be $J_0, J \in \cJ_{\emptyset}^*$ satisfying $J_0 \subsetneq J$ and $c > 0$ such that
  \[
    N(J,G;J_0,G^*)\geq c n^{v_J} p^{e_J}.\qedhere
  \]
\end{proof}

The second key result of this section states that, when $p \le p_H \cdot n^{o(1)}$, large tight cores are very unlikely to appear in $G_{n,p}$.

\begin{proposition}
  \label{lem:large-tight-core-is-expensive}
  There exists a positive constant $\eta$ such that, for all sufficiently large $n$, all $p_H \le p \le p_H \cdot n^{\eta}$, and all $m_c \le f \log(1/p)$,
  \[
    \Pr\bigl(G_{n,p}\in \langle \cC \setminus \cC_{m_c} \rangle\bigr) \leq p^{\eta m_c}.
  \]
\end{proposition}

The third key result of this section states that assertion \ref{item:small-tight-core-with-Poisson-event} in \Cref{lem:seeds-contains-a-small-tight-core} cannot hold unless $p = O(p_H (\log n)^{C_H})$.

\begin{proposition}
  \label{lem:seeds-contains-a-small-tight-core-2}
  For all $c, D > 0$, there exists an $L$ such that the following holds for all $p \ge p_H \cdot (\log n)^{C_H-1/L}$.
  Suppose that $G$ is a graph with at most $Df\log(1/p)$ edges whose maximal tight core $G^*\subseteq G$ has at most $Df$ edges.
  Then, every pair $J_0, J \in \cJ_{\emptyset}^*$ with $J_0 \subsetneq J$ satisfies $N(J,G; J_0,G^*) < c n^{v_J}p^{e_J}$, unless $p \le L p_H  (\log n)^{C_H}$ and $(J_0, J)$ is an optimal pair.
\end{proposition}

The remainder of this section is devoted to proofs of \cref{lem:large-tight-core-is-expensive,lem:seeds-contains-a-small-tight-core-2}.
We start with the former.
The heart of the matter is the following lemma, which sheds some light on the somewhat mysterious definition of tight cores.

\begin{lemma}
  \label{lem:tight-core-extension}
  There exist positive constants $\eta$ and $\sigma$ such that the following holds for all $p_H \le p \le p_H \cdot n^\eta$.
  For all $J' \subseteq J \in \cJ^*$, all large enough $n$, and all $G \subseteq K_n$ with at most $f \cdot \bigl(\log(1/p)\bigr)^2$ edges,
  \[
    \left|\bigl\{\varphi \in \Emb(J,K_n) : \varphi(J) \cap G = \varphi(J')\bigr\} \right| \cdot p^{e_J - e_{J'}} \le
    \begin{cases}
      n^\sigma & \text{if $J' \in \cJ_\emptyset^*$,}\\
      n^{-3\sigma} & \text{otherwise}.
    \end{cases}
  \]
\end{lemma}
\begin{proof}
  Suppose that $p \le p_H \cdot n^{\eta}$ for some $\eta > 0$ and denote the number of embeddings from the assertion of the lemma by $E$.
  Since each such embedding $\varphi$ can be specified by first choosing an embedding of $J'$ in $G$ and then the images of the vertices in $V(J) \setminus V(J')$ in $\br{n}$, we have
  \begin{equation}
    \label{eq:tight-core-extension-upper}
    E \cdot p^{e_J-e_{J'}} \le |\Emb(J', G)| \cdot n^{v_J-v_{J'}} \cdot p^{e_J-e_{J'}} \le C \cdot \frac{N(J',G)}{n^{v_{J'}}p^{e_{J'}}} \cdot n^{v_J}p^{e_J}
  \end{equation}
  for some constant $C$ that depends only on $H$.
  Since $H$ is stable and $J \in \cJ^*$, by \cref{lem:stable-characterisation},
  \[
    n^{v_J}p^{e_J} = n^{v_J} p_H^{e_J} \cdot (p/p_H)^{e_J} = (p/p_H)^{e_J} \le n^{\eta e_H}.
  \]
  Further, \Cref{lem:bound-on-NJG} and the assumption that $e_G \le f \cdot \bigl(\log(1/p)\bigr)^2$ yield
  \[
    \frac{N(J',G)}{n^{v_{J'}}p^{e_{J'}}} \le \bigl(\log(1/p)\bigr)^{2v_H} \cdot
    \begin{cases}
      1 & \text{if $J' \in \cJ_\emptyset^*$}, \\
      \max\{n^{-1/(2\Delta)}, (np^\Delta)^{\gamma_H}\} & \text{otherwise},
    \end{cases}
  \]
  for some $\gamma_H > 0$ that depends only on $H$.
  Finally, since $p_H \le n^{-1/\Delta-2\zeta_H}$ for some $\zeta_H > 0$, we have $np^\Delta \le np_H^\Delta \cdot n^{\eta\Delta}\le n^{-\Delta\zeta_H}$, provided that $\eta \le \zeta_H/2$.
  Substituting these bounds into~\eqref{eq:tight-core-extension-upper} gives
  \[
    E \cdot p^{e_J-e_{J'}} \le n^{2\eta e_H} \cdot
    \begin{cases}
      1 & \text{if $J' \in \cJ_\emptyset^*$}, \\
      \max\{n^{-1/(2\Delta)}, n^{-\Delta \gamma_H \zeta_H}\} & \text{otherwise},
    \end{cases}
  \]
  which clearly implies the assertion of the lemma.
\end{proof}

\begin{proof}[Proof of \cref{lem:large-tight-core-is-expensive}]
  Fix an arbitrary $C \in \cC \setminus \cC_{m_c}$ and let $(C_0, \dotsc, C_w)$ be its witness sequence.
  Since $1 \le e(C_i \setminus C_{i-1}) \le 2e_H$ for all $i \in \br{w}$, there is an $m_c / (2e_H) \le \ell \le m_c$ such that $m_c \le e(C_\ell) \le m_c + 2e_H$.
  Consequently, denoting by $\mathcal{W}_{\ell,m}$ the set of all witness sequences $(C_0, \dotsc, C_\ell)$ with $e(C_{\ell}) = m$, we have  
  \begin{equation}
    \label{eq:large-tight-core-is-expensive}
    \Pr\bigl(G_{n,p}\in \langle \cC \setminus \cC_{m_c} \rangle\bigr) \leq \sum_{m =m_c}^{m_c+2e_H} \sum_{\ell = m_c/(2e_H)}^{m_c} |\mathcal{W}_{\ell,m}| \cdot p^{m}.
  \end{equation}
  It thus suffices to prove the following claim, where $\sigma$ is the constant appearing in \cref{lem:tight-core-extension}.
  \begin{claim}\label{claim:bound-on-Wlm}
    For every $m \le f \cdot \bigl(\log(1/p)\bigr)^2$ and every $\ell \ge 0$,
    \[
      |\mathcal{W}_{\ell, m}| \cdot p^m \le n^{-\sigma \ell}.
    \]
  \end{claim}
  \begin{proof}
    We prove the assertion by induction on $\ell$. The induction basis ($\ell=0$) is vacuously true, so let us assume that $\ell \ge 1$.
    Observe that, for every $(C_0, \dotsc, C_\ell) \in \mathcal{W}_{\ell,m}$, either:
    \begin{enumerate}[label=(\arabic*)]
    \item
      The graph $C_\ell$ was obtained from $C_{\ell-1}$ as in \cref{def:tight-cores}~\ref{item:tight-core-one-copy}.
      More precisely, there are $J' \subseteq J \in \cJ^*$ with $J' \notin \cJ_\emptyset^*$ and $\varphi \in \Emb(J, K_n)$ such that $\varphi(J) \cap C_{\ell-1} = \varphi(J')$ and $C_\ell = C_{\ell-1} \cup \varphi(J)$.
      By \cref{lem:tight-core-extension}, for each such pair $J', J$, there are at most $n^{-3\sigma} \cdot p^{e_{J'} - e_J}$ embeddings $\varphi$ with $\varphi(J) \cap C_{\ell-1} = \varphi(J')$.
      Finally, since $e(C_\ell) = e(C_{\ell-1}) + e_J - e_{J'}$, we have $(C_0, \dotsc, C_{\ell-1}) \in \mathcal{W}_{\ell-1,m-e_J+e_{J'}}$.
    \item
      The graph $C_\ell$ was obtained from $C_{\ell-1}$ as in \cref{def:tight-cores}~\ref{item:tight-core-two-copies}.
      More precisely, there are $J_1' \subseteq J_1 \in \cJ^*$ and $J_2' \subseteq J_2 \in \cJ^*$ satisfying $J_1' \in \cJ_\emptyset^*$ and $J_2' \notin \cJ_\emptyset^*$ as well as $\varphi_1 \in \Emb(J_1, K_n)$ and $\varphi_2 \in \Emb(J_2, K_n)$ such that $\varphi_1(J_1) \cap C_{\ell-1} = \varphi_1(J_1')$, $\varphi_2(J_2) \cap (C_{\ell-1} \cup \varphi_1(J_1)) = \varphi_2(J_2')$, and $C_\ell = C_{\ell-1} \cup \varphi_1(J_1) \cup \varphi_2(J_2)$.
      By \cref{lem:tight-core-extension}, for each choice of $J_1', J_1, J_2', J_2$, there are at most $n^\sigma \cdot p^{e_{J_1'} - e_{J_1}}$ embeddings $\varphi_1$ with $\varphi_1(J_1) \cap C_{\ell-1} = \varphi_1(J_1')$ and then at most $n^{-3\sigma} \cdot p^{e_{J_2'} - e_{J_2}}$ further choices for $\varphi_2$ satisfying $\varphi_2(J_2) \cap (C_{\ell-1} \cup \varphi_1(J_1)) = \varphi_2(J_2')$.
      Finally, since $e(C_\ell) = e(C_{\ell-1}) + e_{J_1} + e_{J_2} - e_{J_1'} - e_{J_2'}$, we have $(C_0, \dotsc, C_{\ell-1}) \in \mathcal{W}_{\ell-1, m-e_{J_1}-e_{J_2}+e_{J_1'}+e_{J_2'}}$.
    \end{enumerate}
    By our inductive assumption, there is a constant $K_H$ such that
    \[
      \begin{split}
        |\mathcal{W}_{\ell, m}| \cdot p^m
        & \le \sum_{\substack{J' \subseteq J \in \cJ^* \\ J' \notin \cJ_\emptyset^*}} |\mathcal{W}_{\ell-1,m-e_J+e_{J'}}| \cdot p^{m-e_J+e_{J'}} \cdot n^{-3\sigma} \\
        & \;+ \sum_{\substack{J_1' \subseteq J_1 \in \cJ^* \\ J_2' \subseteq J_2 \in \cJ^* \\ J_2' \notin \cJ_\emptyset^*}} |\mathcal{W}_{\ell-1,m-e_{J_1}-e_{J_2}+e_{J_1'}+e_{J_2'}}| \cdot p^{m-e_{J_1}-e_{J_2}+e_{J_1'}+e_{J_2'}} \cdot n^{\sigma-3\sigma} \\
        & \le K_H \cdot n^{-2\sigma} \cdot n^{-(\ell-1)\sigma} \le n^{-\ell\sigma},
      \end{split}
    \]
    as claimed.
  \end{proof}
  Finally, substituting the bound from \cref{claim:bound-on-Wlm} into \eqref{eq:large-tight-core-is-expensive}, we obtain
  \[
    \Pr\bigl(G_{n,p}\in \langle \cC \setminus \cC_{m_c} \rangle\bigr) \le 2e_Hm_c \cdot n^{-\sigma m_c / (2e_H)} \le p^{\eta m_c},
  \]
  provided that $\eta$ is sufficiently small (as a function of $\sigma$ and $H$ only), since $p \ge p_H \ge n^{-2}$.
\end{proof}

We now turn to the proof of \cref{lem:seeds-contains-a-small-tight-core-2}.
The heart of the matter here is the following upper bound on $N(J,G;J_0,G^*)$ in terms of the quantities $\dim_{J_0} J$ and $\alpha_{J_0-R_{J_0}(J)}^*$, which appear in the definition of $C_H$.

\begin{lemma}
  \label{lem:counting-extensions-via-dimension}
  Suppose that $G^{*}\subseteq G$ is a maximal tight core.
  For all $J_0,J\in \cJ_\emptyset^*$ with $J_0\subsetneq J$,
  \[
    N(J,G;J_0,G^*) \leq  (2e_{G^*})^{\alpha_{J_0-R_{J_{0}}(J)}^*} \cdot (2e_G)^{\dim_{J_0}J}.
  \]
\end{lemma}

\begin{proof}
  Suppose that $G^*$ is a maximal tight core in $G$.
  Let $E \subseteq J$ be an arbitrary basis of $J$ over $J_0$ and denote $R \coloneqq R_{J_0}(J)$.
  We claim that it suffices to show that two embeddings $\varphi_1, \varphi_2$ of $J$ into $G$ that agree on $E$ and on $J_0 - R$ and satisfy $\varphi_i(J) \cap G^* = \varphi_i(J_0)$ for both $i$ must satisfy $\varphi_1(J) = \varphi_2(J)$.
  Indeed, this means that every embedding $\varphi$ of $J$ into $G$ that satisfies $\varphi(J) \cap G^* = \varphi(J_0)$ is determined by the images of $J_0 - R$ and of $E$, implying that
  \[
    N(J, G; J_0, G^*) \le |\Emb(J_0-R, G^*)| \cdot (2e_G)^{|E|} \le (2e_{G^*})^{\alpha_{J_0-R}^*} \cdot (2e_G)^{\dim_{J_0}J}.
  \]
  Let $\varphi_1, \varphi_2$ be two embeddings with the above properties.
  Since $G^*$ is a maximal tight core, the graph $J' \subseteq J$ defined by $\varphi_2(J') \coloneqq \varphi_2(J) \cap (G^* \cup \varphi_1(J))$ belongs to $\cJ_\emptyset^*$, see \cref{def:tight-cores}~\ref{item:tight-core-two-copies}.
  However, $J' \supseteq J_0 \cup E$, as $\varphi_1$ and $\varphi_2$ agree on $E$ and $\varphi_2(J) \cap G^* = \varphi_2(J_0)$, and thus $J' = J$.
  Since $\varphi_i(J) \cap G^* = \varphi_i(J_0)$ for both $i$, it follows that $\varphi_1(J \setminus J_0) = \varphi_2(J \setminus J_0)$.
  Finally, since every vertex of $J$ has nonzero degree in either $J \setminus J_0$ or $J_0 - R$, we conclude that $\varphi_1(J) = \varphi_2(J)$.
\end{proof}

\begin{proof}[Proof of~\cref{lem:seeds-contains-a-small-tight-core-2}]
  Suppose that $G^* \subseteq G$ satisfy the assumptions of the proposition, fix an arbitrary pair $J_0, J \in \cJ_\emptyset^*$ with $J_0 \subsetneq J$, and let $R\coloneqq R_{J_{0}}(J)$ and $d\coloneqq \dim_{J_0}(J)$.
  By \cref{lem:counting-extensions-via-dimension}, 
  \begin{equation}
    \label{eq:bound-on-N-J-G-J0-G*-part-1}
    N(J, G; J_0, G^*)
    \le  (2e_{G^*})^{\alpha_{J_0- R}^*} \cdot (2e_G)^{d} \le (2Df)^{\alpha_{J_0- R}^*} \cdot \bigl(2Df \log(1/p)\bigr)^{d}.
  \end{equation}
  Let $x$ be the number satisfying $p = x \cdot p_H (\log n)^{C_H}$ and note that
  \[
    f = n(np^\Delta)^r = x^{\Delta r} \cdot (\log n)^{C_H \Delta r} \ge x^{\Delta r} \cdot \bigl(\log(1/p)/2\bigr)^{C_H \Delta r},
  \]
  where the final inequality holds as $p \ge p_H \ge n^{-2}$.
  Substituting the resulting upper bound on $\log(1/p)$ into~\eqref{eq:bound-on-N-J-G-J0-G*-part-1} yields, for some $D' = D'(D,H)$,
  \begin{equation}
    \label{eq:bound-on-N-J-G-J0-G*-in-xf}
    N(J,G;J_0,G^*) \le D' \cdot x^{-d/C_H} \cdot f^{\alpha_{J_0-R}^* + d + d/(C_H\Delta r)}.
  \end{equation}
  The key observation is that, by the definition of $C_H$,
  \[
    \alpha_{J_0-R}^* + d + d/(C_H\Delta r) \le \alpha_J^*
  \]
  and equality holds if and only if $(J_0, J)$ is an optimal pair.
  Since
  \[
    f^{\alpha_J^*} = n^{v_J} p^{e_J} \ge (p/p_H)^{e_J} \ge (\log n)^{C_H/2},
  \]
  by \cref{fact:f-p-pH}, and $x^{-1} \le (\log n)^{1/L}$ for some large $L = L(H)$,
  we may already conclude from~\eqref{eq:bound-on-N-J-G-J0-G*-in-xf} that $N(J,G;J_0,G^*) < c n^{v_J} p^{e_J}$, unless $(J_0, J)$ is an optimal pair.
  In the latter case, \eqref{eq:bound-on-N-J-G-J0-G*-in-xf} becomes
  \[
    N(J, G; J_0, G^*) \le D' \cdot x^{-d/C_H} \cdot n^{v_J} p^{e_J}
  \]
  and the right-hand side is smaller than $c n^{v_J} p^{e_J}$ unless $x < L$ for some constant $L$ that depends only on $c$, $D$, and $H$.
\end{proof}

\subsection{The upper bound on the upper-tail probability}
\label{sec:proof-upper-near-threshold}

Let $L = K(H, \delta, \varepsilon)$ and $D = D(H, \delta, \varepsilon)$ be large constants and assume that either $p \ge p_H \cdot L(\log n)^{C_H}$, as in the setting of \cref{thm:main-near-appearance-threshold}~\ref{item:main-near-appearance-threshold-UB}, or that $p \ge p_H \cdot L(\log n / \log \log n)^{C_H}$ and $H$ is one-dimensional and not clean (see \cref{def:1-dim-graphs,def:clean-graphs}), as in the setting of \cref{thm:non-clean}.

Write
\[
  t \coloneqq (\delta-\varepsilon) \cdot \Ex[X],
  \qquad
  t' \coloneqq (\delta-2\varepsilon) \cdot \Ex[X],
  \qquad
  \text{and}
  \qquad
  t'' \coloneqq (\delta-3\varepsilon) \cdot \Ex[X]
\]
and define, for every positive integer $m$,
\[
  \cL_m\coloneqq \{G \in \cS_{t'',m} \cap \cL:e_G=m\}
  \qquad
  \text{and}
  \qquad
  v_m \coloneqq \max\{v_G : G \in \mathcal{L}_m\}.
\]
Further, let $m_0 \coloneqq \min\{m:\cL_m\neq \emptyset\}$ and $m_1 \coloneqq \lceil D \psi_H(\delta+\varepsilon) \rceil$ and observe that \cref{lem:psi--psi*-lower} implies that $f/D \le m_0 \le Df$ and $m_1 \le D^2 f \log(1/p)$, provided that $D$ is sufficiently large.
Finally, recall that $\mathcal{C}$ is the family of all tight cores and $\mathcal{C}_m$ is the family of all tight cores with at most $m$ edges.

It follows from \cref{cor:exists-a-seed-condition-on-UT,lem:seeds-contains-a-small-tight-core} with $m_c \coloneqq D^2f$ that
\begin{equation}
  \label{eq:UT-stable-bound-by-cores}
  \begin{split}
    \Pr(\UT_{H,\delta})
    & \le 2 \cdot \Pr\bigl(G_{n,p} \in \langle\cS_{t,m_1}\rangle \bigr) \\
    & \leq 2 \cdot \left(\Pr\bigl(G_{n,p} \in  \langle\cC \setminus \cC_{m_c}\rangle\bigr) + \Pr\bigl(G_{n,p} \in \langle \cS_{t', m_c}\rangle\bigr) + \Pr(\mathcal{E})\right),
  \end{split}
\end{equation}
where $\mathcal{E}$ is the event that $G_{n,p}$ contains a subgraph $G$ with at most $m_1$ edges and a maximal tight core $G^* \subseteq G$ with at most $m_c$ edges such that $N(J,G;J_0,G^*) \ge c n^{v_J} p^{e_J}$ for some pair $J_0, J \in \cJ_\emptyset^*$ with $J_0 \subsetneq J$.
We now bound each of the summands in the right-hand side of~\eqref{eq:UT-stable-bound-by-cores}.

First, by \cref{lem:large-tight-core-is-expensive,lem:psi--psi*-lower},
\begin{equation}\label{eq:too-large-of-a-tightcore}
  \Pr\bigl(G_{n,p} \in  \langle\cC \setminus \cC_{m_c}\rangle\bigr) \le p^{\eta m_c} \le p^{Df} \ll e^{-\psi_H^*(\delta)},
\end{equation}
provided that $D$ is sufficiently large.

Second, \Cref{lem:cores}, with $(t,t')$ replaced by $(t', t'')$, yields
\begin{equation}
  \label{eq:seed-hubcore-bound}
  \begin{split}
    \Pr\bigl(G_{n,p} \in \langle \cS_{t', m_c}\rangle\bigr)
    & \le \Pr\bigl(G_{n,p} \in \langle \cS_{t'', m_c} \cap \mathcal{L}\rangle\bigr) \\
    & \le \sum_{m=m_0}^{m_c} |\cL_m|\cdot  p^{m} \le  n^2 \cdot \max \bigl\{|\cL_m|\cdot p^{m}: m_0 \le m \le m_c\bigr\}.
  \end{split}
\end{equation}
Further, applying \cref{lem:core-count} with $\cL'=\cL_m$ for each $m \in [m_0, m_c]$, we get
\[
  \begin{split}
    \log \bigl(|\cL_m| \cdot p^m\bigr)
    & \le v_m\log(en/g) - m \log(1/p) + O(m \log \log(1/p)) \\
    & \le -(1-\varepsilon) \cdot \bigl(m\log (1/p) - v_m \log(en/g)\bigr),
  \end{split}
\]
where the last inequality follows from~\cref{lem:weak-LB-psi-*}, as $H$ is stable and $p \gg p_H$.
Substituting this into~\eqref{eq:seed-hubcore-bound} yields
\[
  \begin{split}
    \log \Pr\bigl(G_{n,p} \in \langle \cS_{t', m_c}\rangle\bigr)
    & \le - (1-\varepsilon) \cdot \min\bigl\{m \log(1/p) - v_m \log(en/g) : m_0 \le m \le m_1 \bigr\} + 2 \log n \\
    & \le - (1-\varepsilon) \cdot \psi_H^*(\delta-3\varepsilon) + 2\log n \le -(1-2\varepsilon) \cdot \psi_H^*(\delta-3\varepsilon) - \log 2,
  \end{split}
\]
where the last inequality follows as, by \Cref{lem:psi--psi*-lower},
\[
  \psi_H^*(\delta-3\varepsilon) = \Omega\bigl(f \log(1/p)\bigr) \ge \Omega\bigl( p/p_H \cdot \log(1/p) \bigr) \gg \log n.
\]

Finally, we bound the probability of $\mathcal{E}$.
Given all the preparations we have made so far, this is now very easy in the setting of \cref{thm:main-near-appearance-threshold}~\ref{item:main-near-appearance-threshold-UB}.
Indeed, if $p \ge L p_H (\log n)^{C_H}$ for a sufficiently large constant $L$, then \cref{lem:seeds-contains-a-small-tight-core-2} implies that $\Pr(\mathcal{E}) = 0$, which yields the desired estimate
\[
  \Pr(\UT_{H,\delta}) \le \exp\bigl(-(1-2\varepsilon) \cdot \psi_H^*(\delta-3 \varepsilon)\bigr) + o\bigl(e^{-\psi_H^*(\delta)}\bigr).
\]

We are now left with the case where $L p_H (\log n / \log \log n)^{C_H}\le p < L p_H (\log n)^{C_H}$ and $H$ is a one-dimensional graph that is not clean, as in \cref{thm:non-clean}.
By \cref{lem:seeds-contains-a-small-tight-core-2}, we may also assume that the pair $(J_0, J)$ appearing in the definition of $\mathcal{E}$ is optimal, as otherwise $\Pr(\mathcal{E}) = 0$.

Fix two graphs $G^*\subseteq G$ with $e_{G^*} \le m_c$ and $e_G \le m_1$ and an optimal pair $(J_0,J)$ such that
\[
  N(J,G;J_0,G^*) \geq c n^{v_J} p^{e_J} = cf^{\alpha_{J}^*},
\]
see~\cref{fact:f-p-pH}.
Our argument will exploit the special structure of the family
\[
  \mathcal{P} \coloneqq \bigl\{\varphi(J \setminus J_0) : \varphi \in \Emb(J,G; J_0,G^*)\bigr\}
\]
that is a consequence of our assumption that $H$ is one-dimensional and hence $\dim_{J_0}J = 1$; here and throughout, $\Emb(J, G; J_0, G^*)$ denotes the family of embeddings of $J$ in $G$ whose image intersects $G^*$ on a subgraph that is isomorphic to $J_0$.

\begin{claim}\label{claim:petals-disjoint}
  The graphs in $\mathcal{P}$ are pairwise edge-disjoint.
\end{claim}
\begin{proof}
  Suppose that two distinct $P_1, P_2 \in \mathcal{P}$ intersect and let $\varphi \in \Emb(J,G;J_0,G^*)$ be some embedding satisfying $\varphi(J \setminus J_0) = P_1$.
  Since $G^*$ is a maximal tight core, the graph $J_0' \coloneqq J_0 \cup \varphi^{-1}(P_1 \cap P_2)$ belongs to $\cJ_\emptyset^*$ and satisfies $J_0 \subsetneq J_0' \subsetneq J$, see~\cref{def:tight-cores}~\ref{item:tight-core-two-copies}.
  In particular, letting $R$ and $R'$ denote the roots of $J$ over $J_0$ and $J_0'$, respectively, \cref{claim:J0-in-J'0-not-maximal} would imply that $\alpha_{J_0'-R'}^* = \alpha_{J_0' - R'} > \alpha_{J_0 - R} = \alpha_{J_0-R}^*$ and thus, using that $\dim_{J_0'} J \ge 1 = \dim_{J_0} J$,
  \[
    \frac{\dim_{J_0'} J}{\alpha_{J}^* - \alpha_{J_0' - R'}^* - \dim_{J_0'}J}
    > \frac{\dim_{J_0} J}{\alpha_{J}^* - \alpha_{J_0 - R}^* - \dim_{J_0} J} = r\Delta C_H,
  \]
  which contradicts the assumption that $(J_0, J)$ is an optimal pair.
\end{proof}

Given a $P \in \mathcal{P}$, let $S_P \subseteq V(G^*)$ be the image of $R \coloneqq R_{J_0}(J) \subseteq V(J \setminus J_0) \cap V(J_0)$ in $P$.
Further, for every $S \subseteq V(G^*)$, let $\Emb(J_0, G^*; S)$ denote the family of all embeddings of $J_0$ into $G^*$ that map $R$ to $S$ and observe that
\begin{equation}
  \label{eq:Emb-SP}
  cf^{\alpha_J^*} \le N(J,G;J_0,G^*) \le \sum_{P \in \mathcal{P}} |\Emb(J_0, G^*; S_P)|
\end{equation}
and that, for every $S$, we have
\begin{equation}
  \label{eq:Emb-S-max}
  |\Emb(J_0, G^*; S)| \le |\Emb(J_0 - R, G^*)| \le (2e_{G^*})^{\alpha_{J_0-R}^*} \le (2m_c)^{\alpha_{J_0-R}^*} \le (2D^2f)^{\alpha_{J_0-R}^*}.
\end{equation}
Now, fix a small $\zeta>0$ and define
\[
  \mathcal{G} = \mathcal{G}(G^*) \coloneqq \bigl\{S \subseteq V(G^*) : |\Emb(J_0,G^*; S)| > f^{\alpha_{J_0- R}^*-\zeta}\bigr\}.
\]
It follows from \cref{claim:petals-disjoint}, \eqref{eq:Emb-SP}, and \eqref{eq:Emb-S-max} that
\begin{equation}
  \label{eq:good-bad-roots}
  |\{P \in \mathcal{P} : S_P \in \mathcal{G}\}| \ge \frac{cf^{\alpha_J^*} - e_G \cdot f^{\alpha_{J_0-R}^*-\zeta}}{(2D^2f)^{\alpha_{J_0-R}^*}}.
\end{equation}
Now, recalling that $p = x \cdot p_H(\log n)^{C_H}$, as in the proof of \cref{lem:seeds-contains-a-small-tight-core-2}, we have
\[
  f = n(np^\Delta)^r = x^{\Delta r} \cdot (\log n)^{C_H \Delta r} \ge x^{\Delta r} \cdot \bigl(\log(1/p)/2\bigr)^{C_H \Delta r}
\]
and thus, recalling the definition of $C_H$,
\begin{equation}
  \label{eq:f-log-p-f}
  e_G \le m_1 \le D^2 f\log(1/p) \le D^2f \cdot x^{-1/C_H} f^{1/(C_H \Delta r)} = D^2x^{-1/C_H} \cdot f^{\alpha_J^* - \alpha_{J_0-R}^*}.
\end{equation}
Since we have assumed that $x \ge (\log \log n)^{-C_H}$, we may substitute the above estimate into~\eqref{eq:good-bad-roots} and conclude that, for some positive $c'$ that depends only on $c$ and $D$,
\begin{equation}
  \label{eq:good-roots}
  |\{P \in \mathcal{P} : S_P \in \mathcal{G}\}| \ge c' f^{\alpha_J^* - \alpha_{J_0-R}^*}.
\end{equation}

To summarise, the event $\mathcal{E}$ implies the following event $\mathcal{E}'$:
For some optimal pair $(J_0, J)$, the random graph $G_{n,p}$ contains a tight core $G^*$ with at most $m_c$ edges and a collection $\mathcal{P}$ of at least $T \coloneqq c' f^{\alpha_J^* - \alpha_{J_0-R}^*}$ edge-disjoint copies of $J \setminus J_0$ rooted at the sets in $\mathcal{G} = \mathcal{G}(G^*)$, as in \eqref{eq:good-roots}.
In particular, recalling that $e_{G^*} \le m_c \le D^2f$,
\begin{equation}
  \label{eq:PrE-non-clean}
  \Pr(\mathcal{E}) \le \Pr(\mathcal{E}')
  \le \binom{n^2}{D^2f} \cdot \max_{G^* : e_{G^*} \le D^2f}  \binom{|\mathcal{G}(G^*)| \cdot n^{v_J-v_{J_0}}}{T} \cdot p^{(e_J-e_{J_0}) \cdot T}.
\end{equation}
In order to estimate the right-hand side of~\eqref{eq:PrE-non-clean}, we will need a nontrivial upper bound on $|\mathcal{G}(G^*)|$.
This is where we crucially use the assumption that $H$ is not clean.
  
\begin{claim}
  \label{claim:entropy-of-roots}
  If $e_{G^*} \le D^2f$, then $|\mathcal{G}(G^*)| \le f^{\alpha_{J_0}^*-\alpha_{J_0-R}^*-\zeta}$, given that $\zeta > 0$ is sufficiently small.
\end{claim}

\begin{proof}
  Suppose that $m \coloneqq e_{G^*} \le D^2f$ and note that it suffices to show that the set
  \[
    \mathcal{G}' \coloneqq \bigl\{S \subseteq V(G^*) : |\Emb(J_0,G^*;S)| > (2m)^{\alpha_{J_0- R}^*-2\zeta}\bigr\}
  \]
  has at most $(2m)^{\alpha_{J_0}^* - \alpha_{J_0-R}^* - 2\zeta}$ elements.
  Let $\preceq$ be an arbitrary linear order on $V(J_0)$ such that $R \prec V(J_0) \setminus R$ and consider a random $\varphi \in \Emb(J_0, G^*)$ chosen in the following way:
  First, choose an $S \in \mathcal{G}'$ uniformly at random.
  Second, choose a uniformly random embedding of $J_0$ in $G^*$ that maps $R$ to $S$.
  Define $x \colon V(J_0) \to [0,\infty)$ by setting $x_v \coloneqq \ent\bigl(\varphi(v) \mid (\varphi(w) : w \prec v)\bigr)$.
  As in the proof of~\Cref{thm:max-copies}, for every $vw \in E(J_0)$ with $v \prec w$,
  \[
    x_v + x_w \le \ent\bigl(\varphi(v)\bigr) + \ent\bigl(\varphi(w) \mid \varphi(v)\bigr) = \ent\bigl(\varphi(v), \varphi(w)\bigr) \le \log (2m).
  \]
  By the definition of $\mathcal{G}'$ and the choice of $\varphi$, we have
  \[
    x(R) = \ent\bigl((\varphi(v) : v \in R)\bigr) = \log |\mathcal{G}'|
  \]
  and
  \[
    x(V(J_0)\setminus R) = \ent\bigl((\varphi(v) : v \in V(J_0) \setminus R) \mid (\varphi(w) : w \in R)\bigr) > (\alpha_{J_0- R}^*-2\zeta) \cdot \log(2m).
  \]
  Consequently,
  \[
    \log |\mathcal{G}'| = x(V(J_0)) - x(V(J_0) \setminus R) < x(V(J_0)) - (\alpha_{J_0-R}^* - 2\zeta) \cdot \log(2m).
  \]
  It thus suffices to show that $x(V(J_0)) \le (\alpha_{J_0}^* - 4\zeta) \cdot \log(2m)$.
  
  Suppose that this were not true for any $\zeta > 0$.
  Since the function $x / \log(2m)$ is a fractional independent set of $J_0$ and the space of all fractional independent sets is compact, by continuity, we would be able to find a fractional independent set $\alpha$ that simultaneously satisfies $\alpha(V(J_0)) = \alpha_{J_0}^*$ and $\alpha(V(J_0) \setminus R) = \alpha_{J_0 - R}^*$.
  \Cref{lem:unique-frac-ind-set} would then imply that $\alpha$ is the characteristic function of $B_0$ and thus $\alpha(R) = |R \cap B_0| = |R|$, contradicting the assumption that $H$ is not clean, see \cref{def:clean-graphs}.
\end{proof}

By~\eqref{eq:PrE-non-clean} and~\cref{claim:entropy-of-roots}, and using the fact that $n^{v_J-v_{J_0}}p^{e_J-e_{J_0}} = f^{\alpha_J^*-\alpha_{J_0}^*}$, see~\cref{fact:f-p-pH},
\[
  \begin{split}
    \Pr(\mathcal{E})
    & \le \exp\left(2D^2f \log n + c'f^{\alpha_J^*-\alpha_{J_0-R}^*} \cdot \log\bigl(e/c' \cdot f^{-\zeta}\bigr)\right) \\
    & \le \exp\left(2\Delta D^2f \log(1/p) - c'\zeta/2 \cdot f^{\alpha_J^*-\alpha_{J_0-R}^*}\log f\right)
  \end{split}
\]
Finally, as $f \log(1/p) \le x^{-1/C_H} \cdot f^{\alpha_J^*-\alpha_{J_0-R}^*}$, cf.~\eqref{eq:f-log-p-f}, if $x \ge L (\log \log n)^{-C_H} \ge L (\log f)^{-C_H}$ for a sufficiently large constant $L = L(\Delta, c', D, \zeta)$, we may conclude that
\[
  \Pr(\mathcal{E}) \le \exp\left(-Df \log (1/p)\right) \ll e^{-\psi_H^*(\delta)},
\]
by \Cref{lem:psi--psi*-lower} and assuming that $D$ is sufficiently large, cf.~\eqref{eq:too-large-of-a-tightcore}.
Substituting this estimate into~\eqref{eq:UT-stable-bound-by-cores} yields
\[
  \Pr(\UT_{H,\delta}) \le \exp\bigl(-(1-2\varepsilon) \cdot \psi_H^*(\delta-3 \varepsilon)\bigr) + o\bigl(e^{-\psi_H^*(\delta)}\bigr),
\]
as in the previous case.

\section{Lower bounds} \label{sec:lower-bounds}

In this section, we prove the lower bound on the probability of $\UT_{H,\delta}$ that is implicit in the assertion of \cref{thm:main-dense-ish}.
More precisely, we shall prove the following statement.

\begin{theorem}\label{thm:lower-bound}
  For every irregular, connected graph $H$ with maximum degree $\Delta \ge 2$ and all $\delta, \eps >0$, there is a constant $K$ such that the following holds.
  Suppose that either $p_H \cdot (\log n)^K \ll p \ll 1$ or $p_H \cdot (\log \log n)^K \ll p \ll 1$ and $H$ is stable.
  Then,
  \[
    \log \Pr(X\geq(1+\delta)\Ex[X]) \ge - (1+\eps)\cdot \psi_H^*(\delta+\eps),
  \]
  provided that $n$ is sufficiently large.
\end{theorem}

\begin{proof}
  Let $G \in \cL$ be a graph achieving the minimum in the definition of $\psi_H^*(\delta + \eps)$, that is,
  \[
    \psi_H^*(\delta+\varepsilon) = e_G \log(1/p) - v_G \log(en/g).
  \]
  Since $G$ is a $(\delta+\varepsilon)$-seed, \Cref{lem:psi-lower} implies that, for some positive constant $c$,
  \begin{equation}
    \label{eq:lower-bound-eG}
    e_G \ge \psi_H(\delta+\varepsilon) / \log(1/p) \ge cf .
  \end{equation}

  Let $D_G \subseteq V(G)$ be a smallest vertex cover, let $\nu_G \coloneqq |D_G|$, set $\mathcal{A} \coloneqq \binom{\br{n} \setminus \br{\nu_G}}{v_G - \nu_G}$, and note that
  \[
    |\mathcal{A}| = \binom{n-\nu_G}{v_G-\nu_G} \ge \left(\frac{n-\nu_G}{v_G-\nu_G}\right)^{v_G-\nu_G} \ge \left(\frac{n}{v_G}\right)^{v_G-\nu_G}.
  \]
  For every $A \in \mathcal{A}$, let $\varphi_A \colon V(G) \to \br{n}$ be an arbitrarily chosen embedding of $G$ into $K_n$ that maps $D_G$ onto $\br{\nu_G}$ and $V(G) \setminus D_G$ onto $A$.
  The crucial observation is that, for any event $\mathcal{T}$,
  \begin{equation}
    \label{eq:Pr-UT-lower-M}
    \sum_{A \in \mathcal{A}} \Pr\bigl(\mathcal{T} \cap \{\varphi_A(G) \subseteq G_{n,p}\} \cap \UT_{H,\delta}\bigr) \le M \cdot \Pr(\UT_{H, \delta}),
  \end{equation}
  where
  \[
    M = M(\mathcal{T}) \coloneqq \max\bigl\{|\{A \in \mathcal{A} : \varphi_A(G) \subseteq \Gamma\}| : \Gamma \in \mathcal{T}\bigr\}.
  \]
  Further, for every $A \in \mathcal{A}$, we have
  \begin{multline}
    \label{eq:PrT-EmbA-UT-lower}
    \Pr\bigl(\mathcal{T} \cap \{\varphi_A(G) \subseteq G_{n,p}\} \cap \UT_{H,\delta}\bigr) \\
    \ge \Pr(\varphi_A(G) \subseteq G_{n,p}) \cdot \bigl(\Pr(\UT_{H,\delta} \mid \varphi_A(G) \subseteq G_{n,p}) - \Pr(\mathcal{T}^c \mid \varphi_A(G) \subseteq G_{n,p})\bigr),
  \end{multline}
  where, by \Cref{lem:lower},
  \begin{equation}
    \label{eq:PrT-EmbA-UT-lower-details}
    \Pr(\varphi_A(G) \subseteq G_{n,p}) = p^{e_G}
    \qquad
    \text{and}
    \qquad
    \Pr(\UT_{H,\delta} \mid \varphi_A(G) \subseteq G_{n,p}) \ge \eps p^{e_H}.
  \end{equation}
  To utilise these estimates, we require an appropriate choice of the event $\mathcal{T}$ that leads to good upper bounds on both $M$ and on the conditional probability of $\mathcal{T}^c$.
  We consider two cases, depending on whether or not  $p \le n^{-1/\Delta} (\log n)^{-\omega}$.

  \medskip
  \noindent
  \textit{Case 1:}  $p > n^{-1/\Delta} (\log n)^{-\omega}$. \\
  We let $\mathcal{T}$ be the sure event, so that $\Pr(\mathcal{T}^c) = 0$, and observe the trivial bound
  \[
    M \le |\mathcal{A}| \le \binom{n}{v_G} \le \left(\frac{en}{v_G}\right)^{v_G} \le \left(\frac{en}{\min\{n,2e_G\}}\right)^{2e_G},
  \]
  where the final inequality follows from the trivial inequality $v_G \leq 2e_G$ and the fact that $(en/x)^x$ is increasing on $[0,n]$.
  By~\eqref{eq:lower-bound-eG} and the case assumption,
  \[
    e_G \ge cf \ge c n \cdot \min\bigl\{1, (np^\Delta)^r\bigr\} \ge cn \cdot \bigl(\Delta \log(1/p)\bigr)^{-\Delta r \omega}.
  \]
  Consequently, for some constant $C'$,
  \[
    \log M \le C' \omega e_G \cdot \log \log(1/p) \ll e_G \log(1/p) - v_G \log(en/g) = \psi_H^*(\delta + \varepsilon),
  \]  
  where the final inequality is due to \Cref{lem:weak-LB-psi-*} and the fact that $p_H \le n^{-1/\Delta - \Omega(1)}$.

  \medskip
  \noindent
  \textit{Case 2:} $p \le n^{-1/\Delta} (\log n)^{-\omega}$. \\
  We let $\mathcal{T}$ be the event that all but at most $v_G + f$ vertices of $G_{n,p}$ have fewer than $d_H$ neighbours in $\br{\nu_G}$.
  Since $G$ is a hub-core, each vertex of $G$ outside of $D_G$ has at least $\delta(G) \ge d_H$ neighbours in $D_G$.
  Therefore, for every $\Gamma \subseteq K_n$,
  \[
    \log_2 M \le \max_{\Gamma \in \mathcal{T}} |\{v \in \br{n} \setminus \br{\nu_G} : \deg_\Gamma(v, \br{\nu_G}) \ge d_H\}| \le v_G + f \le (2+1/c) \cdot e_G,
  \]
  where the last inequality is~\eqref{eq:lower-bound-eG}.  By \Cref{lem:weak-LB-psi-*}, we again have $\log M \ll \psi_H^*(\delta + \varepsilon)$.
  Further, recall from \cref{dfn:hub-core} that
  \[
    \nu_G \le \bigl(\log \log (1/p)\bigr)^3 \cdot \max\{1, (e_G/f)^{v_H}\} \cdot \max\{1, e_G/n\} \le (\log n)^{2v_H+3},
  \]
  since $e_G \le f \cdot (\log(1/p))^2 \ll n$ and $p \ge p_H \ge n^{-2}$.
  
  \begin{claim}
    For every $A \in \mathcal{A}$, we have $\Pr(\mathcal{T}^c \mid \varphi_A(G) \subseteq G_{n,p}) \ll p^{e_H}$.
  \end{claim}
  \begin{proof}
    For every vertex $v \in \br{n} \setminus \varphi_A(V(G))$, denote by $\mathcal{B}_v$ the event that $v$ has at least $d_H$ neighbours in $\br{\nu_G}$.
    Observe that, for some positive constant $C = C(H)$,
    \[
      \Pr(\mathcal{B}_v \mid \varphi_A(G) \subseteq G_{n,p}) = \Pr(\mathcal{B}_v) \le \binom{\nu_G}{d_H} p^{d_H} \le (e\nu_Gp)^{d_H} \le p^{d_H} \cdot (\log n)^C
    \]
    and that these events are independent, also after we condition on the event $\{\varphi_A(G) \subseteq G_{n,p}\}$.
    It follows that
    \[
      \begin{split}
        \Pr(\mathcal{T}^c \mid \varphi_A(G) \subseteq G_{n,p})
        & \le \Pr( \mathcal{B}_v \text{ for at least $f$ vertices $v \notin \varphi_A(V(G))$} \mid \varphi_A(G) \subseteq G_{n,p}) \\
        & \le \binom{n}{f} \cdot \Pr(\mathcal{B}_v)^f \le \left(\frac{enp^{d_H}}{f} \cdot (\log n)^C\right)^f \eqqcolon U^f.
      \end{split}
    \]
    If $p \gg p_H \cdot (\log n)^K$, then $U \ll 1$ and $f \geq p/p_H \gg \log n$, by \Cref{lem:f-npd} and~\cref{fact:f-p-pH}, respectively. If we instead assume that $p \gg p_H \cdot (\log \log n)^K$ and $H$ is stable, the same pair of observations imply that  $U \le n^{-\Omega(1)}$ and $f \ge p/p_H \gg 1$.
    In both cases, $U^f \ll p^{e_H}$, concluding the proof of the claim.
  \end{proof}

  To summarise, in both cases we have
  \[
    \log M \ll \psi_H^*(\delta+\varepsilon)
    \qquad
    \text{and}
    \qquad
    \Pr(\mathcal{T}^c \mid \varphi_A(G) \subseteq G_{n,p}) \ll p^{e_H}
  \]
  for every $A \in \mathcal{A}$.
  We further claim that $\log(1/p) \ll \psi_H^*(\delta+\varepsilon)$.
  Indeed, by \cref{lem:weak-LB-psi-*} and our assumption on $p$, either $\psi_H^*(\delta+\varepsilon) \gg f \gg \log(1/p)$ or $H$ is stable and $\psi_H^*(\delta+\varepsilon) = \Omega(f \log(1/p)) \gg \log(1/p)$.
  Substituting \eqref{eq:PrT-EmbA-UT-lower-details} into \eqref{eq:PrT-EmbA-UT-lower} and then into \eqref{eq:Pr-UT-lower-M}, we obtain
  \begin{equation}
    \label{eq:PrUT-lower-with-errors}
    \begin{split}
      \log \Pr(\UT_{H,\delta})
      & \ge - e_G \log(1/p) + \log |\mathcal{A}| - \log M - \log(2/\varepsilon) - e_H \log(1/p) \\
      & \ge - e_G \log(1/p) + (v_G - \nu_G) \cdot \log (n/v_G) - o\bigl(\psi_H^*(\delta+\varepsilon)\bigr) \\
      & = - (1+o(1)) \cdot \psi_H^*(\delta+\varepsilon) - v_G \log(ev_G/g) - \nu_G \log(n/v_G).
    \end{split}
  \end{equation}
  Finally, we bound the two error terms in the right-hand side of~\eqref{eq:PrUT-lower-with-errors}.

  \medskip
  \noindent
  \textit{Case 1:} $p \ge p_H \cdot (\log n)^K$.\\
  \Cref{lem:weak-LB-psi-*,lem:weak_estimation_for_Phi} imply that, for some constant $D$,
  \begin{equation}
    \label{eq:Pr-LB-eG-upper}
    (K/D) \cdot e_G \log\log(1/p) \le \psi_H^*(\delta+\varepsilon) \le D f \log(1/p),
   \end{equation}
   and thus, recalling that $g = \min\{f, n\}$,
  \[
    v_G \log(ev_G/g) \le 2e_G \cdot \max\{1, \log(6e_G / f)\} \le 3e_G\log\log(1/p) \le (\varepsilon/2) \cdot \psi_H^*(\delta+\varepsilon),
  \]
  since $K$ is sufficiently large.
  Further, since $G \in \mathcal{L}$,
  \[
    \nu_G \le \bigl(\log \log (1/p)\bigr)^3 \cdot \max\{1, (e_G/f)^{v_H}\} \cdot \max\{1, e_G/n\} \le (\log n)^{2v_H+3},
  \]
  where the last inequality follows as $e_G \le f (\log(1/p))^2 \ll n$ and $p \ge p_H \ge n^{-2}$.
  Since $f \ge p/p_H \ge (\log n)^K$, by~\cref{fact:f-p-pH}, we conclude that $\nu_G \cdot \log (n/v_G) \ll cf \le e_G \ll \psi_H^*(\delta+\varepsilon)$, see~\eqref{eq:lower-bound-eG}.

  \medskip
  \noindent
  \textit{Case 2:} $p \ge p_H \cdot (\log \log n)^K$ and that $H$ is stable.\\
  In this case, \Cref{lem:weak-LB-psi-*,lem:weak_estimation_for_Phi} imply that, for some constant $D$,
  \begin{equation}
    \label{eq:Pr-LB-eG-upper-stable}
    D^{-1} \cdot e_G \log(1/p) \le \psi_H^*(\delta+\varepsilon) \le D f \log(1/p) \ll n,
  \end{equation}
  and thus, recalling that $g = \min\{f, n\}$,
  \[
    v_G \log(ev_G/g) \le 2e_G \cdot \max\{1, \log(6e_G / f)\} \le 4e_G \log (6D) \ll \psi_H^*(\delta+\varepsilon).
  \]
  Further, since $G \in \mathcal{L}$,
  \[
    \nu_G \le \bigl(\log \log (1/p)\bigr)^3 \cdot \max\{1, (e_G/f)^{v_H}\} \cdot \max\{1, e_G/n\} \le (\log \log n)^4,
  \]
  where the last inequality follows from~\eqref{eq:Pr-LB-eG-upper-stable} and the inequality $p \ge p_H \ge n^{-2}$.
  Since $f \ge p/p_H \ge (\log \log n)^K$, by~\cref{fact:f-p-pH}, we conclude that $\nu_G \cdot \log (n/v_G) \ll f \log(1/p) \le D\psi_H^*(\delta+\varepsilon)$.
\end{proof}

\section{Lower bounds near the appearance threshold}
\label{sec:LB-near-the-appearance-threshold}

In this section, we prove lower bounds on the upper-tail probability that are implicit in \cref{thm:main-near-appearance-threshold}~\ref{item:main-near-appearance-threshold-LB} and \cref{thm:clean}.
More precisely, we shall prove the following theorem, which implies that the approximation $\log \Pr(\UT_{H,\delta}) \approx -\psi_H^*(\delta)$ is invalid, as $\psi_H^*(\delta) = \Theta_\delta(f \log (1/p))$ for all $p \gg p_H$ whenever $H$ is stable, see \cref{lem:weak-LB-psi-*}~\ref{item:weak-LB-psi-*-stable}.

\begin{theorem}
  \label{thm:Poisson-type-lower-bound}
  For every stable graph $H$ and all $\delta, \eps>0$, there is a $C>0$ such that the following holds.
  Suppose that $p = n^{-1/m(H)}(x\log n)^{C_H}$ for some $x \gg 1/\log n$.
  If either $x \leq (C\log\log n)^{-1}$ or $x \le C^{-1}$ and $H$ is clean, then
  \[
    \log \Pr(\UT_{H,\delta}) \ge - \eps f \log (1/p).
  \]
\end{theorem}

Our proof of \cref{thm:Poisson-type-lower-bound} considers another strategy of triggering the upper tail event that is `cheaper' than planting a hub-core achieving the minimum in the definition of $\psi_H^*(\delta)$.
This strategy is defined in terms of an optimal pair $J_0, J \in \cJ_\emptyset^*$ (in the sense of \cref{dfn:optimal-pair}) and its fan decomposition $J_1, \dotsc, J_d$.
More precisely, we shall show that $\UT_{H,\delta}$ contains a substantial proportion of the intersection of the sequence $\mathcal{E}^*, \mathcal{E}_1, \dotsc, \mathcal{E}_d$ of events defined as follows:
\begin{enumerate}[label=(\roman*)]
\item
  $\mathcal{E}^*$ is the event that $G_{n,p}$ contains a certain blowup $G^*$ of $J_0$ with only $o(f)$ edges;
\item
  for each $i \in \br{d}$, $\mathcal{E}_i$ is the event that $G_{n,p}$ contains $o(f \log n)$ copies of $J_i$, each rooted in the appropriate vertices of $G^*$.
\end{enumerate}
More precisely, we shall show that the event $\mathcal{E}^* \cap \mathcal{E}_1 \cap \dotsb \cap \mathcal{E}_d$ implies that $X_J \ge \Gamma \cdot \Ex[X_J]$ for some large constant $\Gamma$, and this is likely to trigger the upper tail event $\UT_{H,\delta}$.
While the probability of $\mathcal{E}^*$ will be trivial to estimate, the probability of each of $\mathcal{E}_1, \dotsc, \mathcal{E}_d$ will be bounded from below using \Cref{lem:Poisson-lemma}.

\begin{proof}[Proof of~\cref{thm:Poisson-type-lower-bound}]
  Observe first that, since $H$ is stable, we have $m(H) = \Delta r / (1+r)$ and
  
  \begin{equation}
    \label{eq:evaluation-of-f}
    f=n(np^{\Delta})^r = (x \log n)^{r\Delta C_H} \gg 1.
  \end{equation}
  
  We will expose $G_{n,p}$ in two rounds.
  Fix a small constant $\eta=\eta(H,\eps)>0$, to be specified later, let $p_1 \coloneqq \eta p$, and let $p_2$ be such that $1-p = (1-p_1)(1-p_2)$; note that $p_2= (p-p_1)/(1-p_1) \ge (1-2\eta)p$.
  Independently sample $G_1\sim G_{n,p_1}$ and $G_2\sim G_{n,p_2}$, so that $G_1\cup G_2 \sim G_{n,p}$.

  Let $(J_0, J)$ be an optimal pair, let $d \coloneqq \dim_{J_0}J$, and let $J_1, \dotsc, J_d$ be its fan decomposition.
  If $H$ is clean, we further assume that $J_1, \dotsc, J_d$ have the properties described in \cref{def:clean-graphs}.
  For every $i \in \br{d}$, let $R_i$ be the set of roots of $J_i$ over $J_0$ and observe that $R\coloneqq \bigcup_{i=1}^d R_i$ is the set of roots of $J$ over $J_0$.

  We first describe the planting event associated with $J_0$.
  Fix a fractional independent set $a \colon V(J_0)\to [0,1]$ that, depending on whether or not $H$ is clean, satisfies:
  \begin{enumerate}[label=(\roman*)]
  \item
    If $H$ is not clean, then $a(V(J_0) \setminus R) = \alpha_{J_0 - R}^*$ and $a(R) = 0$.
  \item
    If $H$ is clean, then $a$ satisfies the properties listed in \cref{def:clean-graphs}.
  \end{enumerate}
  Fix a small positive constant $\gamma$
  and let $G^*$ be the graph obtained from $J_0$ by blowing up each $u \in V(J_0)$ to a set $V_u \subseteq \br{n}$ of size $\lfloor (\gamma f)^{a_u} \rfloor$, so that $\{V_u : u \in V(J_0)\}$ are pairwise disjoint, and replacing the edges of $J_0$ by complete bipartite graphs;
  note that $e_{G^*} \le e_{J_0} \cdot \gamma f$.
  Finally, let $\mathcal{E}^*$ be the event that $G^* \subseteq G_1$ and note that $\Pr(\mathcal{E}^*) = p_1^{e_{G^*}} \ge (\eta p)^{e_{J_0}\gamma f}$.
  
  Next, we define the Poisson-type events associated with $J_1, \dotsc, J_d$.
  Let $\Gamma = \Gamma(\gamma) > 0$ be a large constant, let
  \[
    t \coloneqq \Gamma x f\log n = \Gamma \cdot f^{1 + 1/(r\Delta C_H)}, 
  \]
  where the equality follows from~\eqref{eq:evaluation-of-f}, and let
  \[
    T \coloneqq (\gamma f/2)^{\alpha_{J_0-R}^*} \cdot t^d.
  \]
  \begin{claim}
    \label{claim:Poisson-LB-t-T-large}
    There is a constant $D$ such that $f \ll t \le (\log n)^D$ and $T \ge \Gamma/D \cdot n^{v_J}p^{e_J}$.
  \end{claim}
  \begin{proof}
    Let $D$ be a large constant.
    The definition of $t$ and \eqref{eq:evaluation-of-f} imply that
    
    \[
      f \ll t = \Gamma \cdot f^{1+1/(r\Delta C_H)} = \Gamma \cdot (x\log n)^{r\Delta C_H +1} \le (\log n)^D.
    \]
    
    Since $(J_0, J)$ is an optimal pair and $d = \dim_{J_0} J$, we have
    \[
      t^d = \Gamma^d \cdot f^{d + d/(r\Delta C_H)} = \Gamma^d \cdot f^{\alpha_J^* - \alpha_{J_0-R}^*}.
    \]
    Consequently, by \cref{fact:f-p-pH},
    \[
      T = \Gamma^d \cdot (\gamma/2)^{\alpha_{J_0-R}^*} \cdot f^{\alpha^*_J} \ge \Gamma/D \cdot n^{v_J}p^{e_J},
    \]
    provided that $D$ is sufficiently large.
  \end{proof}
  Fix some equipartition $U_1, \dotsc, U_d$ of $\br{n} \setminus V(G^*)$.
  Since $|V(G^*)| = O(f)$, we have $|U_i| \ge n/(2d)$ for each $i$.
  For every $i \in \br{d}$, let $\mathcal{B}_i$ be the set of copies of $J_i\setminus J_0$ such that every $u \in R_i$ is embedded into $V_u$  and every $v\in V(J_i)\setminus V(J_0)$ is embedded into $U_i$ and let $Y_i \coloneqq |\mathcal{B}_i[G_1]|$ denote the number of such copies that are contained in $G_1$.
  
  \begin{claim}
    \label{claim:Yi-moments}
    There is a constant $c$ that depends only on $H$, $\gamma$, and $\eta$ such that, for every $i \in \br{d}$,
    \[
      \Ex[Y_i] \ge c f^{\alpha_{J_i}^* - \alpha_{J_0}^* + a(R_i)}
      \qquad
      \text{and}
      \qquad
      \Ex[Y_i^2] \le \Ex[Y_i] \cdot f^{1/c}.
    \]
  \end{claim}
  \begin{proof}
    By construction,
    \[
      \Ex[Y_i]
      = |\mathcal{B}_i| \cdot p_1^{e_{J_i} - e_{J_0}} \ge (\gamma f / 2)^{a(R_i)} \cdot (n/(2d))^{v_{J_i} - v_{J_0}} \cdot (\eta p)^{e_{J_i}-e_{J_0}}
      \ge c f^{a(R_i) + \alpha_{J_i}^* - \alpha_{J_0}^*},
    \]
    where the last inequality follows from \cref{fact:f-p-pH}.
    Further,
    \[
      \Ex[Y_i^2] \le \Ex[Y_i] \cdot \max_{K_1 \in \mathcal{B}_i} \sum_{K_2 \in \mathcal{B}_i} p^{e_{K_2} - e_{K_1 \cap K_2}}.
    \]
    Fix a $K_1 \in \mathcal{B}_i$ and choose an embedding $\varphi$ of $J_i$ into $K_n$ with $\varphi(J_i \setminus J_0) = K_1$ that maps $J_0 \subseteq J_i$ into $G^*$.
    Note that, for every $U \subseteq V(J_i) \setminus V(J_0)$, there are at most $O(n^{v_{J_i} - v_{J_0} - |U|} \cdot f^{a(R_i)})$ graphs $K_2 \in \mathcal{B}_i$ such that $\varphi^{-1}(K_1 \cap K_2) \setminus V(J_0) = U$.
    Further, since $J_0$ is either empty or it achieves the maximum density $m(H)$ of $H$, and thus it is an induced subgraph of $J_i$, we have
    \[
      e_{K_1 \cap K_2} \le e_{J_i[V(J_0) \cup U]} - e_{J_i[V(J_0)]} \le m(H) \cdot (v_{J_0} + |U| - v_{J_0}) = m(H) \cdot |U|.
    \]
    It follows that, for some constant $C$ that depends only on $H$,
    \[
      \sum_{K_2 \in \mathcal{B}_i} p^{e_{K_2} - e_{K_1 \cap K_2}} \le C f^{a(R_i)} \cdot \max_{0 \le u \le v_{J_i} - v_{J_0}} n^{v_{J_i} - v_{J_0} - u} \cdot p^{e_{J_i} - e_{J_0} - m(H) \cdot u}.
    \]
    Since $np^{m(H)} \ge np_H^{m(H)} = 1$, the claimed upper bound on $\Ex[Y_i^2]$ now follows from \cref{fact:f-p-pH}.
  \end{proof}
  
  Finally, for each $i \in \br{d}$, let $\mathcal{E}_i \coloneqq \{ Y_i \ge t\}$ and let $\mathcal{E} \coloneqq \mathcal{E}^*\cap \mathcal{E}_{1} \cap \dotsb \cap \mathcal{E}_{d}$.
  
  \begin{claim}
    \label{claim:PrUT-given-E}
    We have $\Pr(\UT_{H,\delta} \mid \mathcal{E}) \geq \eps p^{e_H}$, provided that $\Gamma$ is sufficiently large.
  \end{claim}

  \begin{proof}
    By \cref{lem:lower}, it suffices to show that $\Ex[X \mid \mathcal{E}] \geq (1+\delta+\eps) \Ex[X]$.
    The key observation is that, on the event $\mathcal{E}$, the graph $G_1$ contains $T$ copies of $J$.
    To see this, note that the number of copies of $J_0 - R$ in $G^* \subseteq G_1$ such that the image of each $u \in V(J_0) \setminus R$ is embedded into $V_u$ is at least
    \[
      \prod_{u \in V(J_0) \setminus R} \lfloor (\gamma f)^{a_u} \rfloor \ge (\gamma f/2)^{a(V(J_0) \setminus R)} = (\gamma f/2)^{\alpha^*_{J_0 - R}}.
    \]
    Further, each tuple $(J_1 \setminus J_0, \dotsc, J_d \setminus J_0) \in \mathcal{B}_1[G_1] \times \dotsb \times \mathcal{B}_d[G_1]$ forms a copy of $J$ in $G_1$ together with each of the copies of $J_0 - R$ described above.
    Indeed, this is because $J_1 \cup \dotsb \cup J_d = J$ and $J_i \cap J_j = J_0$ for all distinct $i, j \in \br{d}$, see \cref{prop:optimal-pair-decomposition}.
    Finally, on the event $\mathcal{E}$, there are at least $t^d$ such tuples.

    Now, let $X'$ be the number of copies of $H$ in $G_1 \cup G_2$ that extend one of the copies of $J$ in $G_1$ and note that
    \[
      \Ex[X \mid \mathcal{E}] \ge \Ex[X' \mid \mathcal{E}] \ge T \cdot \binom{n-v_J}{v_H-v_J} \cdot p_2^{e_H-e_J} \ge \frac{(1-\eta)^{e_H}T}{v^{v_H}_H}\cdot \frac{n^{v_H}p^{e_H}}{n^{v_J}p^{e_J}} \ge (1+\delta+\eps) \cdot \Ex[X],
    \]
    where the last inequality follows from~\Cref{claim:Poisson-LB-t-T-large}, provided that $\Gamma$ is sufficiently large.
  \end{proof}

  Since \eqref{eq:evaluation-of-f} implies that $\varepsilon p^{e_H} \gg p^{\varepsilon f/2}$, in view of \Cref{claim:PrUT-given-E}, it suffices to establish the following.

  \begin{claim}
    \label{claim:PrE-LB}
    We have $\log \Pr(\mathcal{E}) \ge -(\eps/2) f \log(1/p)$.
  \end{claim}

  \begin{proof}
    Since $\mathcal{E}^*$ and $\mathcal{E}_1, \dotsc, \mathcal{E}_d$ are all increasing events, we have
    \[
      \log \Pr(\mathcal{E}) \ge  \log \Pr(\mathcal{E}^*) + \sum_{i=1}^{d}\log \Pr(\mathcal{E}_{i}) \ge - e_{G^*} \log(1/(\eta p)) + \sum_{i=1}^{d}\log \Pr(\mathcal{E}_{i}).
    \]	
    Since $e_{G^*} \le \gamma f$ and we may let $\gamma$ be arbitrarily small, it suffices to show that $\log \Pr(\cE_i)\ge -\eps/(4d) \cdot f\log (1/p)$ for all $i\in \br{d}$.
    
    Fix some $i \in \br{d}$, recall that $\Pr(\mathcal{E}_i) = \Pr(Y_i \ge t)$, and let $\mu_i \coloneqq \Ex[Y_i]$.
    Since $Y_i \le |\mathcal{B}_i| = \mu_i \cdot p_1^{e_{J_0}-e_{J_i}}$ with probability one, see~\cref{claim:Yi-moments}, we have 
    \[
      \mu_i \le \mu_i/2 + \mu_i \cdot p_1^{e_{J_0} - e_{J_i}}\cdot \Pr(Y_i \ge \mu_i/2).
    \]
    implying that
    \[
      \log \Pr(Y_i \ge \mu_i/2) \ge (e_{J_0} - e_{J_i}) \cdot \log(1/(\eta p)) - \log 2 \ge - \eps /(4d) \cdot f \log(1/p).
    \]
    Therefore, we may assume from now on that $t > \mu_i/2$.

    Set $t^*\coloneqq \max\{2\mu_i,t\}$ and note that $2\mu_i \le t^* \le 4t$.
    Further, $t \ll |\mathcal{B_i}|$ and $\mu_i \gg 1$, by \cref{claim:Poisson-LB-t-T-large,claim:Yi-moments}, respectively.
    We may thus apply \cref{lem:Poisson-lemma}, with $(\cH,\mu,t)$ replaced by $(\cB_i,\mu_i,t^*)$, to obtain

    \begin{equation}\label{eq:lower-bound-for-E_i}
      \begin{split}
        \log \Pr(\mathcal{E}_i)
        & \ge 4\log \Pr(\Pois(\mu_i) \geq t^*) - C_1t^* - 3\log(4\Ex[Y_i^2]/\mu_i) \\
        & \ge -4t^*\log (t^*/\mu_i) - 4C_1t - (1/c) \log (4f) \ge -16t\log(C_2t/\mu_i).
      \end{split}
    \end{equation}
    for some absolute constants $C_1, C_2$; the penultimate inequality follows from \cref{fact:lower-bound-for-Poisson} and \cref{claim:Yi-moments} and the final inequality holds as $t \ge f \gg 1$, by \cref{claim:Poisson-LB-t-T-large} and~\eqref{eq:evaluation-of-f}.
    We now consider two cases.

    \medskip
    \noindent
    \textit{Case 1. $x \le (C \log \log n)^{-1}$.}
    Since $\mu_i \gg 1$, by \cref{claim:Yi-moments} and \eqref{eq:evaluation-of-f}, we have, by \cref{claim:Poisson-LB-t-T-large},
    \[
      t \log (C_2 t / \mu_i) \leq t\log t = \Gamma xf \log n \cdot \log t \le D\Gamma x f \log n \log \log n \leq \frac{\Gamma D}{C}\cdot f \log (1/p),
    \]
    where $D$ is a constant that depends only on $H$.
    Thus, \eqref{eq:lower-bound-for-E_i} implies the desired inequality, provided that $C$ is sufficiently large.

    \medskip
    \noindent
    \textit{Case 2. $x \le C^{-1}$ and $H$ is clean.}
    Since $\dim_{J_0} J_i = 1$, the assumption that $H$ is clean yields
    \[
      1 + 1/(r\Delta C_H) = \alpha_{J_i}^* - \alpha_{J_0-R_i}^*
    \]
    and
    \[
      \alpha_{J_0}^* - |R_i| \le \alpha_{J_0 - R_i}^* \le \alpha_{J_0 - R}^* + |R \setminus R_i| = a(V(J_0) \setminus R) + |R \setminus R_i| = a(V(J_0)) - |R_i| \le \alpha_{J_0}^* - |R_i|.
    \]
    Consequently, by \cref{claim:Yi-moments},
    \[
      t = \Gamma f^{1+1/(r\Delta C_H)} = \Gamma f^{\alpha^*_{J_i}-\alpha^*_{J_0-R_i}} = \Gamma f^{\alpha^*_{J_i}-\alpha^*_{J_0}+|R_i|} \le \Gamma \mu_i / c.
    \]
    Therefore, for some constant $C_3$ that depends only on $\Gamma$, $\eta$, and $H$,
    \[
      t \log (C_2t /\mu_i) \leq \Gamma fx \log n \cdot \log(\Gamma C_2/c) \le C_3/C \cdot f \log n \le \eps/(4d) \cdot f\log(1/p),
    \]
    provided that $C$ is sufficiently large.
  \end{proof}
  The proof of the theorem is now complete.
\end{proof}

\bibliographystyle{abbrv}
\bibliography{UT-irregular-dense}

@article{vsileikis2019counterexample,
  title={A counterexample to the DeMarco-Kahn upper tail conjecture},
  author={{\v{S}}ileikis, Matas and Warnke, Lutz},
  journal={Random Structures \& Algorithms},
  volume={55},
  number={4},
  pages={775--794},
  year={2019},
  publisher={Wiley Online Library}
}

@article{barbour1989central,
  title={A central limit theorem for decomposable random variables with applications to random graphs},
  author={Barbour, Andrew D and Karo{\'n}ski, Michal and Ruci{\'n}ski, Andrzej},
  journal={Journal of Combinatorial Theory, Series B},
  volume={47},
  number={2},
  pages={125--145},
  year={1989},
  publisher={Elsevier}
}

@article {harris1960lower,
    AUTHOR = {Harris, T. E.},
     TITLE = {A lower bound for the critical probability in a certain
              percolation process},
   JOURNAL = {Proc. Cambridge Philos. Soc.},
  FJOURNAL = {Proceedings of the Cambridge Philosophical Society},
    VOLUME = {56},
      YEAR = {1960},
     PAGES = {13--20},
}

@article{janson1990poisson,
  title={Poisson approximation for large deviations},
  author={Janson, Svante},
  journal={Random Structures \& Algorithms},
  volume={1},
  number={2},
  pages={221--229},
  year={1990},
  publisher={Wiley Online Library}
}

@article{kozma2023lower,
  title={Lower tails via relative entropy},
  author={Kozma, Gady and Samotij, Wojciech},
  journal={The Annals of Probability},
  volume={51},
  number={2},
  pages={665--698},
  year={2023},
  publisher={Institute of Mathematical Statistics}
}

@article{janson2004deletion,
  title={The deletion method for upper tail estimates},
  author={Janson, Svante and Ruci{\'n}ski, Andrzej},
  journal={Combinatorica},
  volume={24},
  number={4},
  pages={615--640},
  year={2004},
  publisher={Springer-Verlag Berlin, Heidelberg}
}

@article{kim2004divide,
  title={Divide and conquer martingales and the number of triangles in a random graph},
  author={Kim, Jeong Han and Vu, Van H},
  journal={Random Structures \& Algorithms},
  volume={24},
  number={2},
  pages={166--174},
  year={2004},
  publisher={Wiley Online Library}
}

@article{chatterjee2012missing,
  title={The missing log in large deviations for triangle counts},
  author={Chatterjee, Sourav},
  journal={Random Structures \& Algorithms},
  volume={40},
  number={4},
  pages={437--451},
  year={2012},
  publisher={Wiley Online Library}
}

@article{demarco2012tight,
  title={Tight upper tail bounds for cliques},
  author={DeMarco, Robert and Kahn, Jeff},
  journal={Random Structures \& Algorithms},
  volume={41},
  number={4},
  pages={469--487},
  year={2012},
  publisher={Wiley Online Library}
}

@article{chatterjee2016nonlinear,
  title={Nonlinear large deviations},
  author={Chatterjee, Sourav and Dembo, Amir},
  journal={Advances in Mathematics},
  volume={299},
  pages={396--450},
  year={2016},
  publisher={Elsevier}
}

@article{eldan2018gaussian,
  title={Gaussian-width gradient complexity, reverse log-Sobolev inequalities and nonlinear large deviations},
  author={Eldan, Ronen},
  journal={Geometric and Functional Analysis},
  volume={28},
  number={6},
  pages={1548--1596},
  year={2018},
  publisher={Springer}
}

@article{augeri2021transportation,
  title={A transportation approach to the mean-field approximation},
  author={Augeri, Fanny},
  journal={Probability Theory and Related Fields},
  volume={180},
  number={1},
  pages={1--32},
  year={2021},
  publisher={Springer}
}

@article{augeri2020nonlinear,
  title={Nonlinear large deviation bounds with applications to Wigner matrices and sparse Erd{\H{o}}s--R{\'e}nyi graphs},
  author={Augeri, Fanny},
  journal={The Annals of probability},
  volume={48},
  number={5},
  pages={2404--2448},
  year={2020},
  publisher={JSTOR}
}

@article{cook2020large,
  title={Large deviations of subgraph counts for sparse Erd{\H{o}}s--R{\'e}nyi graphs},
  author={Cook, Nicholas and Dembo, Amir},
  journal={Advances in Mathematics},
  volume={373},
  pages={107289},
  year={2020},
  publisher={Elsevier}
}

@article{lubetzky2017variational,
  title={On the variational problem for upper tails in sparse random graphs},
  author={Lubetzky, Eyal and Zhao, Yufei},
  journal={Random Structures \& Algorithms},
  volume={50},
  number={3},
  pages={420--436},
  year={2017},
  publisher={Wiley Online Library}
}

@article{bhattacharya2017upper,
  title={Upper tails and independence polynomials in random graphs},
  author={Bhattacharya, Bhaswar B and Ganguly, Shirshendu and Lubetzky, Eyal and Zhao, Yufei},
  journal={Advances in Mathematics},
  volume={319},
  pages={313--347},
  year={2017},
  publisher={Elsevier}
}

@article {HarMouSam22,
    AUTHOR = {Harel, Matan and Mousset, Frank and Samotij, Wojciech},
     TITLE = {Upper tails via high moments and entropic stability},
   JOURNAL = {Duke Mathematical Journal},
    VOLUME = {171},
      YEAR = {2022},
    NUMBER = {10},
     PAGES = {2089--2192},
}

@article {FriKah98,
    AUTHOR = {Friedgut, Ehud and Kahn, Jeff},
     TITLE = {On the number of copies of one hypergraph in another},
   JOURNAL = {Israel Journal of Mathematics},
    VOLUME = {105},
      YEAR = {1998},
     PAGES = {251--256},
}

@article {Alo81,
    AUTHOR = {Alon, Noga},
     TITLE = {On the number of subgraphs of prescribed type of graphs with a
              given number of edges},
   JOURNAL = {Israel Journal of Mathematics},
    VOLUME = {38},
      YEAR = {1981},
    NUMBER = {1-2},
     PAGES = {116--130},
      ISSN = {0021-2172},
   MRCLASS = {05C35},
  MRNUMBER = {599482},
MRREVIEWER = {David E. Daykin},
       DOI = {10.1007/BF02761855},
}

@article {JanRuc2002,
    AUTHOR = {Janson, Svante and Ruci\'nski, Andrzej},
     TITLE = {The infamous upper tail},
      NOTE = {Probabilistic methods in combinatorial optimization},
   JOURNAL = {Random Structures \& Algorithms},
    VOLUME = {20},
      YEAR = {2002},
    NUMBER = {3},
     PAGES = {317--342},
      ISSN = {1042-9832,1098-2418},
   MRCLASS = {60C05 (90C27)},
  MRNUMBER = {1900611},
       DOI = {10.1002/rsa.10031},
       URL = {https://doi.org/10.1002/rsa.10031},
}

@article {BasBas2023,
    AUTHOR = {Basak, Anirban and Basu, Riddhipratim},
     TITLE = {Upper tail large deviations of regular subgraph counts in
              {E}rd{\H{o}}s-{R}\'enyi graphs in the full localized regime},
   JOURNAL = {Communications on Pure and Applied Mathematics},
    VOLUME = {76},
      YEAR = {2023},
    NUMBER = {1},
     PAGES = {3--72},
      ISSN = {0010-3640,1097-0312},
   MRCLASS = {60F10 (05C50 05C80)},
  MRNUMBER = {4544794},
       DOI = {10.1002/cpa.22036},
       URL = {https://doi.org/10.1002/cpa.22036},
}

@article {CooDemPha2024,
    AUTHOR = {Cook, Nicholas A. and Dembo, Amir and Pham, Huy Tuan},
     TITLE = {Regularity method and large deviation principles for the
              {E}rd{\H{o}}s-{R}\'enyi hypergraph},
   JOURNAL = {Duke Mathematical Journal},
    VOLUME = {173},
      YEAR = {2024},
    NUMBER = {5},
     PAGES = {873--946},
      ISSN = {0012-7094,1547-7398},
   MRCLASS = {05C65 (05C80 15A69 60F10)},
  MRNUMBER = {4740212},
       DOI = {10.1215/00127094-2023-0029},
       URL = {https://doi.org/10.1215/00127094-2023-0029},
}

@article{BasSha2025,
  title={Upper tail bounds for irregular graphs},
  author={Basak, Anirban and Karmakar, Shaibal},
  journal={arXiv preprint arXiv:2503.05311},
  year={2025}
}

@article {AkhSil2025,
    AUTHOR = {Akhmejanova, Margarita and {\v{S}}ileikis, Matas},
     TITLE = {On the upper tail of star counts in random graphs},
   JOURNAL = {Electronic Journal of Probability},
    VOLUME = {30},
      YEAR = {2025},
     PAGES = {Paper No. 82, 20},
      ISSN = {1083-6489},
   MRCLASS = {05C80 (60C05 60F10)},
  MRNUMBER = {4904105},
       DOI = {10.1214/25-ejp1345},
       URL = {https://doi.org/10.1214/25-ejp1345},
}

@article {Coh2024,
    AUTHOR = {Cohen Antonir, Asaf},
     TITLE = {The upper tail problem for induced 4-cycles in sparse random
              graphs},
   JOURNAL = {Random Structures \& Algorithms},
    VOLUME = {64},
      YEAR = {2024},
    NUMBER = {2},
     PAGES = {401--459},
      ISSN = {1042-9832,1098-2418},
   MRCLASS = {05C80 (60F10)},
  MRNUMBER = {4704274},
MRREVIEWER = {Yilun\ Shang},
       DOI = {10.1002/rsa.21187},
       URL = {https://doi.org/10.1002/rsa.21187},
}

@unpublished{MicNieSer24,
  author       = {Wojciech Michalczuk AND Mikołaj Nieradko AND
                  Grzegorz Serafin},
  title	       = {{Normal approximation for subgraph count in random
                  hypergraphs}},
  note	       = {arXiv:2408.06112}
}

@article {DeMKah12-K3,
    AUTHOR = {DeMarco, B. and Kahn, J.},
     TITLE = {Upper tails for triangles},
   JOURNAL = {Random Structures \& Algorithms},
    VOLUME = {40},
      YEAR = {2012},
    NUMBER = {4},
     PAGES = {452--459},
      ISSN = {1042-9832,1098-2418},
   MRCLASS = {60F10 (05C80)},
  MRNUMBER = {2925307},
       DOI = {10.1002/rsa.20382},
}

@article {Vu02,
    AUTHOR = {Vu, V. H.},
     TITLE = {Concentration of non-{L}ipschitz functions and applications},
   JOURNAL = {Random Structures \& Algorithms},
    VOLUME = {20},
      YEAR = {2002},
    NUMBER = {3},
     PAGES = {262--316},
      ISSN = {1042-9832,1098-2418},
   MRCLASS = {60F10 (90C27)},
  MRNUMBER = {1900610},
       DOI = {10.1002/rsa.10032},
}

@article {Ros11,
    AUTHOR = {Ross, Nathan},
     TITLE = {Fundamentals of {S}tein's method},
   JOURNAL = {Probability Surveys},
    VOLUME = {8},
      YEAR = {2011},
     PAGES = {210--293},
      ISSN = {1549-5787},
   MRCLASS = {60F05 (05C80 60C05)},
  MRNUMBER = {2861132},
MRREVIEWER = {Anant\ P.\ Godbole},
       DOI = {10.1214/11-PS182},
}

@article {KupSam24,
    AUTHOR = {Kuperwasser, Eden and Samotij, Wojciech},
     TITLE = {The list-{R}amsey threshold for families of graphs},
   JOURNAL = {Combinatorics, Probability and Computing},
    VOLUME = {33},
      YEAR = {2024},
    NUMBER = {6},
     PAGES = {829--851},
      ISSN = {0963-5483,1469-2163},
   MRCLASS = {05C80 (05D10)},
  MRNUMBER = {4822417},
MRREVIEWER = {Dingding\ Dong},
       DOI = {10.1017/s0963548324000245},
}

@article {JanOleRuc04,
    AUTHOR = {Janson, Svante and Oleszkiewicz, Krzysztof and Ruci\'{n}ski,
              Andrzej},
     TITLE = {Upper tails for subgraph counts in random graphs},
   JOURNAL = {Israel Journal of Mathematics},
    VOLUME = {142},
      YEAR = {2004},
     PAGES = {61--92},
      ISSN = {0021-2172,1565-8511},
   MRCLASS = {05C80},
  MRNUMBER = {2085711},
MRREVIEWER = {David\ B.\ Penman},
       DOI = {10.1007/BF02771528},
}

@Unpublished{Gal,
  author       = {David Galvin},
  title	       = {{Three tutorial lectures on entropy and counting}},
  note	       = {arXiv:1406.7872}
}

\appendix

\section{Proof of \cref{lem:Poisson-lemma}}
\label{appendix:Poisson-lemma}

In this section, we prove \cref{lem:Poisson-lemma}, which we restate here for the convenience of the reader.

\Poissonlemma*

\begin{proof}
  For the sake of brevity, write $M \coloneqq |\mathcal{H}|$.
  By the pigeonhole principle and Markov's inequality, there is some integer $t^* \in [0,2\mu)$ such that
  \begin{equation}
    \label{eq:definition-of-t*}
    \Pr(Y = t^*) \geq \frac{\Pr(0 \le Y < 2\mu)}{\lceil 2 \mu \rceil} \ge \frac{1 - \Ex[Y]/2\mu}{2\mu+1} \ge \frac{1}{6\mu}.
  \end{equation}

  Our goal is to estimate the `discrete derivative' of $s \mapsto \log \Pr(Y = s)$ and compare it to that of $s \mapsto \log \Pr(Z = s)$, where $Z$ denotes a $\Pois(\mu)$ random variable.
  \begin{claim}
    \label{claim:switching-equation-weighted}
    For every $s \leq M/2$ with $\Pr(Y=s) > 0$, there exists positive integer $d$ satisfying $s + d \le 4\Ex[Y^2] / (\mu \cdot \Pr(Y=s))$ such that
    \[
      \frac{\Pr(Y=s+d)}{\Pr(Y=s)} \geq \frac{\mu}{s+d} \cdot \frac{1}{7d^2}.
    \]
  \end{claim}

  \begin{proof}
    Fix an $s \leq M/2$ and examine the sum
    \[
      \Sigma \coloneqq \sum_{\substack{R \subseteq V \\ Y(R) = s}} \sum_{\substack{A \in \mathcal{H} \\ A \nsubseteq R}} \Pr(V_p = R \cup A) \cdot \left(\frac{1-p}{p}\right)^{|A \setminus R|}.
    \]
    On the one hand, we have
    \[
      \Sigma = \sum_{\substack{R \subseteq V \\ Y(R) = s}} \Pr(V_p = R) \cdot |\{A \in \mathcal{H} : A \nsubseteq R\}|
      \ge \Pr(Y = s) \cdot (M-s).
    \]
    Since $M - s \ge M/2 = \mu / (2p^r)$, we may conclude that
    \begin{equation}
      \label{eq:Sigma-lower}
      \Sigma \ge \Pr(Y=s) \cdot \frac{\mu}{2p^r}.
    \end{equation}
    On the other hand, as $Y(R \cup A) \ge Y(R) + 1$ for all $R \subseteq V$ and $A \in \mathcal{H}$ with $A \nsubseteq R$, we have
    \[
      \begin{split}
        \Sigma
        & \le \sum_{\substack{R' \subseteq V \\ Y(R') \ge s+1}} \sum_{\substack{A \in \mathcal{H} \\ A \subseteq R'}} \Pr(V_p = R') \sum_{R' \setminus A \subseteq R \subsetneq R'} \left(\frac{1-p}{p}\right)^{|A \setminus R|}\\
        & = \sum_{\substack{R' \subseteq V \\ Y(R') \ge s+1}} \sum_{\substack{A \in \mathcal{H} \\ A \subseteq R'}} \Pr(V_p = R') \cdot (p^{-|A|} - 1) \\
        & \le \sum_{d \ge 1} (s + d) \cdot \Pr(Y = s+d) \cdot p^{-r}.
      \end{split}
    \]
    Further, for every $K \ge 0$, we have
    \[
      \sum_{d \ge K} (s+d) \cdot \Pr(Y=s+d) \le \frac{1}{s+K} \cdot \sum_{d \ge K} (s+d)^2 \cdot \Pr(Y = s+d) \le \frac{\Ex[Y^2]}{s+K}
    \]
    and thus
    \begin{equation}
      \label{eq:Sigma-upper}
      \Sigma \cdot p^r \le \sum_{d=1}^{\lfloor K \rfloor} (s+d) \cdot \Pr(Y = s+d) + \frac{\Ex[Y^2]}{s+K}.
    \end{equation}
    Combining \eqref{eq:Sigma-lower} and~\eqref{eq:Sigma-upper} with $s + K \coloneqq \lfloor 4\Ex[Y^2] / (\mu \cdot \Pr(Y=s)) \rfloor$, we get
    \[
      \Pr(Y=s) \cdot \frac{\mu}{7} \cdot \sum_{d \ge 1} \frac{1}{d^2} \le \Pr(Y=s) \cdot \frac{\mu}{4} \le \sum_{d=1}^K (s+d) \cdot \Pr(Y = s+d),
    \]
    which clearly implies the existence of the desired integer $d$.
  \end{proof}
  Since $t^* < t \le M/2$, we may use~\cref{claim:switching-equation-weighted} to find a sequence $t^* \eqqcolon t_0 < t_1 < \dotsb < t_\ell$ with $t_{\ell -1} < t \le t_\ell$ such that, for every $i \in \br{\ell}$,
  \begin{equation}
    \label{eq:Pr-ratio-sequence}
    \frac{\Pr(Y = t_i)}{\Pr(Y = t_{i-1})} \ge \frac{\mu}{t_i} \cdot \frac{1}{7(t_i-t_{i-1})^2}
  \end{equation}
  and $t_\ell \le 4\Ex[Y^2] / (\mu \cdot \Pr(Y = t_{\ell-1}))$.
  
  Let $t' \coloneqq t_{\ell-1}$.
  Multiplying \eqref{eq:Pr-ratio-sequence} for all $i \in \br{\ell-1}$ and using \eqref{eq:definition-of-t*} as well as the inequality $7d^2 \le 7^d$, valid for every positive integer $d$, yields
  \[
    \Pr(Y = t') \ge \Pr(Y = t^*) \cdot \prod_{i=1}^{\ell-1} \left( \frac{\mu}{t_i} \cdot\frac{1}{7(t_i-t_{i-1})^2} \right) \ge \frac{1}{6\mu \cdot 7^t} \cdot \prod_{i=1}^{\ell-1} \frac{\mu}{t_i}.    
  \]
  On the other hand,
  \[
    \Pr(Z = t) = \frac{\mu^t}{t!} \cdot e^{-\mu} \le \prod_{i=1}^{\ell-1} \frac{\mu}{t_i} \cdot \frac{\mu^{t-\ell+1}}{(t - \ell+1)!} \cdot e^{-\mu} \le \prod_{i=1}^{\ell-1} \frac{\mu}{t_i},
  \]
  while a routine calculation shows that $\Pr(Z \ge t) \le 2 \Pr(Z = t)$, as $t \ge 2\mu$.
  Combining these estimates yields
  \begin{equation}
    \label{eq:Y-t'-lower}
    \Pr(Y = t') \ge C^{-t} \cdot \Pr(Z \ge t)
  \end{equation}
  for some absolute constant $C$.
  
  Finally, it follows from~\eqref{eq:Pr-ratio-sequence}, and the assumption that $\mu \ge 1$, that
  \[
    \frac{\Pr(Y \ge t)}{\Pr(Y = t')} \ge \frac{\Pr(Y=t_\ell)}{\Pr(Y=t')} \ge \frac{1}{7t_\ell^3} \ge \frac{1}{7} \cdot \max\left\{\frac{1}{M}, \frac{\mu \cdot \Pr(Y=t')}{4\Ex[Y^2]}\right\}^3.
  \]
  Substituting \eqref{eq:Y-t'-lower} into the above estimate yields the desired bounds.
\end{proof}

\end{document}